\documentclass[3p]{elsarticle}
\biboptions{sort&compress}
\usepackage{url}
\usepackage{graphicx}
\usepackage{amsmath}
\usepackage{amsthm}
\usepackage{algorithm}
\usepackage{algorithmicx}
\usepackage{algpseudocode}
\usepackage{xcolor}
\usepackage{mathtools}
\usepackage{cleveref}
\usepackage[normalem]{ulem}

\usepackage[bbgreekl]{mathbbol}

\newcommand{\DIV}[1]{\nabla_{\mathbf{X}} \cdot #1}

\newcommand{\etal}{\emph{et al}.}

\newcommand{\FF}{\mathbb{F}}

\newcommand{\PP}{\mathbb{P}}

\newcommand{\FFb}{\overline{\FF}}

\newcommand{\ub}{\boldsymbol{u}}

\newcommand{\Chib}{\boldsymbol{\chi}}
\newcommand{\cb}{\boldsymbol{\chi}}
\newcommand{\xb}{\boldsymbol{x}}
\newcommand{\Xb}{\boldsymbol{X}}
\newcommand{\fb}{\boldsymbol{f}}
\newcommand{\Fb}{\boldsymbol{F}}
\newcommand{\Tb}{\boldsymbol{T}}
\newcommand{\Ub}{\boldsymbol{U}}
\newcommand{\UbVecIB}{\vec{\Ub}\vphantom{\Ub}^{\text{IB}}}

\newcommand{\Lb}{\boldsymbol{L}}
\newcommand{\LbVecIB}{\vec{\Lb}\vphantom{\Lb}^{\text{IB}}}
\newcommand{\Nb}{\boldsymbol{N}}
\newcommand{\nb}{\boldsymbol{n}}
\newcommand{\hb}{\boldsymbol{h}}

\newcommand{\Ab}{\boldsymbol{A}}
\newcommand{\phib}{\boldsymbol{\phi}}
\newcommand{\psib}{\boldsymbol{\psi}}

\newcommand{\cauchy}{\bbsigma}
\newcommand{\cauchyf}{\bbsigma^{\text{f}}}
\newcommand{\cauchyv}{\bbsigma^{\text{v}}}
\newcommand{\cauchys}{\bbsigma^{\text{e}}}

\newcommand{\ztensor}{\mathbb{0}}
\newcommand{\pstab}{\pi_{\text{stab}}}

\newcommand{\rhos}{\rho^{\text{s}}}
\newcommand{\rhof}{\rho^{\text{f}}}

\newcommand{\PPs}{\PP^{\text{e}}}

\newcommand{\soliddom}{{\Omega^{\text{s}}_t}}
\newcommand{\soliddomO}{{\Omega^{\text{s}}_0}}
\newcommand{\fluiddom}{\Omega^{\text{f}}_t}
\newcommand{\fluiddomO}{\Omega^{\text{f}}_0}

\newcommand{\efac}{E_\text{FAC}}
\newcommand{\mfac}{M_\text{FAC}}

\newcommand{\kappas}{\kappa_\text{stab}}

\newcommand{\supp}{\text{supp}}

\newcommand{\Pone}{\mathcal{P}^1}
\newcommand{\Ptwo}{\mathcal{P}^2}
\newcommand{\Qone}{\mathcal{Q}^1}
\newcommand{\Qtwo}{\mathcal{Q}^2}

\newcommand{\tria}{{\ensuremath{\mathcal{T}^h}}}
\newcommand{\euleriandx}{{\ensuremath{\Delta x}}}
\newcommand{\lagrangiandx}{{\ensuremath{\Delta X}}}
\newcommand{\half}{{\ensuremath{\frac{1}{2}}}}

\newcommand{\dAb}{\, \mathrm{d}\Ab}
\newcommand{\dXb}{\, \mathrm{d}\Xb}

\newcommand{\dxb}{\, \mathrm{d}\xb}

\newcommand{\xface}{{i + \half, j, k}}
\newcommand{\yface}{{i, j + \half, k}}
\newcommand{\zface}{{i, j, k + \half}}

\newcommand{\vecT}[1]{\ensuremath{\vec{#1}\vphantom{#1}^T}}
\newcommand{\vecOneT}{%
  \ensuremath{%
    \vec{\vphantom{\boldsymbol{1}}\hphantom{-}}% Put an arrow over a box of about the right shape
    \llap{\ensuremath{\boldsymbol{1}}}% place the symbol to the left of the current location of the cursor
    ^T%
  }}

\newcommand{\fespace}{\ensuremath{V^h}}
\newcommand{\fesuperspace}{\ensuremath{V}}

\newtheorem{theorem}{Theorem}
\newtheorem{remark}{Remark}
\DeclareMathOperator*{\argmin}{arg\,min}

\definecolor{vargreen}{rgb}{0.0, 0.5, 0.0}
\definecolor{varpurp}{rgb}{0.5, 0.0, 0.5}

\begin{document}

\begin{frontmatter}

\title{A Nodal Immersed Finite Element-Finite Difference Method}

\author[1]{David R. Wells\corref{cor1}\fnref{fn:authorship}}
\ead{drwells@email.unc.edu}

\author[2]{Ben Vadala-Roth\fnref{fn:authorship,fn:independence}}

\author[3]{Jae H. Lee\fnref{fn:affiliation}}

\author[4,5,6,7]{Boyce E.~Griffith\corref{cor1}}
\ead{boyceg@email.unc.edu}

\address[1]{Department of Mathematics, University of North Carolina, Chapel Hill, NC, USA}
\address[2]{Westborough, MA, USA}
\address[3]{Department of Mechanical Engineering and Institute for Computational Medicine, Johns Hopkins University, Baltimore, MD, USA}
\address[4]{Departments of Mathematics, Applied Physical Sciences, and Biomedical Engineering, University of North Carolina, Chapel Hill, NC, USA}
\address[5]{Carolina Center for Interdisciplinary Applied Mathematics, University of North Carolina, Chapel Hill, NC, USA}
\address[6]{Computational Medicine Program, University of North Carolina, Chapel Hill, NC, USA}
\address[7]{McAllister Heart Institute, University of North Carolina, Chapel Hill, NC, USA}

\cortext[cor1]{Corresponding authors}
\fntext[fn:authorship]{These authors made equal contributions to this manuscript.}
\fntext[fn:independence]{Independent Researcher}
\fntext[fn:affiliation]{Present address: Center for Drug Evaluation and Research, U.S. Food and Drug Administration, Silver Spring, MD, USA}

\begin{abstract}
% lineno needs help with things in boxes
%%\begin{linenumbers}
  The immersed finite element-finite difference (IFED) method is a computational approach to modeling interactions between a fluid and an immersed structure.
  The IFED method uses a finite element (FE) method to approximate the stresses, forces, and structural deformations on a \emph{structural mesh} and a finite difference (FD) method to approximate the momentum and enforce the incompressibility of the entire fluid-structure system on a \emph{Cartesian grid}.
  The fundamental approach used by this method follows the immersed boundary framework for modeling fluid-structure interaction (FSI), in which a force spreading operator prolongs structural forces to a Cartesian grid, and a velocity interpolation operator restricts a velocity field defined on that grid back onto the structural mesh.
  With an FE structural mechanics framework, force spreading first requires that the force itself be projected onto the finite element space.
  Similarly, velocity interpolation requires projecting velocity data onto the FE basis functions.
  Consequently, evaluating either coupling operator requires solving a matrix equation at every time step.
  Mass lumping, in which the projection matrices are replaced by diagonal approximations, has the potential to accelerate this method considerably.
  This paper provides both numerical and computational analyses of the effects of this replacement for evaluating the force projection and for the IFED coupling operators.
  Constructing the coupling operators also requires determining the locations on the structure mesh where the forces and velocities are sampled.
  Here we show that sampling the forces and velocities at the nodes of the structural mesh is equivalent to using lumped mass matrices in the IFED coupling operators.
  A key theoretical result of our analysis is that if both of these approaches are used together, the IFED method permits the use of lumped mass matrices derived from nodal quadrature rules for any standard interpolatory element.
  This is different from standard FE methods, which require specialized treatments to accommodate mass lumping with higher-order shape functions. Our theoretical results are confirmed by numerical benchmarks, including standard solid mechanics tests and examination of a dynamic model of a bioprosthetic heart valve.
%%\end{linenumbers}
\end{abstract}

\begin{keyword}
  Immersed boundary method \sep fluid-structure interaction \sep finite elements \sep finite differences \sep mass lumping \sep nodal quadrature
\end{keyword}

\end{frontmatter}

\section{Introduction}
\label{sec:intro}
The immersed boundary (IB) method was introduced by Peskin to model fluid-structure interaction (FSI) in heart valves~\cite{Peskin1972,Peskin1977}.
This method describes thin elastic structures immersed in a Newtonian fluid with Lagrangian variables for the forces and resultant deformations of the structure and Eulerian variables for the momentum of the coupled fluid-structure system.
In its original implementations, the equations are discretized via finite differences, and interactions between Lagrangian and Eulerian variables are handled through discretized integral equations with regularized Dirac delta function kernels.
Specifically, the force defined on the structural mesh (where Lagrangian quantities are evaluated) is spread onto the Cartesian grid (where Eulerian quantities are evaluated), and the velocity defined on the Cartesian grid is interpolated back to the structural mesh.
The IB method has been extended to treat volumetric structures~\cite{Peskin2002} and has been used in a wide range of applications.
Griffith and Patankar~\cite{Griffith2020} discuss many applications, including swimmers, esophageal transport, and heart valve dynamics.
We focus on the mathematical formulation of Boffi \etal~\cite{Boffi2008}, which is systematically derived from the theory of large-deformation continuum mechanics.
One advantage of this formulation is that it can leverage the broad range of structural constitutive models that have been developed, including many with parameters that can be determined directly from experimental data.
When the IB method is combined with such models, it can achieve excellent agreement between physical experiments and numerical simulations; see, e.g., the work of Lee \etal~\cite{Lee2020}.

Several finite element (FE)-based extensions of the IB method have been created, including the works of Boffi \etal~\cite{Boffi2008}, Zhang \etal~\cite{Zhang2004}, Devendran and Peskin~\cite{Devendran2012}, and Griffith and Luo~\cite{Griffith2017}.
Some of these approaches~\cite{Boffi2008,Zhang2004} discretize the entire IB system of equations with finite elements, whereas the works of Devendran and Peskin~\cite{Devendran2012} and Griffith and Luo~\cite{Griffith2017} use finite differences to discretize the Eulerian equations and finite elements to discretize the Lagrangian equations.
We focus on modifying the method of Griffith and Luo~\cite{Griffith2017}, referring to this method as the immersed finite element-finite difference (IFED) method.
Aside from using both FE and finite difference methods, the IFED method is also notable for introducing the concept of \emph{interaction points}.
These points are used in the regularized delta function-based discrete coupling operators that link the Eulerian and Lagrangian representations and can be chosen to be distinct from the points used to discretize the structure, referred to as \emph{control points}.
In many FE-based IB methods, the control points are chosen as the interaction points; herein, we refer to this approach as \emph{nodal interaction}.
This is distinct from using interaction points chosen from the interior of the elements, which we term \emph{elemental interaction}.

Even in the continuous formulation of the IB equations, the Eulerian form of elastic force belongs to the class of distributions and as such is generally singular at the fluid-solid interface.
However, upon the introduction of a regularized delta function, as is common in IB methods, this discontinuity is smoothed out across the interface.
Similarly, for FE-based IB methods, the solid is often discretized with $H^1$-conforming finite elements, and the solid stress will in general be discontinuous along the sides of elements edges, even away from the fluid-solid interface.
Force transmission from the structural mesh to the Cartesian grid offers multiple possibilities: the method of Devendran and Peskin~\cite{Devendran2012} spreads, and thereby regularizes, the divergence of the structural stress; the IFED method of Griffith and Luo~\cite{Griffith2017} introduces a weak notion of force and projects it onto the set of FE basis functions, further regularizing the force, before it is spread to the Cartesian grid.
In fact, as we detail herein, these approaches are equivalent in certain practical cases.

A consistent discretization of a time-dependent problem with the finite element method contains a mass matrix multiplied by the vector of time derivatives of the approximation.
Such discretizations require, even with explicit time stepping, solving a linear system at every time step.
The primary contribution of this paper is an extension of the IFED method introduced by Griffith and Luo~\cite{Griffith2017} that only uses diagonal mass matrices, thereby avoiding the need to solve linear systems of equations arising from finite element discretizations.
The technique we utilize to avoid solving nontrivial (i.e., nondiagonal) matrix equations is \emph{mass lumping}, in which the mass matrix (i.e., the matrix corresponding to an $L^2$ projection onto the finite element space) is replaced by an appropriate diagonal approximation.
For linear finite elements, this can be accomplished by evaluating the inner products with a lower-order quadrature, which is usually chosen to be the one based on interpolation with the finite element space itself.
This approach is usually called \emph{nodal quadrature}.
FE practitioners have long used mass lumping to increase performance and, in some cases, do so without loss of accuracy; see the work of Fried and Malkus~\cite{Fried1975}.
Mass lumping usually requires that the finite element space consists of basis functions with positive mean values, which is, in general, only true for linear elements.
There is a large body of literature available on different methods to avoid this difficulty~\cite{Fried1975, Geevers_2018, cohen2001higher, guermond2013correction, hansbo1994aspects}.
For a recent summary of work on mass lumping, including some new higher-order tetrahedral elements, we refer the reader to the work of Geevers \etal~\cite{Geevers_2018}.
Other notable lumping techniques include row-summing and formulae that scale the diagonal entries to maintain elemental mass~\cite{Hinton1976, ThomasJRHughes2000}; these techniques are equivalent for certain element types.
Nodal quadrature and row sums, when applied to standard Lagrange-type finite elements, may result in non-positive entries, which are undesirable or even unusable in practice as they typically correspond to unstable modes in time evolution~\cite{cohen2001higher}.
A remarkable feature of our new nodal IFED scheme is that it does not require the use of any of these complex approaches for enabling higher-order structural elements. In particular, with the nodal IFED approach introduced herein, simple lumping approaches are effective for both low-~and high-order elements.

\begin{figure}
  \centering
  \begin{tabular}{c c}
    \includegraphics[width=0.4\linewidth]{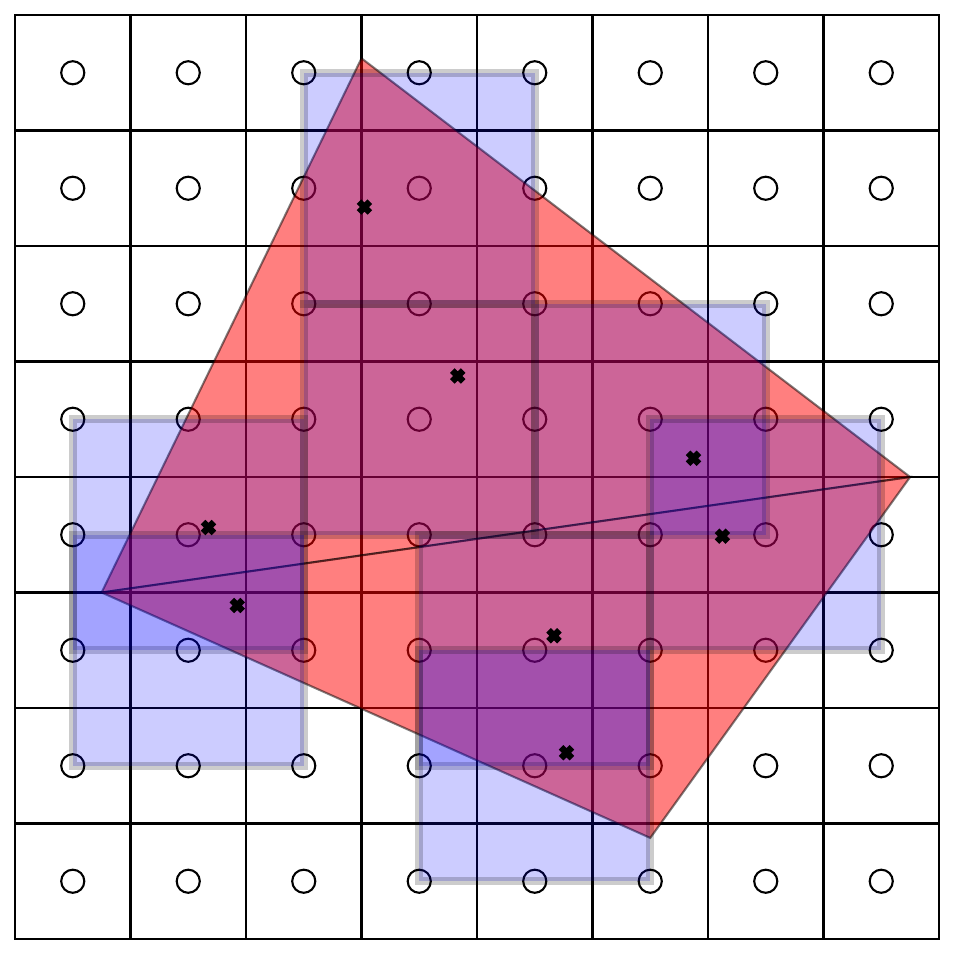} &
    \includegraphics[width=0.4\linewidth]{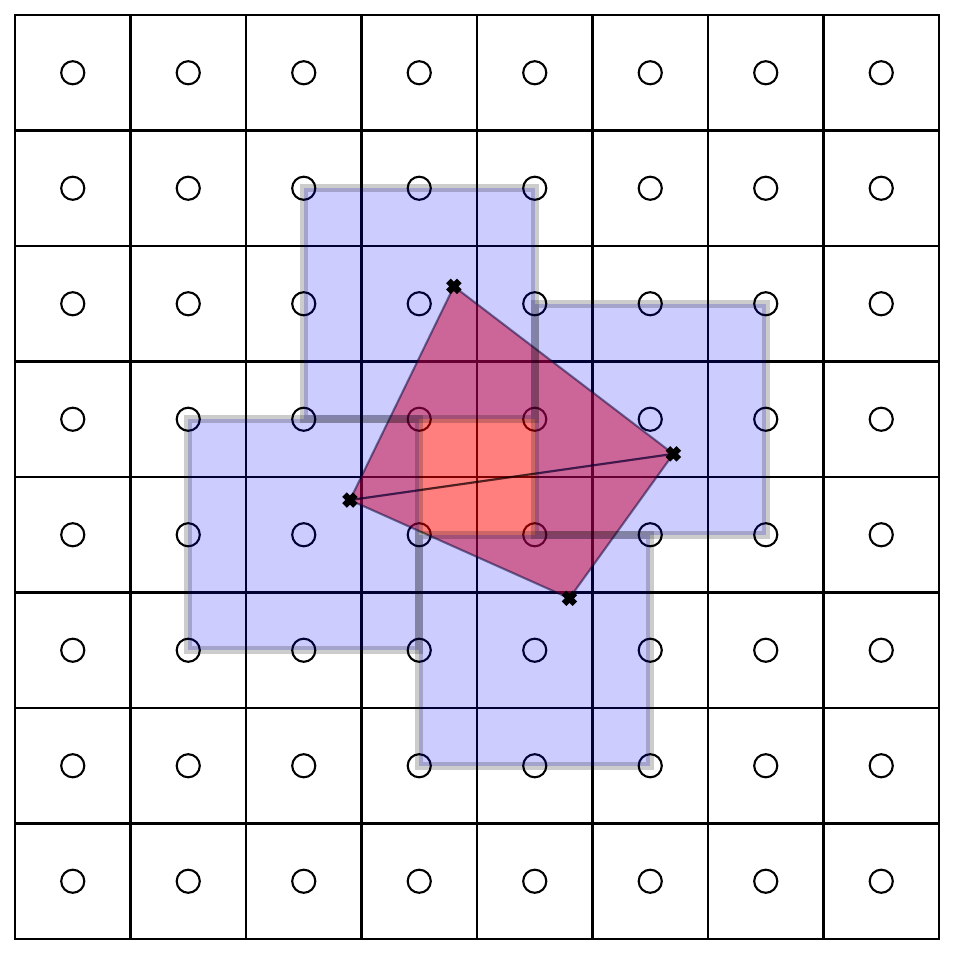}
  \end{tabular}

  \caption{Schematic of elemental (left) and nodal (right) coupling.
    Each coupling is between a Cartesian grid and a finite element mesh with two elements and uses a regularized delta function with a $3 \times 3$ stencil, depicted as a blue box around each interaction point.
    We use a Gauss quadrature rule for elemental coupling and the mesh nodes for nodal coupling, which provide eight and four interaction points, respectively.
    The triangles used for elemental coupling are larger than those that would typically be used in practice to make the visualization clear.
    %(For interpretation of the colors in the figure(s), the reader is referred to the web version of this article.)
  }
  \label{fig:kernel-schematic}
\end{figure}

In the method of Griffith and Luo~\cite{Griffith2017}, mass matrices appear in two places: the projection of the force onto FE basis functions, which appears in the weak form of the divergence of the stress, and the projection that computes Lagrangian representations of the Eulerian velocity.
We show that using nodal interaction is equivalent to lumping the mass matrix associated with velocity interpolation.
A major theoretical result of this work, which appears in Theorem~\ref{thm:fully-nodal-ignore-weights}, is that the positive mean restriction can be avoided so long as the same diagonal mass matrices are used for all projections.
This explains why it is straightforward to use simple mass lumping techniques with the IFED formulation, even for high-order elements, and yields a very substantial simplification in the development and implementation of diagonal mass matrices.
In this study, we refer to the combination of nodal interaction, which implies a lumped mass matrix for velocity projection, and a lumped force projection matrix as \emph{nodal coupling}.
Similarly, we call the combination of consistent mass matrices and elemental interaction \emph{elemental coupling} (see Figure~\ref{fig:kernel-schematic} for the schematic of elemental (left) and nodal (right) coupling).

We demonstrate the efficacy of the IFED method with nodal coupling using a series of static and dynamic numerical benchmarks and two realistic three-dimensional FSI examples.
We begin by examining the impact of different choices of relative mesh spacings for the Lagrangian and Eulerian discretizations with the nodal IFED approach.
We also use Cook's membrane~\cite{RDCook1974}, a compressed block benchmark~\cite{Reese1999}, a torsion benchmark~\cite{Bonet2015}, and a modified Turek-Hron benchmark~\cite{Turek2007,Lee2021}.
The first three-dimensional FSI case that we consider is a FSI benchmark originally introduced to test a nonconforming arbitrary Lagrangian-Eulerian (ALE) method~\cite{Hessenthaler2017a,Hessenthaler2017b}.
The final example is a model of a bioprosthetic heart valve as introduced by Lee \etal~\cite{Lee2020, LeeJTCVS}.
All benchmarks confirm that we get comparable results with elemental and nodal coupling as long as the structural mesh discretization is not too coarse with respect to the background Cartesian grid and, further, that high quality results are obtained at practical Eulerian and Lagrangian grid spacings.

%% \DRW{If we move around the benchmarks we will have to modify these descriptions.}

\section{IFED Formulation}
\subsection{Continuous Equations of Motion}
\label{subsec:continuous-eom}
We specify the equations of motion for the case that the coupled fluid-structure system occupies a fixed computational domain $\Omega = \fluiddom \cup \soliddom \subset \mathbb{R}^3$, in which $\fluiddom$ and $\soliddom$ are the subregions occupied by the fluid and the structure, respectively, at time $t$.
The IFED method uses both Eulerian descriptions of motion, which use fixed physical coordinates $\xb \in \Omega$, and Lagrangian descriptions of motion, which use reference coordinates $\Xb \in \soliddomO$.
The deformation mapping $\Chib (\Xb,t): (\soliddomO, t) \mapsto \soliddom$ connects the reference configuration of the structure to its current configuration.
The dynamics of coupled fluid-structure system are described by
\begin{align}
    \rho \frac{D\ub}{Dt}(\xb, t) &= -\nabla p(\xb, t)
    + \mu \nabla ^2 \ub(\xb, t) + \fb(\xb, t),
    \label{eq:ns} \\
    \nabla \cdot \ub(\xb, t) &= \, 0,
    \label{eq:divfree} \\
    % no PPs yet
    \fb(\xb, t) &= \int_{\soliddomO} \Fb(\Xb,t) \, \delta(\xb - \Chib (\Xb,t)) \dXb,
    \text{ and }
    \label{eq:general-structure-force} \\
    \frac{\partial \Chib}{\partial t}(\Xb, t) &= \ \Ub(\Xb,t) = \int_{\Omega}
    \ub(\xb, t) \, \delta(\xb - \Chib (\Xb,t)) \dxb,
    \label{eq:noslip}
\end{align}
in which $\ub(\xb,t)$ is the Eulerian velocity, $\Ub(\Xb,t)$ is the structure's velocity, $\rho$ is the uniform mass density of both the fluid and the structure, $\mu$ is the uniform dynamic viscosity, $\Fb(\Xb,t)$ is the Lagrangian force density, $\fb(\xb, t)$ is the Eulerian form of $\Fb$, and $\Nb(\Xb)$ is the outward unit normal along $\partial \soliddomO$, the boundary of the structure, in the reference configuration.
The physical pressure $p(\xb, t)$ is responsible for maintaining the incompressibility constraint (Equation~\eqref{eq:divfree}).
The operators $\nabla^2, \nabla\cdot\mbox{}$, and $\nabla$ are with respect to Eulerian coordinates, and $\frac{D}{Dt} = \frac{\partial}{\partial t} + \ub \cdot \nabla$ is the (Eulerian) material time derivative.
For rigid structures, $\Fb$ may be computed in a variety of ways, such as a penalty force that approximately enforces zero displacement~\cite{Lee2021}.
For flexible structures, our formulation defines Cauchy stress on the computational domain to be
\begin{equation}
    \cauchy(\xb,t) = \cauchyf(\xb, t) +
    \begin{cases}
        \ztensor        & \xb \in \fluiddom, \\
        \cauchys(\xb,t) & \xb \in \soliddom.
    \end{cases}
    \label{cauchy-def}
\end{equation}
The first Piola-Kirchhoff stress $\PPs$ is a convenient way to describe the elastic response of the structure.
For the hyperelastic constitutive models considered here, we determine $\PPs$ from a strain energy functional $\Psi(\FF)$ via
\begin{equation}
  \PPs = \frac{\partial \Psi(\FF)}{\partial \FF},
\end{equation}
in which $\FF = \frac{\partial \Chib}{\partial \Xb}$ is the deformation gradient tensor.
The first Piola-Kirchhoff stress is related to the corresponding Cauchy stress by
\begin{equation}
  \cauchys = \frac{1}{J}\PPs \FF^T,
\end{equation}
in which $J = \det(\FF)$.
To match Equation~\eqref{eq:ns}, we consider a Newtonian fluid with Cauchy stress given by
\begin{equation}
  \cauchyf = -p \mathbb{I} + \mu\left(\nabla \ub + \nabla \ub^T \right).
\end{equation}
The resulting IB form of Equation~\eqref{eq:general-structure-force} for elastic structures, as derived by Boffi \etal~\cite{Boffi2008}, is:
\begin{equation}
    \label{eq:elastic-force}
    \fb(\xb, t) = \int_{\soliddomO} \nabla_{\Xb} \cdot \PPs (\Xb,t)
    \, \delta(\xb - \Chib (\Xb,t)) \dXb
    - \int_{\partial \soliddomO} \PPs (\Xb,t) \Nb(\Xb)
    \, \delta(\xb - \Chib (\Xb,t)) \dAb.
\end{equation}
The differential operator $\nabla_{\Xb}\cdot\mbox{}$ is the divergence operator in Lagrangian coordinates.

\subsection{Weak Structural Formulation}
Addressing the material coordinate derivative in Equation~\eqref{eq:elastic-force} is a primary concern of the weak structural formulation.
Let $\fesuperspace \subseteq (H^1(\soliddomO))^3$.
Perhaps the simplest approach to address this issue is to use the chain rule to evaluate $\DIV \PPs$, as in the work of Devendran and Peskin~\cite{Devendran2012}.
Another solution (see the work by Boffi \etal~\cite{Boffi2008}) is to use a weak formulation of the structural force and to project it onto the finite element space by defining a Lagrangian structural force $\Fb(\Xb,t) \in \fesuperspace$ satisfying
\begin{equation}
    \int_{\soliddomO} \Fb(\Xb,t) \cdot \psib(\Xb) \dXb
    = \int_{\soliddomO} \left(\DIV \PPs (\Xb,t) \right)
    \cdot \psib(\Xb) \dXb
    - \int_{\partial \soliddomO} (\PPs(\Xb, t) \Nb(\Xb))
    \cdot \psib(\Xb) \dAb
    \label{eq:strong-force}
\end{equation}
for all test functions $\psib(\Xb) \in \fesuperspace$.
In practice, we integrate Equation~\eqref{eq:strong-force} by parts to move the derivative to the test function:
\begin{equation}
    \int_{\soliddomO} \Fb(\Xb,t) \cdot \psib(\Xb) \dXb =
    -\int_{\soliddomO} \PPs(\Xb,t) : \nabla_{\Xb} \psib(\Xb) \dXb.
    \label{eq:weak-force}
\end{equation}
In fact, as mentioned above, these two approaches (Equation~\eqref{eq:weak-force} and Devendran and Peskin~\cite{Devendran2012}) are exactly equivalent in particular and practically useful cases.
The primary difference between them is in the formulation, as this work begins with a finite element approximation and Devendran and Peskin's instead avoids using a weak form.

We discretize the structure $\soliddomO$ via a triangulation $\tria$ with $m$ nodes.
We define the $3 m$-dimensional vector-valued approximation space as
\begin{equation}
  \fespace \subset \fesuperspace \subseteq H^1(\tria)^3.
  \label{eq:fe-space}
\end{equation}
We assume $\fespace$ has an interpolation operator $\Pi_{\fespace}$ corresponding to evaluation of the interpolated function at the nodes of $\tria$, i.e.,
\begin{equation}
  \Pi_{\fespace} f(\Xb) = \sum_{\ell = 1}^{3 m} f(\Xb_\ell) \circ \phib_\ell(\Xb),
  \label{eq:nodal-fe-basis}
\end{equation}
in which $\circ$ is the componentwise (i.e., Hadamard) product.
Hence $\{\phib_\ell\}$ is the standard primitive (i.e., nonzero in exactly one component) nodally interpolating finite element basis of $\fespace$.
Consequently, for each node $\Xb_k$ of $\tria$ there are exactly three basis functions that are nonzero at that node.
For example, $k$, $m + k$, and $2 m + k$ are the unique indices such that
\begin{equation}
    \label{eq:nodal-fe-assumption}
    \phib_{k}(\Xb_k) = (1, 0, 0), \phib_{m + k}(\Xb_k) = (0, 1, 0), \text{ and } \phib_{2 m + k}(\Xb_k) = (0, 0, 1),
\end{equation}
and otherwise $\phib_{i}(\Xb_k) = (0, 0, 0)$ for $i \neq k$, $i \neq m + k$, and $i \neq 2 m + k$.

The deformation $\Chib$, velocity $\Ub$, and force $\Fb$ are all approximated in $\fespace$ and can be written as
\begin{align}
  \label{eq:fem-mapping-basis}
  \Chib_h(\Xb,t) &= \sum_{\ell=1}^{3 m} \chi_{\ell}(t)\phib_{\ell}(\Xb), \\
  \label{eq:fem-velocity-basis}
  \Ub_h(\Xb,t) &= \sum_{\ell=1}^{3 m} U_{\ell}(t)\phib_{\ell}(\Xb), \text{ and} \\
  \label{eq:fem-force-basis}
  \Fb_h(\Xb,t) &= \sum_{\ell=1}^{3 m} F_{\ell}(t)\phib_{\ell}(\Xb).
\end{align}
Because each basis function is nonzero in only one component, for convenience we also define $\Ub_h^n$ and $\Fb_h^n$ for $1 \leq n \leq 3$ as the $n$th components of each finite element field, whose basis functions are indexed in the same order as the mesh nodes (e.g., $U_k^1$ is the $x$-component of the finite element velocity at mesh node $\Xb_k$).
As in Equations~\eqref{eq:fem-mapping-basis}--\eqref{eq:fem-force-basis}, we omit the subscript $h$ when indexing the weights associated with individual basis functions.
We also define a discrete deformation gradient tensor and corresponding discrete first Piola-Kirchhoff stress as
\begin{align}
  \FF_h(\Xb, t) &= \dfrac{\partial \Chib_h}{\partial \Xb}, \text{ and } \\
  \PP_h^\text{e}(\Xb, t) &= \dfrac{\partial \Psi(\FF)}{\partial \FF} \bigg|_{\FF = \FF_h}.
\end{align}
Because $\Chib_h$ is based on nodal finite elements, $\FF_h$ is generally discontinuous at interelement boundaries, and $\DIV \PP$ is only well-defined in the element interiors.
We notate our finite element space as $\Pone$ for linear triangles or tetrahedra, $\Ptwo$ for quadratic triangles or tetrahedra, $\Qone$ for bilinear quadrilaterals, and $\Qtwo$ for biquadratic quadrilaterals.

\subsection{Discretization of the Navier-Stokes Equations}
\label{subsec:discretization-of-the-momentum}
For the remainder of the statement of the discretization, the Eulerian variables are discretized with the marker-and-cell staggered-grid scheme~\cite{harlow1965numerical}, using uniform cells of length $\euleriandx$ in each coordinate direction.
There are several prominent advantages of this staggered scheme, such as its mass conservation properties, efficiency of linear algebra, and inf-sup stability.
In this approximation scheme, $\fb$, $\xb$, and $\ub$ are defined on this staggered-grid such that the $d$th component of each variable is approximated at the midpoint of the cell faces that are perpendicular to the $d$th coordinate axis.
For example, in three spatial dimensions, consider a cell at $(i, j, k)$ in index space.
The cell's barycenter is located at $(\euleriandx (i + \half), \euleriandx (j + \half), \euleriandx (k + \half))$,
its $x$ components of velocity and force are defined at $(\euleriandx i, \euleriandx (j + \half), \euleriandx (k + \half))$ and $(\euleriandx (i + 1), \euleriandx (j + \half), \euleriandx (k + \half))$,
$y$ components at $(\euleriandx (i + \half), \euleriandx j, \euleriandx (k + \half))$ and $(\euleriandx (i + \half), \euleriandx (j + 1), \euleriandx (k + \half))$, and
$z$ components at $(\euleriandx (i + \half), \euleriandx (j + \half), \euleriandx k)$ and $(\euleriandx (i + \half), \euleriandx (j + \half), \euleriandx (k + 1))$.
Without loss of generality, for the sake of simplicity we assume a three-dimensional staggered discretization for the rest of this paper (though the results are immediately applicable to two spatial dimensions).
We index these values with index space coordinates: for instance, $f^3_{\zface}$ is the value of the $z$ component of $\fb$ on the top face of Cartesian grid cell $(i, j, k)$.
Griffith and Luo provide additional information on the numerical scheme (including stabilization, handling ghost values and boundary conditions, and adaptive refinement) in Section 3.1 of their work~\cite{Griffith2017}.

\subsection{Eulerian-Lagrangian Coupling Operators}
\label{subsec:eulerian-lagrangian-coupling-operators}
This section presents a semidiscrete IB method which couples the Eulerian and Lagrangian equations of motion.
By semidiscrete, we mean that the regularized delta function $\delta_h$ has been selected but we do not yet use quadrature formulas to discretize any integrals.
The major goal of this paper is to discretize the interaction equations with nodal quadrature, so precision in the discretizations of these integrals is critical.

Forces are transferred from the structural mesh to the Cartesian grid (i.e., the right-hand side in Equation~\eqref{eq:elastic-force}) by \emph{spreading}.
Following the discretization described in Subsection~\ref{subsec:discretization-of-the-momentum} and a projection onto the finite element space (such as Equation~\eqref{eq:weak-force}), the forces are
\begin{subequations}
\begin{align}
    \label{eq:semidiscrete-spread-1}
    f^1_{\xface} &= \int_\soliddom F^1(\Xb)
    \, \delta_h(\xb_{\xface} - \Chib_h(\Xb,t)) \dXb,
    \\
    f^2_{\yface} &= \int_\soliddom  F^2(\Xb)
    \, \delta_h(\xb_{\yface} - \Chib_h(\Xb,t)) \dXb,
    \\
    \label{eq:semidiscrete-spread-3}
    f^3_{\zface} &= \int_\soliddom  F^3(\Xb)
    \, \delta_h(\xb_{\zface} - \Chib_h(\Xb,t)) \dXb.
\end{align}
\end{subequations}
Velocities are transferred from the Cartesian grid to the structural mesh by \emph{interpolation} to an intermediate Lagrangian velocity $\Ub^{\text{IB}}$:
\begin{subequations}
\begin{align}
    \label{eq:semidiscrete-interp-1}
    U^{\text{IB},1}(\Xb,t) &= \sum_{i,j,k} u^1_{\xface}
    \, \delta_h(\xb_{\xface} - \Chib_h(\Xb,t)) \euleriandx^3,
    \\
    U^{\text{IB},2}(\Xb,t) &= \sum_{i,j,k} u^2_{\yface}
    \, \delta_h(\xb_{\yface} - \Chib_h(\Xb,t)) \euleriandx^3,
    \\
    \label{eq:semidiscrete-interp-3}
    U^{\text{IB},3}(\Xb,t) &= \sum_{i,j,k} u^3_{\zface}
    \, \delta_h(\xb_{\zface} - \Chib_h(\Xb,t)) \euleriandx^3.
\end{align}
\end{subequations}
Note that $\Ub^{\text{IB}}$ is defined for all points in the structural mesh (in particular, at the quadrature points) but is typically not in $\fespace$.
This is the semidiscretization of Equation~\eqref{eq:noslip}.
The remaining integral equation arises from requiring that
\begin{equation}
    \int_{\soliddomO} \Ub_h(\Xb,t) \cdot \phib(\Xb) \,\dXb = \int_{\soliddomO}
    \Ub^{\text{IB}}(\Xb,t) \cdot \phib(\Xb)\,\dXb,
    \label{eq:semidiscrete-velocity-projection}
\end{equation}
for all test functions $\phib(\Xb) \in \fespace$, i.e., $\Ub_h$ is the projection of $\Ub^{\text{IB}}$ onto the finite element space.
Further, note that if we discretize the integrals in Equations~\eqref{eq:semidiscrete-spread-1}--\eqref{eq:semidiscrete-spread-3} and Equation~\eqref{eq:semidiscrete-velocity-projection} with the same quadrature formula then the spreading and interpolation operators are discretely adjoint.
For a thorough discussion on the adjointness of these two operators, see Griffith and Luo~\cite{Griffith2017}.

\section{Approximating the Stress Projection and Coupling}
\label{sec:approximating-the-stress-projection-and-IB-coupling}
As described by Griffith and Luo~\cite{Griffith2017}, force spreading in the IFED method first evaluates the Lagrangian force density at the interaction points using the FE basis functions, and then spreads these point forces to the Cartesian grid using a regularized delta function.
Similarly, interpolation first evaluates the Cartesian grid velocity at the same interaction points using the same regularized delta function, and then uses those sampled velocities as data in solving an $L^2$ projection equation to determine the velocity of the structure.
Many versions of the IB method and its extensions used the structural nodes as interaction points~\cite{Wang2012,Zhang2004,Zhang2007}, but Griffith and Luo introduced the possibility of determining the interaction points through quadrature rules.

In the subsections below, we consider four types of quadratures: a \emph{consistent} quadrature $\mathbb{C}_q = \{(\Xb_q, w_q)\}$, in which
\begin{equation}
  \label{eq:consistent-quadrature}
    \sum_{(\Xb_q, w_q) \in \mathbb{C}_q} \phib_i(\Xb_q) \cdot \phib_j(\Xb_q) w_q = \int_\soliddom \phib_i(\Xb) \cdot \phib_j(\Xb) \dXb,
    \forall \phib_i, \phib_j \in \fespace;
\end{equation}
a \emph{higher-order} quadrature $\mathbb{H}_q = \{(\Xb_q, w_q)\}$, in which
\begin{equation}
  \label{eq:higher-order-quadrature}
    \sum_{(\Xb_q, w_q) \in \mathbb{H}_q}
    -\PPs(\Xb_q, t) :
    \nabla_{\Xb} \phib(\Xb_q) w_q \approx
    -\int_\soliddom \PPs(\Xb, t) :
    \nabla_{\Xb} \phib(\Xb) \dXb,\, \forall \phib_j \in \fespace
\end{equation}
is at least as accurate as interpolation with the finite element space (i.e., the quadrature error here does not dominate the overall error in the scheme);
an \emph{adaptive} quadrature $\mathbb{A}_q = \{(\Xb_q, w_q)\}$ of at least the same approximation order as $\mathbb{C}_q$, in which
\begin{equation}
  \label{eq:adaptive-quadrature}
    \forall \Xb \in \soliddomO,
    \text{ } \min_{q} \|\Chib_h(\Xb, t) - \Chib_h(\Xb_q, t)\| \leq C_{\mathbb{A}} \euleriandx;
\end{equation}
and a \emph{nodal} quadrature $\mathbb{N}_q = \{(\Xb_q, w_q)\}_{q=1}^{m}$, in which
\begin{equation}
  \label{eq:nodal-quadrature}
  \Xb_q \text{ is a node of } \tria \text{ and } w_q = \int \phib_q(\Xb) \cdot \boldsymbol{1} \dXb.
\end{equation}
Here $\delta_{ij}$ is the Kronecker delta, $\phib_i$, $\phib_j$, and $\phib_q$ are elements of the FE basis defined in Equation~\eqref{eq:nodal-fe-basis}, $C_{\mathbb{A}}$ is an $O(1)$ constant (e.g., Peskin~\cite{Peskin2002} uses $C_{\mathbb{A}} = \half$), and $\euleriandx$ is the Cartesian grid spacing.
In general, the adaptive quadrature varies as the structure deforms, and $C_{\mathbb{A}}$ is chosen to be small enough to avoid gaps in the Cartesian grid representation of the structural force density.
For a further discussion on the selection of adaptive quadratures, see Figure 3 and related discussion in the work by Griffith and Luo~\cite{Griffith2017}.

Notice that each of these quadratures is defined across the entire mesh, rather than being defined on a single element.
In practice, we evaluate the non-nodal quadratures in the standard way by looping over the elements of the mesh, but we may instead loop over only the mesh \emph{nodes} for the nodal quadrature.
Consequently, for any quadrature, we evaluate the integrand at each quadrature point exactly once.
We define the quadratures this way (on the entire mesh instead of on individual elements) because it allows us to manipulate the elemental and nodal quadratures in the same manner.

\subsection{Projecting the Divergence of the Elastic Stress}
This subsection examines the difference between using consistent and lumped mass matrices for projecting the divergence of the elastic stress onto the finite element field.

\subsubsection{Consistent Projection}
We may interpret the standard finite element discretization of Equation~\eqref{eq:weak-force} as a gradient recovery of $\DIV \PPs$, specifically one that minimizes the $L^2$ error norm over the entire structural domain through a standard projection onto the finite element space after integrating by parts.

In this case, we apply the standard finite element discretization to Equation~\eqref{eq:weak-force} and obtain the linear system,
\begin{subequations}
\begin{align}
    \label{eq:consistent-div-pk1-projection}
    \mathbb{M} \vec{\Fb} &= \vec{\Lb}, \text{ with }
    \\
    \label{eq:consistent-div-pk1-projection-load-vector}
     \vec{L}_i &= \underset{(\Xb_q, w_q) \in \mathbb{H}_q}{-\sum} \PPs(\Xb_q, t)
    : \nabla_{\Xb} \phib_i(\Xb_q) w_q,
\end{align}
\end{subequations}
in which $\mathbb{M}$ is the standard finite element mass matrix defined for $\fespace$, $\vec{\Fb}$ is the vector of force coefficients defined by Equation~\eqref{eq:fem-force-basis}, and $\vec{\Lb}$ is a load vector.

\subsubsection{Inconsistent Projection}
Alternatively, we may discretize Equation~\eqref{eq:weak-force} by approximating the projection operator with $\mathbb{N}_q$ and the load vector with $\mathbb{H}_q$.
Consequently, we use $\mathbb{N}_q$ from Equation~\eqref{eq:nodal-quadrature} to assemble an \emph{inconsistent} system,
\begin{subequations}
\begin{align}
    \label{eq:inconsistent-div-pk1-projection}
    \mathbb{D} \vec{\Fb} &= \vec{\Lb}, \text{ with }
    \\
    \label{eq:inconsistent-div-pk1-projection-load-vector}
    \vec{L}_i &= \underset{(\Xb_q, w_q) \in \mathbb{H}_q}{-\sum} \PPs(\Xb_q, t)
    : \nabla_{\Xb} \phib_i(\Xb_q) w_q.
\end{align}
\end{subequations}
This results in a diagonal matrix $\mathbb{D}$ with
\begin{equation}
  \mathbb{D}_{i, j} = \sum_{(\Xb_q, w_q) \in \mathbb{N}_q} \phib_i(\Xb_q) \cdot \phib_j(\Xb_q) \tilde{w}_q
  = \delta_{ij} \tilde{w}_k,
  \label{eq:define-lumped-mass}
\end{equation}
in which $\phib_i(\Xb_k) \neq \mathbf{0}$.
We avoid the problem with potentially zero weights by defining $\tilde{w}_q$ as
\begin{equation}
  \tilde{w}_q
  =
  \begin{dcases*}
    w_q & if $w_q > 0$, \\
    1   & otherwise,
  \end{dcases*}
\end{equation}
in which $(\Xb_k, w_k)$ is the $k$th element of $\mathbb{N}_q$.
The choice of $1$ here is justified by Theorem~\ref{thm:fully-nodal-ignore-weights}.
Notice that the same load vector $\vec{\Lb}$ is used in Equations~\eqref{eq:consistent-div-pk1-projection-load-vector} and \eqref{eq:inconsistent-div-pk1-projection-load-vector}.

Using nodal quadrature guarantees that $\mathbb{D}$ is a diagonal matrix, which leads to an immediate formula for each $F_i(t)$ (i.e., the coefficients defined in Equation~\eqref{eq:fem-force-basis}):
\begin{equation}
    \label{eq:inconsistent-force-coefficient-definition}
    F_i(t) =
    \dfrac
    {
    \underset{(\Xb_q, w_q) \in \mathbb{H}_q}{-\sum} \PPs(\Xb_q, t) : \nabla_{\Xb} \phib_i(\Xb_q) w_q
    }
    {
      \mathbb{D}_{i,i}
    }.
\end{equation}
In general, because the transmission force acts as a singular force layer~\cite{Griffith2017}, the notion of pointwise convergence of $\Fb_h$ is not well defined on the boundary of the solid domain.
Hence we will only prove results about pointwise convergence on the interior of $\soliddom$, i.e., for nodes that are not on the boundary of $\soliddom$.
\begin{theorem}
  Assume $\mathbb{H}_q$ is chosen such that the maximum error in load vector entries on each element $K$ is bounded by $C |K| \lagrangiandx^n$, in which $\lagrangiandx$ is the largest element diameter (for all $K_i \in \tria$), $C$ is a constant dependent on $\PPs$'s derivatives, and $n$ is a positive integer.

  If $\PPs$ is sufficiently smooth and each $\phib_i(\Xb)$ from the basis defined in Equation~\eqref{eq:nodal-fe-basis} is nonnegative (e.g., for linear elements) and zero on the boundary of $\soliddom$, then the finite element coefficients calculated in Equation~\eqref{eq:inconsistent-force-coefficient-definition} are first-order accurate, i.e.,
  \begin{equation}
    \max_{(\Xb_q, w_q) \in \mathbb{N}_q}
    \left|
    \Fb_h(\Xb_q, t) - \nabla_{\Xb} \cdot \PPs(\Xb_q, t)
    \right|
    \leq
    C_3 \lagrangiandx,
  \end{equation}
  in which $C_3$ depends on the derivatives of $\nabla_{\Xb} \cdot \PPs$ and the mean value of the basis functions defined in Equation~\eqref{eq:nodal-fe-basis} on their support.
  \label{thm:first-order-forces}
\end{theorem}

\begin{proof}
Rearranging Equation~\eqref{eq:inconsistent-force-coefficient-definition} yields

\begin{equation}
  \int_{\tria} F_i(t) \phib_i(\Xb) \cdot \mathbf{1} \dXb
  =
  \underset{(\Xb_q, w_q) \in \mathbb{H}_q}{-\sum} \PPs(\Xb_q, t) : \nabla_{\Xb} \phib_i(\Xb_q) w_q.
\end{equation}
  Next, we subtract the right-hand side of Equation~\eqref{eq:strong-force} in which $\phib(\Xb) = \phib_i(\Xb)$ from both sides:
\begin{align}
  \int_{\tria}
  \left(
  F_i(t) \mathbf{1}
  -
  \nabla_{\Xb} \cdot \PPs(\Xb, t)\right) \cdot \phib_i(\Xb)
  \dXb
  &=
  \underset{(\Xb_q, w_q) \in \mathbb{H}_q}{-\sum} \PPs(\Xb_q, t) : \nabla_{\Xb} \phib_i(\Xb_q) w_q \\
  &\phantom{=} -
  \int_{\tria} \left(\nabla_{\Xb} \cdot \PPs(\Xb, t)\right) \cdot \phib_i(\Xb) \dXb
  \nonumber
  \\
  \label{eq:integrate-pk1-by-parts}
  &=
  %% note that the boundary integral goes away since we are completely ignoring
  %% transmission forces
  \underset{(\Xb_q, w_q) \in \mathbb{H}_q}{-\sum} \PPs(\Xb_q, t) : \nabla_{\Xb} \phib_i(\Xb_q) w_q \\
  &\phantom{=} +
  \int_{\tria} \PPs(\Xb, t) : \nabla_{\Xb} \phib_i(\Xb) \dXb.
  \nonumber
\end{align}
By assumption on the accuracy of $\mathbb{H}_q$, we can bound the right-hand side of Equation~\eqref{eq:integrate-pk1-by-parts} (and, therefore, the left-hand side of the equality) as
\begin{equation}
  \left|
  \int_{\tria}
  \left(
  F_i(t) \mathbf{1}
  -
  \nabla_{\Xb} \cdot \PPs(\Xb, t)\right) \cdot \phib_i(\Xb)
  \dXb
  \right|
  \leq
  C |\supp(\phib_i(\Xb))| \lagrangiandx^n.
\end{equation}
Because the basis functions are nodal interpolants, $\phib_i(\Xb)$ must be nonzero in exactly one component.
We record that component's index as $d$ and the nonzero part of $\phib_i(\Xb)$ as $\phi_{i,d}(\Xb)$.
Let
\begin{equation}
  \Xb_c = \argmin_{\Xb \in \supp(\phib_i(\Xb))}
  \left|
  (F_i(t) \mathbf{1} - \nabla_{\Xb} \cdot \PPs(\Xb, t))_d
  \right|.
\end{equation}
Hence, as $\phi_{i,d}(\Xb) \geq 0$,
\begin{equation}
  \left|
  \int_\tria
  \left(F_i(t) \mathbf{1} - \nabla_{\Xb} \cdot \PPs(\Xb, t)\right) \cdot \phib_i(\Xb) \dXb
  \right|
  \geq
  \left|
  \left(F_i(t) \mathbf{1} - \nabla_{\Xb} \cdot \PPs(\Xb, t)\right)_d
  \right| \bigg|_{\Xb = \Xb_c}
  \int_\tria \phi_{i,d}(\Xb) \dXb
\end{equation}
so
\begin{align}
  \left|
  \left(F_i(t) \mathbf{1} - \nabla_{\Xb} \cdot \PPs(\Xb, t)\right)_d
  \right| \bigg|_{\Xb = \Xb_c}
  &\leq
  \dfrac{C |\supp(\phib_i(\Xb))| \lagrangiandx^n}{\int_\tria \phi_{i,d}(\Xb) \dXb} \\
  &=
  C_2 \lagrangiandx^n,
\end{align}
in which $C_2$ is a constant dependent on $C$ and the mean value of $\phib_i(\Xb)$ (which must be between $0$ and $1$) on $\supp(\phib_i(\Xb))$.
Consequently, applying a constant extrapolation from $\Xb_c$ to the node at which $\phib_i(\Xb)$ interpolates a value, we achieve the stated result that each $F_i(t)$ is at least first-order accurate.
\end{proof}

\subsection{Force Spreading}
The standard discretization of force spreading using elemental quadrature involves $\mathbb{A}_q$:
\begin{subequations}
\begin{align}
    \label{eq:discrete-spread-elemental-1}
    f^1_{\xface} = \sum_{(\Xb_q, w_q) \in \mathbb{A}_q} F^1(\Xb_{q}, t)
    \, \delta_h(\xb_{\xface} - \Chib_h(\Xb_{q},t))w_{q},
    \\
    \label{eq:discrete-spread-elemental-2}
    f^2_{\yface} = \sum_{(\Xb_q, w_q) \in \mathbb{A}_q} F^2(\Xb_{q}, t)
    \, \delta_h(\xb_{\yface} - \Chib_h(\Xb_{q},t))w_{q},
    \\
    \label{eq:discrete-spread-elemental-3}
    f^3_{\zface} = \sum_{(\Xb_q, w_q) \in \mathbb{A}_q} F^3(\Xb_{q}, t)
    \, \delta_h(\xb_{\zface} - \Chib_h(\Xb_{q},t))w_{q}.
\end{align}
\end{subequations}
Alternatively, if we use $\mathbb{N}_q$ to approximate the integral, we obtain
\begin{subequations}
\begin{align}
    \label{eq:discrete-spread-nodal-1}
    f^1_{\xface} &= \sum_{(\Xb_q, w_q) \in \mathbb{N}_q} F^1_{q}(t)
    \, \delta_h(\xb_{\xface} - \Chib_h(\Xb_{q}, t)) \mathbb{D}_{q, q},
    \\
    \label{eq:discrete-spread-nodal-2}
    f^2_{\yface} &= \sum_{(\Xb_q, w_q) \in \mathbb{N}_q} F^2_{q}(t)
    \, \delta_h(\xb_{\yface} - \Chib_h(\Xb_{q}, t)) \mathbb{D}_{q, q},
    \\
    \label{eq:discrete-spread-nodal-3}
    f^3_{\zface} &= \sum_{(\Xb_q, w_q) \in \mathbb{N}_q} F^3_{q}(t)
    \, \delta_h(\xb_{\zface} - \Chib_h(\Xb_{q}, t)) \mathbb{D}_{q, q}.
\end{align}
\end{subequations}

A theoretical investigation of the effect of spreading with $\mathbb{N}_q$ versus spreading with $\mathbb{A}_q$ is beyond the scope of this study. However, Section~\ref{sec:benchmarks} includes some computational benchmarks showing that, for a sufficiently dense mesh, the differences in practice are acceptably small.

\subsection{Velocity Projection and Velocity Interpolation}
Discretizing Equation~\eqref{eq:semidiscrete-velocity-projection} with $\mathbb{A}_q$ recovers the velocity coupling operator described by Griffith and Luo~\cite{Griffith2017}:
\begin{align}
    \label{eq:elemental-velocity-projection}
    \mathbb{M} \vec{\Ub} &= \LbVecIB, \text{ with }
    \\
    \LbVecIB_i &= \sum_{(\Xb_q, w_q) \in \mathbb{A}_q} \Ub^{\text{IB}}(\Xb_q) \cdot \phib_i(\Xb_q) w_q.
\end{align}
Alternatively, if we discretize both the left and right sides with $\mathbb{N}_q$ from Equation~\eqref{eq:nodal-quadrature} then we obtain
\begin{equation}
    \label{eq:inconsistent-velocity-projection}
    \mathbb{D} \vec{\Ub} = \mathbb{D} \UbVecIB,
\end{equation}
in which $\vec{\Ub}$ is the vector of finite element coefficients of Equation~\eqref{eq:fem-velocity-basis}, $\UbVecIB$ is a vector populated with values of $\Ub^{\text{IB}}$ at the nodes of $\tria$,
and $\mathbb{D}$ is the same diagonal mass matrix that appeared in Equation~\eqref{eq:inconsistent-div-pk1-projection} (since we use the same space $\fespace$ for both $\Ub$ and $\Fb$).
In this case, $\mathbb{D}$ appears on both the left and right sides because, with the choice of nodal quadrature,
\begin{align}
    \int \Ub^{\text{IB}}(\Xb, t) \cdot \phib_i(\Xb) \dXb
    \approx
    \sum_{(\Xb_q, w_q) \in \mathbb{N}_q} \Ub^{\text{IB}}(\Xb_q, t) \cdot \phib_i(\Xb_q) w_q
    = \Ub^{\text{IB}}(\Xb_i) \mathbb{D}_{i,i}.
\end{align}
Hence, with nodal quadrature, the projection of the velocity with $\Ub^{\text{IB}}$ evaluated at the points of $\mathbb{N}_q$ is exactly the same as interpolating $\ub$ at the nodes of $\tria$ with $\Pi_{\fespace}$.

\subsection{A Fully Nodal Coupling Approach}
Combining results from the last three sections we arrive at the following straightforward, but critical, result:

\begin{theorem}
  \label{thm:fully-nodal-ignore-weights}
  If, in the IFED method, the force projection, force spreading, and velocity projection operators are all discretized with the same nodal quadrature rule $\mathbb{N}_q$, then the values of $w_q$, which correspond to the diagonal entries of $\mathbb{D}$, can be chosen as \emph{arbitrary} nonzero values.
\end{theorem}

\begin{proof}
  Equation~\eqref{eq:inconsistent-velocity-projection} clearly establishes the result for the velocity.
  Without loss of generality, we shall only examine the first component of force; the others are the same except for the changes of indices.
  Substituting Equation~\eqref{eq:inconsistent-force-coefficient-definition} into Equation~\eqref{eq:discrete-spread-nodal-1} yields
  \begin{subequations}
  \begin{align}
        f^1_{\xface}
        &= \sum_{(\Xb_q, w_q) \in \mathbb{N}_q} \Fb^{1}(\Xb_{q}, t)
           \, \delta_h(\xb_{\xface} - \Chib_h(\Xb_{q}, t)) \mathbb{D}_{q, q}
        \\
        &= \sum_{(\Xb_q, w_q) \in \mathbb{N}_q} \Fb^{1}_q(t)
           \, \delta_h(\xb_{\xface} - \Chib_h(\Xb_{q}, t)) \mathbb{D}_{q, q}
        \\
        &= \sum_{(\Xb_q, w_q) \in \mathbb{N}_q}
           \left(
           \dfrac
           {
           \underset{(\Xb_j, w_j) \in \mathbb{H}_q}{-\sum} \PPs(\Xb_j, t) : \nabla_{\Xb} \phib_q(\Xb_j) w_j
           }
           {
             \mathbb{D}_{q,q}
           }
           \right)
           \delta_h(\xb_{\xface} - \Chib_h(\Xb_{q}, t)) \mathbb{D}_{q, q}
        \\
        &= \sum_{(\Xb_q, w_q) \in \mathbb{N}_q}
           \left(
           \underset{(\Xb_j, w_j) \in \mathbb{H}_q}{-\sum} \PPs(\Xb_j, t) : \nabla_{\Xb} \phib_q(\Xb_j) w_j
           \right)
           \delta_h(\xb_{\xface} - \Chib_h(\Xb_{q}, t)),
           \label{eq:nodal-spreading-weight-independence}
  \end{align}
  \end{subequations}
  which is independent of $\mathbb{D}$.
\end{proof}

\begin{remark}
  Integrating $\nabla_{\Xb} \phib_q(\Xb)$ by parts in Equation~\eqref{eq:nodal-spreading-weight-independence} (and, like in Theorem~\ref{thm:first-order-forces}, ignoring boundary integrals) yields the force density
  \begin{equation}
    \underset{(\Xb_j, w_j) \in \mathbb{H}_q}{-\sum} \PPs(\Xb_j, t) : \nabla_{\Xb} \phib_q(\Xb_j) w_j
    \approx
    % don't use italics here, even though we are in a remark
    \int_{\mathrm{supp}(\phib_q(\Xb))} \nabla_{\Xb} \cdot \PPs(\Xb, t) \cdot \phib_q(\Xb) \dXb.
  \end{equation}
  % same (use \mathrm{supp}, not \supp)
  This is a weighted average (biased towards its value at $\Xb_q$) of $\nabla_{\Xb} \cdot \PPs$ across $\mathrm{supp}(\phib_q(\Xb))$.
  This is analogous to Equation 4.18 in the work of Peskin~\cite{Peskin2002}.

  The fully nodal coupling approach described here is essentially the classic IB method, in which the force is computed in a slightly different way (resulting from its definition as a finite element field instead of pointwise values on a grid).
  \label{rem:fully-lumped-ib-equivalence}
\end{remark}

\begin{remark}
  Although mass lumping is commonly used with linear finite elements, in general it cannot be used directly with higher-order finite elements because the corresponding nodal quadrature rules will have either zero or negative weights at least for some components.
  For example, the integrals of the vertex basis functions of the standard two-dimensional and three-dimensional tetrahedral quadratic elements are zero, so their corresponding entries in $\mathbb{D}$ will be zero.
  Other high-order elements, such as those based on tensor products or special bubble functions~\cite{Geevers_2018}, do not exhibit such issues.

  Since the fully nodal scheme is independent of $\mathbb{D}$, it may be used with finite elements like $\Ptwo$ that do not normally work with mass lumping.
  This is examined in Section~\ref{sec:benchmarks}.
\end{remark}

The next theorem shows that the use of $\mathbb{D}$, when combined with $\mathbb{N}_q$ for interaction, does not affect the typical moment conservation properties of the IB method.
\begin{theorem}
  \label{thm:equal-forces}
  Let the kernel function $\delta_h$ satisfy the zeroth and first discrete moment conditions, as described by Peskin~\cite{Peskin2002},
  \begin{align}
    \label{eq:zeroth-moment-condition}
    \sum_{i,j,k} \euleriandx^3 \delta_h(\xb_{\xface} - \xb) = 1 \text{ and } \\
    \label{eq:first-moment-condition}
    \sum_{i,j,k} \euleriandx^3 x^1_{\xface}\delta_h(\xb_{\xface} - \xb) = x^1,
  \end{align}
  along with the equivalent conditions in the second and third components.
  If we use the same quadrature rule $\mathbb{Q}_q$ for both integration of the mass matrix and force spreading, then the force defined on the Cartesian grid will always satisfy the \emph{same} zeroth and first force moment conditions, \emph{independent} of $\mathbb{Q}_q$.
\end{theorem}

\begin{proof}
  Let $\mathcal{M}$ be the mass matrix associated with the quadrature rule $\mathbb{Q}_q$ (e.g., $\mathcal{M}$ = $\mathbb{D}$ if $\mathbb{Q}_q = \mathbb{N}_q$ and $\mathcal{M} = \mathbb{M}$ if $\mathbb{Q}_q = \mathbb{A}_q$).
  Then
  \begin{equation}
    \sum_{(\Xb_q, w_q) \in \mathbb{Q}_q} (F^1(\Xb_q, t) + F^2(\Xb_q, t) + F^3(\Xb_q, t)) w_q
    =
    \vecOneT \mathcal{M} \vec{\Fb}
    = \vecOneT \vec{\Lb}
    \label{eq:summation-of-mass-matrix}
  \end{equation}
  by either Equation~\eqref{eq:consistent-div-pk1-projection} or Equation~\eqref{eq:inconsistent-div-pk1-projection}.
  Without loss of generality, we show the $x$-component in detail:
  \begin{align}
    \sum_{i,j,k} \euleriandx^3 f^1_{\xface}
    &= \sum_{i,j,k} \euleriandx^3 \left\{
        \sum_{(\Xb_q, w_q) \in \mathbb{Q}_q} F^1(\Xb_q, t) \delta_h\left(\xb_{\xface} - \Chib_h(\Xb_q, t)\right) w_q
       \right\} \\
    &= \sum_{ (\Xb_q, w_q) \in \mathbb{Q}_q}  F^1(\Xb_q, t) w_q
       \sum_{i, j,k} \euleriandx^3 \delta_h\left(\xb_{\xface} - \Chib_h(\Xb_q, t)\right) \\
    &= \sum_{ (\Xb_q, w_q) \in \mathbb{Q}_q} F^1(\Xb_q, t) w_q \cdot 1.
  \end{align}
  The last equality is the $x$-component of left hand side of Equation~\eqref{eq:summation-of-mass-matrix}.
  Combining this with the corresponding equations for the other components yields
  \begin{align}
    \sum_{i,j,k} \euleriandx^3 \left(f^1_{\xface} + f^2_{\yface} + f^3_{\zface} \right) =
    \vecOneT \vec{\Lb}.
  \end{align}
  The argument for first moments follows similarly.
  We show the calculation for the $f^1_{\xface} x^1_{\xface}$ moment, in which $x^1_{\xface} = i\euleriandx$ is the $x$-coordinate of the corresponding edge.
  For this moment we have
  \begin{align}
    \sum_{i,j,k} \euleriandx^3 f^1_{\xface} x^1_{\xface}
    &= \sum_{i,j,k} \euleriandx^3 x^1_{\xface} \left\{
       \sum_{(\Xb_q, w_q) \in \mathbb{Q}_q} F^1(\Xb_q,t) \delta_h\left(\xb_{\xface} - \Chib_h(\Xb_q,t)\right)w_q
       \right\} \\
    &= \sum_{(\Xb_q, w_q) \in \mathbb{Q}_q}  F^1(\Xb_q, t) w_q
       \sum_{i,j,k} \euleriandx^3 x^1_{\xface} \delta_h\left(\xb_{\xface} - \Chib_h(\Xb_q, t)\right) \\
    &= \sum_{(\Xb_q, w_q) \in \mathbb{Q}_q} F^1(\Xb_q, t)\chi^1(\Xb_q, t) w_q.
  \end{align}
  The last equality follows from Equation~\eqref{eq:first-moment-condition}.
  Summing the components yields
  \begin{equation}
    \sum_{(\Xb_q, w_q) \in \mathbb{Q}_q}
    (F^1(\Xb_q, t) \chi^1(\Xb_q, t) + F^2(\Xb_q, t) \chi^2(\Xb_q, t) + F^3(\Xb_q, t) \chi^3(\Xb_q, t)) w_q
    =
    \vecT{\Chib} \mathcal{M} \vec{\Fb} = \vecT{\Chib} \vec{\Lb}
  \end{equation}
  by the definition of $\mathcal{M}$.
  Hence
  \begin{equation}
    \sum_{i,j,k} \euleriandx^3\left(f^1_{\xface} x^1_{\xface} + f^2_{\yface} x^2_{\yface} + f^3_{\zface} x^3_{\zface} \right) = \Chib^T \vec{\Lb}.
  \end{equation}
\end{proof}

\begin{remark}
  Essentially, this means if we use the same quadrature rule to approximate the mass matrix and the coupling operators, $\fb$ and $\Fb$ are equivalent as densities, i.e., the discrete integral of the discrete force $\fb$ on the Cartesian grid will equal the integral of $\Fb_h$ with $\mathbb{Q}_q$ (from substituting $\vec{\Lb} = \mathcal{M} \vec{\Fb}$) on the structural mesh.
  Further, an important consequence of this result is that using nodal quadrature for spreading and interpolation but consistent quadrature for force projection is \emph{not guaranteed} to discretely maintain this equivalence.
  However, this correspondence is necessary to avoid the spurious creation or destruction of momentum in the fluid-structure coupling~\cite{Peskin2002}.
  Hence, if we use nodal quadrature for interaction, we should also use nodal quadrature to define the mass matrix used in the force projection, and vice versa.
\end{remark}

\begin{remark}
  The above arguments can be applied to all other first force moments.
  We may combine them all together to show that the total Eulerian torque (defined on a Cartesian staggered grid) always equals the same quantity if we use the same quadrature rule for the mass matrix and the Eulerian-Lagrangian coupling.
  Furthermore, by using the first moments of force we may demonstrate that the total potential energy in the Lagrangian frame $\left(\vec{\Chib} - \vec{\Chib}_0\right)^T \vec{\Lb}$ is conserved when prolonged to the Eulerian grid.
  Although we do not have an identity for the Lagrangian kinetic energy, the adjointness of the coupling operators ensures that energy is not spuriously created or destroyed.
\end{remark}

\section{Implementation}
\label{sec:implementation}
In this study, the IFED nodal coupling scheme uses both Gaussian and nodal quadrature schemes.
Specifically, the force projection uses Gaussian quadrature for the load vector and nodal quadrature for the approximate mass matrix, whereas the coupling operators only use nodal quadrature.
In contrast, the elemental coupling scheme uses Gaussian quadrature to approximate all integrals.
Algorithms~\ref{alg:nodal-algorithm} and \ref{alg:elemental-algorithm} summarize both the nodal and elemental coupling methods.

\begin{algorithm}
  \caption{Nodal Coupling Scheme}
  \label{alg:nodal-algorithm}
  \begin{algorithmic}[1]
    \State Define the nodal quadrature rule as
    \begin{equation}
      \mathbb{N}_q = \{(\Xb_q, w_q)\}_{q=1}^{m} \text{ in which }
      \Xb_q \text{ is a node of } \tria \text{ and } w_q = \int \phib_q(\Xb) \cdot \boldsymbol{1} \dXb
    \end{equation}

    \Procedure{Velocity Interpolation}{$\Chib_h, \ub_h$, $\delta_h$}\Comment{Interpolate $\ub_h$ in the FE space}
    \State Compute the components of $\UbVecIB$ by evaluating $\delta_h$ at each node's displaced location $\Chib_h(\Xb_q)$ for each $(\Xb_q, w_q) \in \mathbb{N}_q$:
    \begin{subequations}
    \begin{align}
      U^{\text{IB},1}_q &= \sum_{i,j,k} u^1_{\xface}
      \, \delta_h(\xb_{\xface} - \Chib_h(\Xb_q)) \euleriandx^3
      \\
      U^{\text{IB},2}_q &= \sum_{i,j,k} u^2_{\yface}
      \, \delta_h(\xb_{\yface} - \Chib_h(\Xb_q)) \euleriandx^3
      \\
      U^{\text{IB},3}_q &= \sum_{i,j,k} u^3_{\zface}
      \, \delta_h(\xb_{\zface} - \Chib_h(\Xb_q)) \euleriandx^3
    \end{align}
    \end{subequations}
    \State \textbf{return} $\UbVecIB$\Comment{vector of FE velocity coefficients}
    \EndProcedure%

    \Procedure{Force Spreading}{$\Chib_h, \delta_h$}
    \State For each FE basis function compute a mean force contribution
    \begin{equation}
      \tilde{F}_i =
      \underset{(\Xb_q, w_q) \in \mathbb{H}_q}{-\sum} \PPs(\Xb_q) : \nabla_{\Xb} \phib_i(\Xb_q) w_q
    \end{equation}
    \State Spread each component of the mean force contribution at each node's displaced location $\Chib_h(\Xb_q)$ as
    \begin{subequations}
    \begin{align}
      f^1_{\xface} &= \sum_{(\Xb_q, w_q) \in \mathbb{N}_q} \tilde{F}^1_{q}
      \, \delta_h(\xb_{\xface} - \Chib_h(\Xb_{q}))
      \\
      f^2_{\yface} &= \sum_{(\Xb_q, w_q) \in \mathbb{N}_q} \tilde{F}^2_{q}
      \, \delta_h(\xb_{\yface} - \Chib_h(\Xb_{q}))
      \\
      f^3_{\zface} &= \sum_{(\Xb_q, w_q) \in \mathbb{N}_q} \tilde{F}^3_{q}
      \, \delta_h(\xb_{\zface} - \Chib_h(\Xb_{q}))
    \end{align}
    \end{subequations}

    \State \textbf{return} $\fb_h$\Comment{Cartesian grid force representation}
    \EndProcedure%
  \end{algorithmic}
\end{algorithm}

Several of our benchmarks do not have analytic solutions.
In these cases, we compute benchmark solutions using a stabilized $\Pone$/$\Pone$ FE method for large-deformation incompressible elasticity~\cite{Chiumenti, Masud2013} implemented in BeatIt~\cite{Beatit}.
The numerical methods described in Section~\ref{sec:approximating-the-stress-projection-and-IB-coupling} are implemented in the IBAMR library~\cite{IBAMR, Griffith2007}.
The time stepping and fluid discretization schemes are described at length by Griffith and Luo~\cite{Griffith2017}.
Both IBAMR and BeatIt rely on the parallel C++ FE library libMesh~\cite{libMeshPaper}, and on linear and nonlinear solver infrastructure provided by the PETSc library~\cite{petsc-web-page}.

Earlier work~\cite{Griffith2017} suggested that the fluid-solid coupling is sensitive to the relative grid spacing between the Cartesian grid and structural mesh.
This ratio of grid spacings, which we call the \emph{mesh factor}, is defined by $\mfac = \frac{\lagrangiandx}{\efac \euleriandx}$, in which the \emph{element factor} $\efac$ is $1$ for linear elements and $2$ for quadratic elements.
$\efac$ reflects the fact that nodes are approximately $\lagrangiandx/2$ apart for quadratic elements.
Specifically, it has been shown that models in which shear stress dominates pressure along the fluid-structure interface give higher accuracy with a relatively coarser structural mesh, whereas pressure-loaded models give higher accuracy with a relatively finer structural mesh compared to the Cartesian grids~\cite{Griffith2017,Lee2021}.

We also remark that the nodal quadrature rule used for $\Ptwo$ elements is not the typical Newton-Cotes rule used for these elements, because that rule assigns weights of zero to the vertices.
Instead, we use a composite trapezoid rule, which has positive weights but cannot integrate quadratics exactly.
Fortunately, as discussed in Theorem~\ref{thm:fully-nodal-ignore-weights}, with nodal coupling (i.e., using nodal quadrature for both interaction and approximating the projection operators) this implementation detail is irrelevant because the nodal coupling scheme does not require any nodal quadrature weights.
In particular, the precision of this nodal quadrature rule has no impact on the results of the overall nodal IFED methodology.

For our benchmarks, we use various types of boundary conditions, including traction and Dirichlet, on the Cartesian grid.
Details are described by Griffith~\cite{Griffith2009}.
Similarly, for the structural mechanics, both traction and Dirichlet boundary conditions are applied on their respective parts of the boundary of the structural mesh.
However, we use a penalty approach that approximately imposes Dirichlet conditions on the structure along a portion of the fluid-structure interface via penalizing deviations from a prescribed displacement and damping non-zero velocities.
Specifically, we apply a surface traction force
\begin{equation}
  \Tb_{\text{S}} = \kappa_{\text{S}}\left(\cb_{\text{S}} - \cb \right) - \eta_{\text{S}}\Ub
\end{equation}
on the part of $\partial\soliddom$ where Dirichlet conditions are desired.
Here, $\kappa_{\text{S}}$ is a penalty parameter that scales like $\kappa_{\text{S}} \propto \frac{\Delta x}{\left(\Delta t\right)^2}$, and $\eta_{\text{S}}$ is a damping parameter that scales like $\eta_{\text{S}} \propto \frac{\rho}{\Delta t}$.
%% \DRW{If we did anything else (like picked values based on elastic moduli) we should document that here.}

In a similar fashion, penalty methods and damping parameters are used to enforce rigidity in the interior of the structural mesh for some of the benchmarks.
We achieve this by imposing a body force density of the form
\begin{equation}
  \Fb_{\text{B}} = \kappa_{\text{B}} \left(\cb_{\text{B}} - \cb \right) - \eta_{\text{B}} \Ub,
\end{equation}
in which $\Fb_{\text{B}}$ has units of force per unit volume.
Additionally, a body force with only the damping term may be used for a flexible structure to assist the system in reaching steady state.
We use the same scaling as used for the surface penalty parameters.

Finally, both the fluid and solid are modeled as incompressible in our benchmarks.
Incompressibility on the Cartesian grid is ensured in the discretized equations throughout the entire computational domain, i.e., $\nabla_h \cdot \ub = 0$ in each Cartesian grid cell.
However, when the discrete velocity is restricted to the structural mesh, pointwise incompressibility is generally lost, and $J \neq 1$ for some points $\xb \in \soliddom$.
This results mainly from the discrete coupling operators, the choice of finite element space for the structural velocity and deformation, and time integration errors.
Using a penalty function valid in the solid region and material models with modified tensor invariants substantially decreases errors in solid incompressibility in the discretized IFED equations~\cite{Vadala-Roth2020}.
Section 2.4 of Vadala-Roth \etal~\cite{Vadala-Roth2020} defines the modified material models used in the benchmarks below.
These modified material models lead to Cauchy stresses of the form
\begin{equation}
    \cauchy = \cauchyv - p \mathbb{I} +
    \begin{cases}
        \ztensor & \xb \in \fluiddom, \\
        \dfrac{1}{J} \dfrac{\partial \Psi(\FFb)}{\partial \FF} \FF^T - \pstab\mathbb{I} & \xb \in \soliddom,
    \end{cases}
    \label{dev-stab}
\end{equation}
in which $\FFb = J^{-1/3} \FF$, $\pstab$ is an additional isotropic (pressure-like) stress that penalizes dilatational deformations, and $\cauchyv = \frac{\mu}{2}\left(\nabla \ub + \nabla \ub^T \right)$ is the viscous stress.
We choose a stabilizing pressure such that $\pstab = 0$ when $J = 1$.
This ensures that the penalty stress has no effect in the continuous formulation, in which incompressibility ($J=1$) is exactly maintained.
For our benchmarks, we use $\pstab = -\frac{\kappas}{J}\ln J$, in which $\kappas$ is the \emph{numerical bulk modulus} determined from other material parameters.
Note that the values for $\kappas$ are particularly low if compared to other penalty methods such as Reese \etal~\cite{Reese1999} because the structure ``inherits'' some incompressibility from the Cartesian discretization used to approximate the Eulerian velocity that is projected to obtain the Lagrangian velocity; see Table~\ref{tb:cb-param}.

\section{Benchmarks}
\label{sec:benchmarks}
This section examines the effect of mass lumping and nodal quadrature rules on the IFED method using a series of benchmarks.
In the all of the benchmarks, we examine two versions of the IFED method: the original method as proposed by Griffith and Luo~\cite{Griffith2017} that uses adaptive quadrature to define the discrete coupling operators and the new nodally coupled method presented in this study.
Recall that Theorem~\ref{thm:equal-forces} tells us that using nodal quadrature only for the projection operators or only for the coupling operators is not guaranteed to maintain discrete zeroth and first order force moments and may possibly generate spurious momentum changes.
We remark that different time-integration schemes may be used with the coupling schemes detailed in Algorithms~\ref{alg:nodal-algorithm} and \ref{alg:elemental-algorithm}.
All benchmarks herein use an explicit midpoint method described by Griffith and Luo~\cite{Griffith2017} for the Eulerian-Lagrangian coupling and a Crank-Nicolson Adams-Bashforth 2 method for the incompressible Navier-Stokes equations.
Preliminary results indicated that using nodal quadrature for the right hand side of the force calculation with nodal coupling made little difference in the results; in other words, using $\mathbb{N}_q$ in Equation~\eqref{eq:inconsistent-div-pk1-projection-load-vector} to achieve a consistently lumped projection had no noticeable effect.
Additionally, Gauss quadrature generally uses fewer integration points than the corresponding equal-order nodal rules, providing for a more economical method.
However, exploration of the quadrature rules used for the force calculation is a possible area for future research.

All benchmarks, unless otherwise noted, use the three-point B-spline kernel.
See Lee \etal~\cite{Lee2021} for a discussion on the relative efficacy of different kernels.

In general, nodal coupling requires about $5-10$ times fewer interaction points than elemental coupling.
The exact number is highly dependent on the mesh topology.
For example, nodes tend to be shared among more elements in three spatial dimensions than in two dimensions, which can lead to a more substantial reduction in the number of interaction points for three-dimensional models.
Because the computational effort for the coupling operators is proportional to the number of interaction points, nodal coupling is about $5-10$ times less expensive than elemental coupling.
We report interaction point counts for the bioprosthetic heart valve benchmark in Table~\ref{tb:bhv-ib-points}.
In addition, nodal interaction completely avoids solving nontrivial linear systems in force spreading and velocity interpolation.
The percentage of total solver time used by the coupling operators is highly dependent on the ratio of structural mesh elements to Cartesian grid cells, performance of the fluid solver, and material model complexity.
Elemental coupling typically requires about 10--50\% of total solver time, with relatively thin structures requiring less time and relatively thick structures requiring more.

\begin{algorithm}
  \caption{Elemental Coupling Scheme}
  \label{alg:elemental-algorithm}
  \begin{algorithmic}[1]
    \State Define the adaptive quadrature rule as a composite Gauss rule such that
    \begin{equation}
      \mathbb{A}_q = \{(\Xb_q, w_q)\} \text{ in which }
      \forall \Xb \in \soliddomO,
      \text{ } \min_{q} \|\Chib_h(\Xb) - \Chib_h(\Xb_q)\| \leq C_{\mathbb{A}} \euleriandx
    \end{equation}
    in which typically $C_\mathbb{A} = 1/2$

    \Procedure{Velocity Interpolation}{$\Chib_h, \ub_h$, $\delta_h$}\Comment{Interpolate and project $\ub_h$ in the FE space}

    \State Compute the components of $\Ub^{\text{IB}}$, for each $(\Xb_q, w_q) \in \mathbb{A}_q$, as
    \begin{subequations}
      \begin{align}
        U^{\text{IB},1}(\Xb_q) &= \sum_{i,j,k} u^1_{\xface}
        \, \delta_h(\xb_{\xface} - \Chib_h(\Xb_q)) \euleriandx^3
        \\
        U^{\text{IB},2}(\Xb_q) &= \sum_{i,j,k} u^2_{\yface}
        \, \delta_h(\xb_{\yface} - \Chib_h(\Xb_q)) \euleriandx^3
        \\
        U^{\text{IB},3}(\Xb_q) &= \sum_{i,j,k} u^3_{\zface}
        \, \delta_h(\xb_{\zface} - \Chib_h(\Xb_q)) \euleriandx^3
      \end{align}
    \end{subequations}

    \State Set up the projection right-hand side $\LbVecIB$ with entries
    \begin{equation}
      L^{\text{IB}}_i = \sum_{(\Xb_q, w_q) \in \mathbb{A}_q} \Ub^{\text{IB}}(\Xb_q) \cdot \phib_i(\Xb_q) w_q
    \end{equation}

    \State Solve $\mathbb{M} \vec{\Ub} = \LbVecIB$ for $\vec{\Ub}$\Comment{$\mathbb{M}$ is the standard mass matrix}

    \State \textbf{return} $\vec{\Ub}$\Comment{vector of FE velocity coefficients}
    \EndProcedure%

    \Procedure{Force Spreading}{$\Chib_h, \delta_h$}
    \State Calculate $\vec{\Lb}$ with
    \begin{equation}
      L_i = \underset{(\Xb_q, w_q) \in \mathbb{H}_q}{-\sum} \PPs(\Xb_q, t)
      : \nabla_{\Xb} \phib_i(\Xb_q) w_q
    \end{equation}

    \State Solve $\mathbb{M} \vec{\Fb} = \vec{\Lb}$ for $\vec{\Fb}$\Comment{$\mathbb{M}$ is the standard mass matrix}

    \State Define $\Fb_h(\Xb) = \sum_i F_i \phib_i(\Xb)$

    \State Spread $\Fb_h(\Xb)$ with into $\fb_h$ with
    \begin{subequations}
      \begin{align}
        f^1_{\xface} = \sum_{(\Xb_q, w_q) \in \mathbb{A}_q} F^1(\Xb_{q})
        \, \delta_h(\xb_{\xface} - \Chib_h(\Xb_{q}))w_{q}
        \\
        f^2_{\yface} = \sum_{(\Xb_q, w_q) \in \mathbb{A}_q} F^2(\Xb_{q})
        \, \delta_h(\xb_{\yface} - \Chib_h(\Xb_{q}))w_{q}
        \\
        f^3_{\zface} = \sum_{(\Xb_q, w_q) \in \mathbb{A}_q} F^3(\Xb_{q})
        \, \delta_h(\xb_{\zface} - \Chib_h(\Xb_{q}))w_{q}
      \end{align}
    \end{subequations}

    \State \textbf{return} $\fb_h$\Comment{Cartesian grid force representation}
    \EndProcedure%
  \end{algorithmic}
\end{algorithm}

\subsection{Effects of $\mfac$}
\label{subsec:mfac}
Elemental coupling permits the use of relatively coarse structural meshes while preventing gaps in the Cartesian grid representation of the force density by evaluating the regularized delta function at adaptively determined interaction points, like those specified by Equation~\eqref{eq:adaptive-quadrature}.
Lee and Griffith~\cite{Lee2021} explore the effect of varying the relative structural mesh sizes (or $\mfac$) in a range of test cases, which can be classified into shear-dominant flow (considered here in Section~\ref{subsubsec:channel_flow}) and fluid pressure-loaded (considered here in Section~\ref{subsubsec:elastic_band}) cases.
The results from these studies suggest that the accuracy is improved for shear-dominant cases when using relatively coarser structural meshes, whereas for pressure-loaded cases the structural mesh needs to be as fine or finer than the Cartesian grid ($\mfac \leq 1$) to prevent the aforementioned gaps.
This section investigates the effect of $\mfac$ on the accuracy of our solutions, using nodal coupling, to benchmark cases adopted from the study by Lee and Griffith~\cite{Lee2021}.
We then compare these results to results obtained using elemental coupling.

Because we use the scalar $\mfac$ to control the relative size of elements compared to the background grid, it typically does not capture any information on the aspect ratios of different elements in the structural mesh.
Just as the FE method may suffer from poor accuracy in bending cases when there is a large aspect ratio, we have observed the IFED method to suffer from similar issues.
Using second-order elements often helps to ameliorate this issue, as demonstrated in the study by Lee~\etal~\cite{Lee2020}.
The thesis work of Vadala-Roth~\cite{Vadala-RothThesis} describes an improvement to the IFED method for first order quadrilateral elements with large aspect ratios in cases with bending.

Each benchmark uses the largest possible stable time step.
Although we have observed some time integration instabilities when we are outside the range of acceptable $\mfac$ values for a given simulation, we have not systematically explored this relationship.
Further, we have not observed a noticeable relationship between the choice of coupling scheme and the stability of the IFED methodology.

\subsubsection{Two-dimensional Channel Flow}
\label{subsubsec:channel_flow}
This benchmark examines flow in a channel that is not aligned with the coordinate axes.
Lee and Griffith~\cite{Lee2021} show that the accuracy, with elemental coupling, is improved for this shear-dominant case when using relatively coarser structural meshes.
The computational domain for this nondimensionalized benchmark is $[0,L]\times[0,L]$, with $L = 6$.
The Cartesian grid uses $N = 128$ cells in each coordinate direction on the coarsest level and three refinement levels with a refinement ratio of $2$.
The time step size is $5.85 \cdot 10^{-4}$, which is near the largest stable time step for this problem.
The structures are two parallel plates of width $D=1$ and wall thickness $w = 0.24$ (see Figure~\ref{fig:channel_error}a).
At the inlet and outlet of the channel, we impose velocity boundary conditions using the analytical steady-state velocity solution described by the plane Poiseuille equation rotated by an angle $\theta$, $u(y) = \frac{\chi D}{2\mu}\eta\left(1-\frac{\eta}{D}\right)$, in which $\eta = -x\sin\theta+(y-y_0)\cos(\theta)$, $y_0$ is the height of the inner wall of the lower plate, and $\chi=\frac{2p_0}{L/\cos(\theta)+D\tan(\theta)}$ is the pressure gradient between the inlet and the outlet.
Our simulation results are compared to this analytical solution with the parameters listed in Table~\ref{tb:channel-flow}.

The channel walls are discretized using $\Pone$ elements.
Figure~\ref{fig:channel_error}b compares results from Lee and Griffith~\cite{Lee2021} with our nodally coupled results.
Our results indicate that accuracy is still improved in the nodal scheme by using relatively coarser structural meshes up to $\mfac \approx 2$.
For structural meshes that are even coarser, substantial spurious flows begin to appear with nodal coupling. This is to be expected, since the support of the regularized delta function is fixed, and eventually there must be gaps if the structural mesh nodes are sufficiently far apart compared to the background Cartesian grid.

\begin{figure}
\centering
\setlength\tabcolsep{2pt}
\begin{tabular}{c c l c}
\includegraphics[width=.3\linewidth]{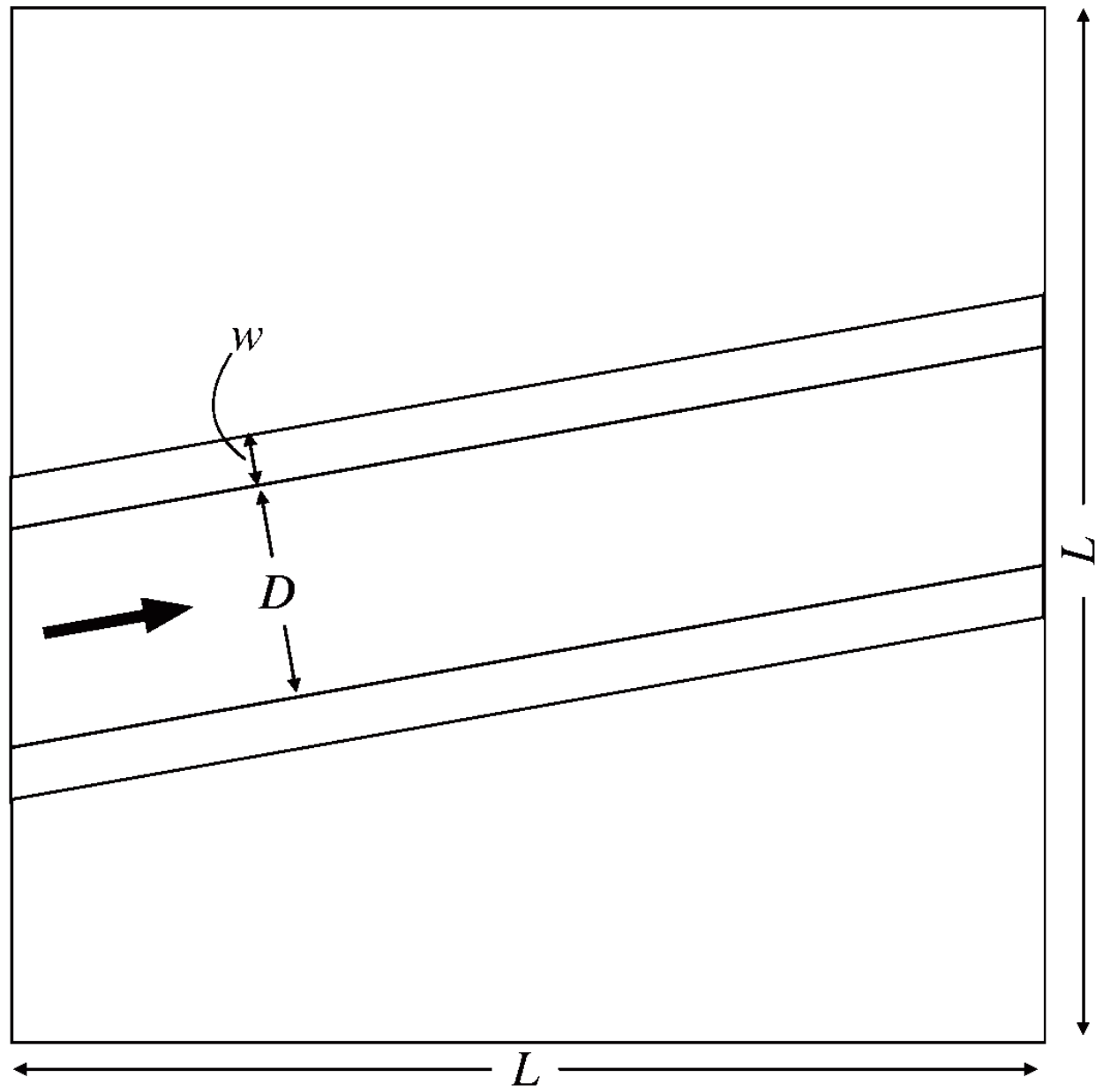} & $\quad\quad\quad$&
\rotatebox{90}{$\quad\quad\quad\quad\quad \text{error}$} &
\includegraphics[width=.35\linewidth]{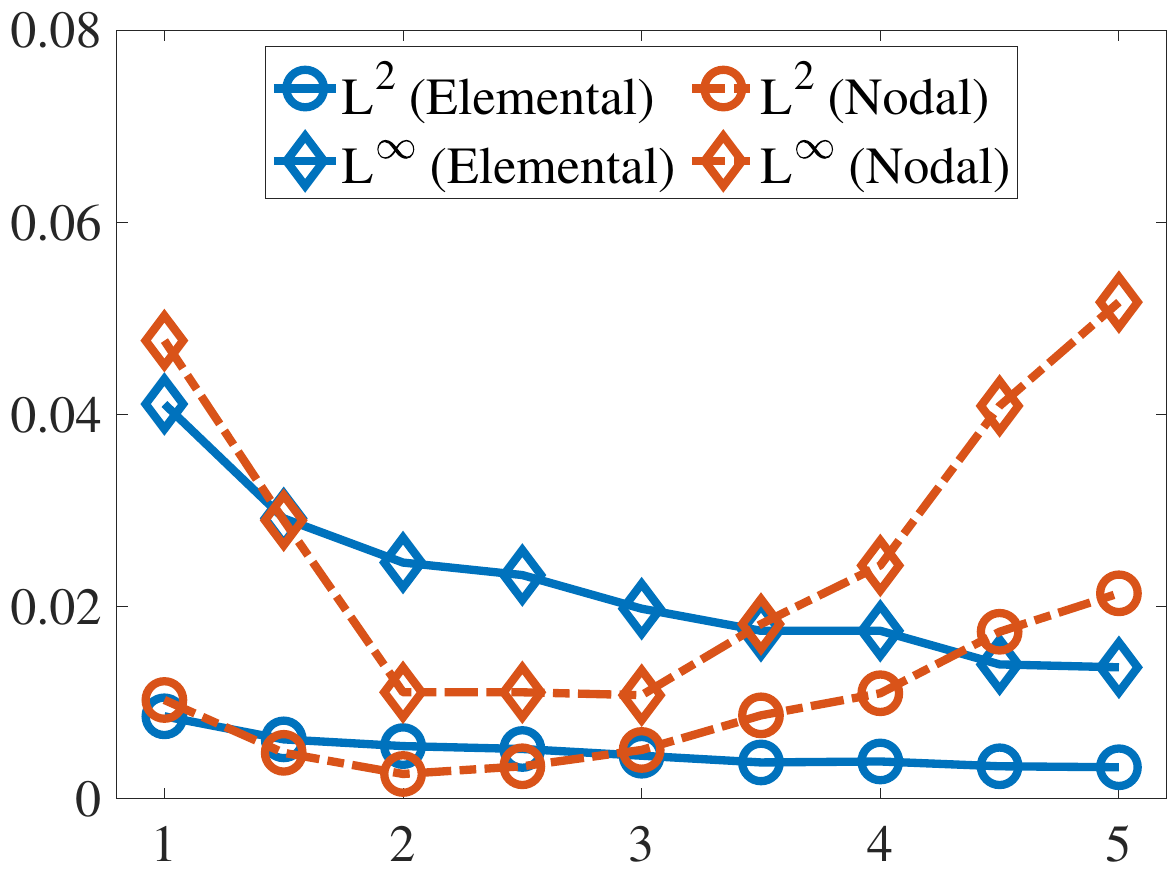}\\
& & & $\mfac$\\
(a) & & & (b)
\end{tabular}
\caption{(a) Schematic of two-dimensional flow through a slanted channel adopted from Lee and Griffith~\cite{Lee2021}.
(b) Comparison of the error norms in velocity for values of $\mfac = 1, 1.5, 2, 2.5, 3, 3.5, 4, 4.5,$ and $5$ at $N = 128$ between elemental and nodal coupling.
The results indicate that for nodal coupling, we are still able to use a
relatively coarser structural mesh ($\mfac\approx2$) but start to lose accuracy once the structure becomes too coarse ($\mfac \ge 3.5$) unlike the increased accuracy for elemental coupling.}
\label{fig:channel_error}
\end{figure}

\begin{table}
\centering
\begin{tabular}{| c | c | c | }
\hline
Density & $\rho$ & $1.0$ \\
\hline
Viscosity & $\mu$ & $0.01$ \\
\hline
Material model & - & rigid plate \\
\hline
Pressure gradient constant & $p_0$ & $0.2$ \\
\hline
Plate angle & $\theta$ & $\pi/18$ \\
\hline
\end{tabular}
\caption{Parameters for the channel flow benchmark (Section~\ref{subsubsec:channel_flow}).}
\label{tb:channel-flow}
\end{table}

\subsubsection{Two-dimensional Pressure-loaded Elastic Band}
\label{subsubsec:elastic_band}
In this benchmark, an elastic band deforms and ultimately reaches a steady-state configuration determined by the pressure difference across the band.
The steady-state fluid velocity should be zero.
The computational domain for this nondimensional benchmark is $2L \times L$, with $L = 1$.
The time step size is $\Delta t = 10^{-3} \euleriandx$ and is systematically reduced as needed to maintain stability throughout the simulation.
The Cartesian grid for this benchmark uses $N = 128$ cells in each coordinate direction on a single level.
The structure is an elastic beam with an incompressible neo-Hookean material model.
We impose fluid traction boundary conditions $\bbsigma (\xb, t)\nb(\xb) = -\boldsymbol{h}$ and $\bbsigma (\xb, t)\nb(\xb) = \boldsymbol{h}$ on the left and right boundaries of the computational domain, respectively, in which $\boldsymbol{h}=(5,0)^T$, and zero velocity is enforced along the top and bottom boundaries.
Figure~\ref{fig:elastic_band_error}a depicts a schematic of this benchmark and Table~\ref{tb:elastic_band} lists its parameters (see also Table~\ref{tbl:elastic_band-ib-points} for the number of elements and interaction points for the elastic band benchmark).
Here the elastic band was discretized with $\Ptwo$ elements.
Figure~\ref{fig:elastic_band_error}b shows comparison of the error norms in velocity for values of $\mfac = 0.5, 0.75, 1.0, 2.0,$ and $4.0$ between elementally coupled simulations obtained by Lee and Griffith~\cite{Lee2021} and nodal coupling.
Our results indicate, for both methods, that there is loss of accuracy with relatively coarser meshes when the structure is loaded by fluid pressure.
In particular, the nodally coupled method shows a significant increase in the errors when $\mfac > 1$.
For $\mfac \le 1$, the errors are comparable between different cases.
This suggests that for pressure-loaded cases, the structural mesh needs to be as fine or finer (i.e., $\mfac \leq 1$) than the Cartesian grid for both nodal and elemental coupling schemes.

\begin{figure}
\centering
\setlength\tabcolsep{2pt}
\begin{tabular}{c c l c}
\includegraphics[width=.55\linewidth]{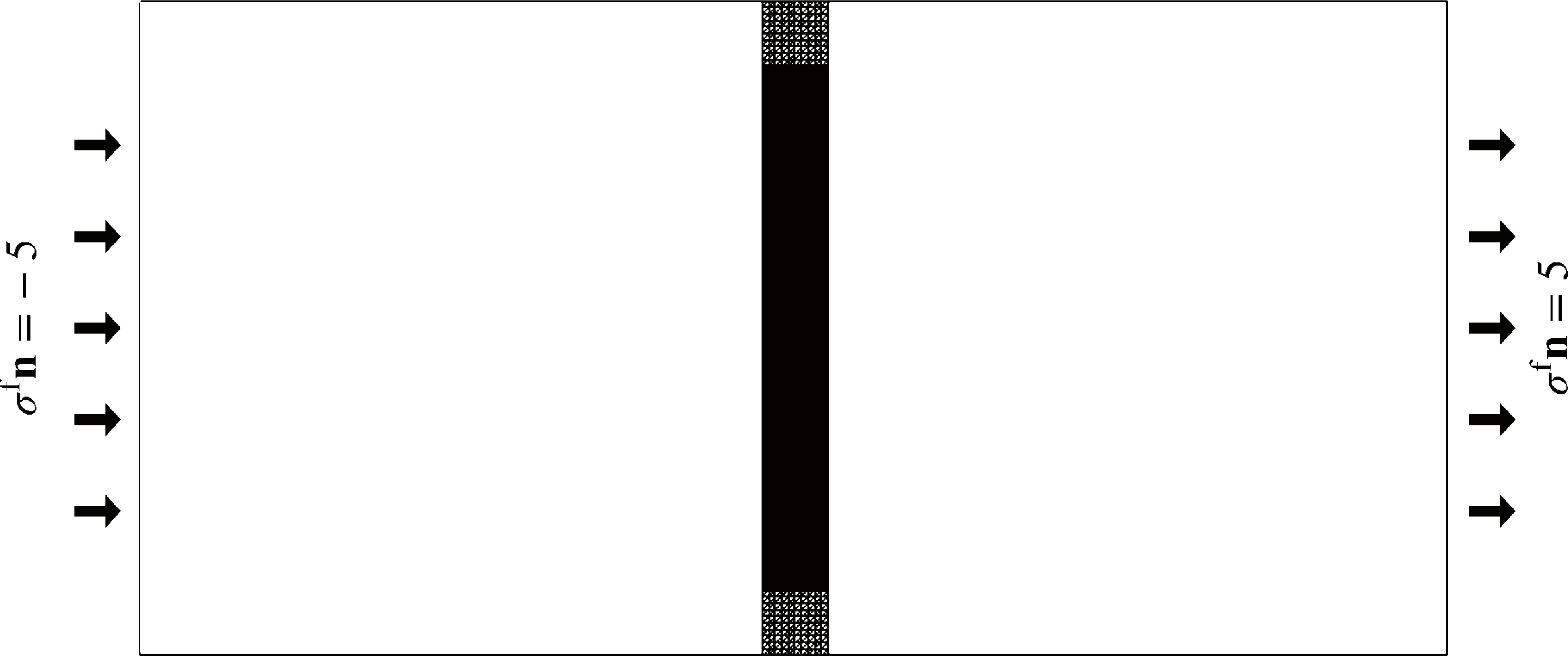} & $\quad\quad$&
\rotatebox{90}{$\quad\quad\quad\quad \log{(\text{error})}$} &
\includegraphics[width=.35\linewidth]{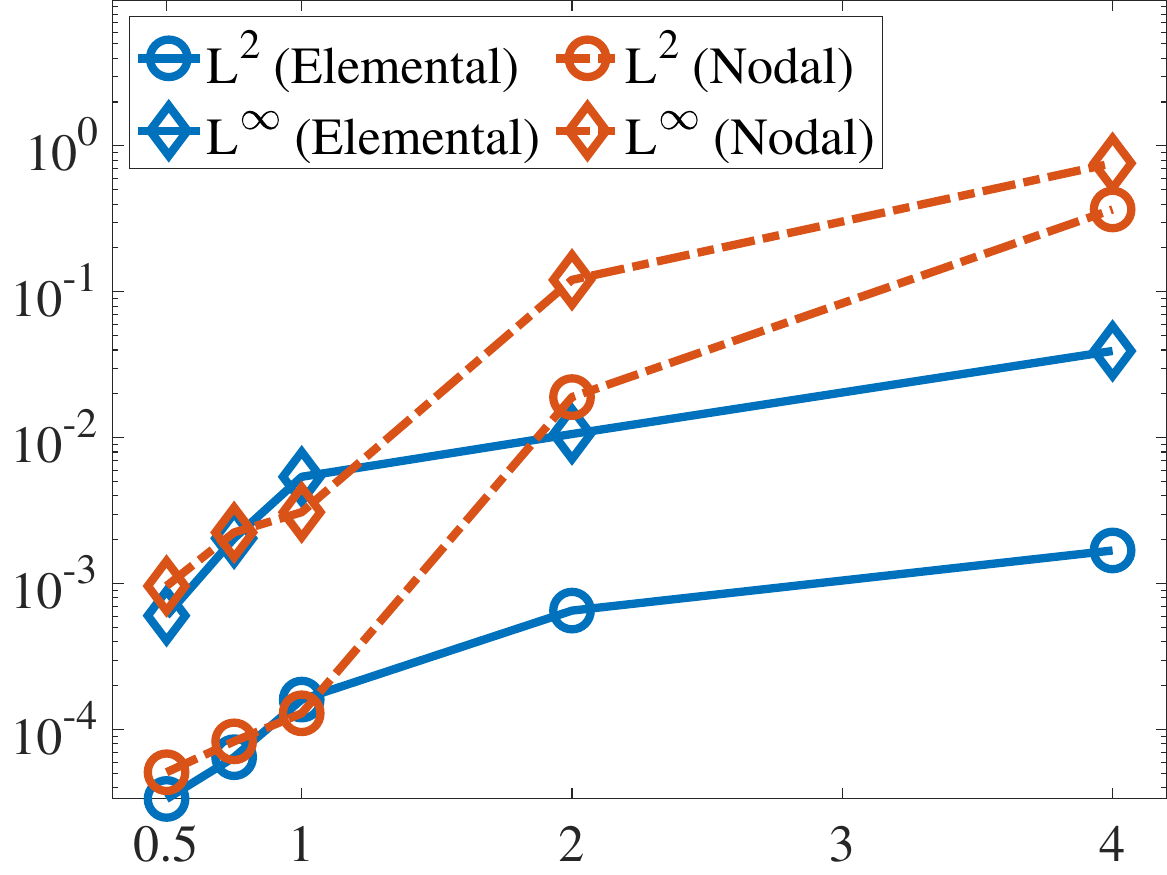}\\
& & & $\mfac$\\
(a) & & & (b)
\end{tabular}
\caption{(a) Schematic of two-dimensional pressure-loaded elastic band adopted from Lee and Griffith~\cite{Lee2021}.
The elastic band experiences a pressure load induced by the pressure gradient between the left and right boundaries of the computational domain.
(b) Comparison of the error norms in velocity for values of $\mfac = 0.5, 0.75, 1, 2,$ and $4$ at $N = 128$ between elemental and nodal coupling.
The results indicate that we obtain accuracy comparable to the elemental scheme for nodal coupling up to $\mfac = 1$, but we see significant loss of accuracy when we use a relatively coarser structural mesh ($\mfac \ge 1$).}
\label{fig:elastic_band_error}
\end{figure}

\begin{table}
\centering
\begin{tabular}{| c | c | c | }
\hline
Density & $\rho$ & $1.0$ \\
\hline
Viscosity & $\mu$ & $0.01$ \\
\hline
Material model & - & incompressible neo-Hookean \\
\hline
\end{tabular}
\caption{Parameters for the pressure-loaded elastic band benchmark (Section~\ref{subsubsec:elastic_band}).}
\label{tb:elastic_band}
\end{table}

%% nodal
%% MFAC 0.50: 7722 points
%% MFAC 0.75: 4376 points
%% MFAC 1.00: 3186 points
%% MFAC 2.00: 1962 points
%% MFAC 4.00: 1586 points
%%
%% elemental
%% MFAC 0.50: 25844 points
%% MFAC 0.75: 31066 points
%% MFAC 1.00: 29092 points
%% MFAC 2.00: 32804 points
%% MFAC 4.00: 28732 points
%%
%% number of elements
%% MFAC 0.50: 3692
%% MFAC 0.75: 2062
%% MFAC 1.00: 1488
%% MFAC 2.00: 908
%% MFAC 4.00: 736

\begin{table}
\centering
\begin{tabular}{| r | r | r | r | }
\hline
$\mfac$ & total elements & nodal interaction points & elemental interaction points\\
\hline
0.5     & 3692     & 7722                     & 25844 \\
\hline
0.75    & 2062     & 4376                     & 31066 \\
\hline
1.0     & 1488     & 3186                     & 29092 \\
\hline
2.0     & 908      & 1962                     & 32804 \\
\hline
4.0     & 736      & 1586                     & 28732 \\
\hline
\end{tabular}
\caption{Number of elements and interaction points for the elastic band benchmark (Section~\ref{subsubsec:elastic_band}).
For elemental coupling, using $\mathcal{A}_q$ ensures that the number of interaction points does not decrease as $\mfac$ increases.
The number of elements is not inversely proportional to $\mfac$ since the band mesh must be at least one element wide.
}
\label{tbl:elastic_band-ib-points}
\end{table}

\subsection{Static Benchmarks}
For quasi-static benchmarks, we immerse these structures in an incompressible fluid, apply loading forces, and allow the models to reach a steady state equilibrium.
This enables us to make direct comparisons to the results of an FE method for incompressible elastostatics~\cite{Chiumenti, Masud2013}.
These benchmarks use zero velocity boundary conditions on the computational domain, which ensures that the fluid-structure system reaches a quiescent steady state.

\subsubsection{Compression Test}
\label{subsec:compression-test}
This benchmark is a plane strain quasi-static problem involving a rectangular block with a downward traction applied in the center of the top side of the mesh and zero vertical displacement applied on the bottom boundary.
It was used by Reese \etal~\cite{Reese1999} to test a stabilization technique for low order finite elements.
The computational domain for this benchmark is $[0,L]\times[0,L]$, with $L = 40\ \text{cm}$.
The Cartesian grid uses a single refinement level with $N = \mathrm{ceil}\left(2 M \efac \mfac\right)$, in which $M$ is the number of elements per the longest edge in the Lagrangian mesh and $2$ is the ratio of $L$ to that longest edge.
The time step size is $\Delta t = 0.001 \euleriandx \ \text{s}$.
Figure~\ref{fig:comp-mesh} depicts the loading configuration and dimensions of the structure.
The structure uses a modified Neo-Hookean material model.
Table~\ref{tb:cb-param} lists the physical and numerical parameters used in this benchmark.

Zero horizontal displacement is also imposed along the top side.
All other boundary conditions are zero traction.
The primary quantity of interest is the $y$-displacement of the point at the center of the top face.
The penalty parameter to fix the bottom in place is $\kappa_{\text{S}} = 2.5 \cdot \frac{\Delta x}{\Delta t} \frac{\text{dyn}}{\text{cm}^3}$.
We gradually apply the traction to the solid boundary linearly in time so that the full load is applied at $T_{\text{l}} = 40.0$ s, and we wait until time $T_\text{f} = 100.0$ s for the structure to reach equilibrium.
The numbers of solid degrees of freedom (DoFs) range from $m = 15$ to $m = 4753$ for the FE ($\Pone$/$\Pone$) results and all the IFED results.
We explore the effect of using mesh factors of $\mfac = 0.5,$ $0.75,$ and $1.0$.

Figures~\ref{fig:cb_def} and \ref{fig:cb_disp_bs3} show representative steady state deformations and results for the displacement of the point of interest.
All cases appear to converge to the FE benchmark solution under grid refinement.
Both the nodally and elementally coupled methods perform well for the presented range of $\mfac$ values.
However, for coarse cases, the nodally coupled case over-shoots or under-shoots the FE numerical solution for larger values of $\mfac$, indicating that coarse discretizations may perform better with smaller relative grid spacings.
Figure~\ref{fig:cb_def} shows the deformations for both the elemental and nodal cases with $\Qone$ elements.
The displacements for both coupling strategies are qualitatively similar.

\begin{table}
\centering
\begin{tabular}{| c | c | c | c | }
\hline
Density & $\rho$ & $1.0$ & $\frac{\text{g}}{\text{cm}^3}$\\
\hline
Viscosity & $\mu$ & $0.16$ & $\frac{\text{dyn} \cdot \text{s}}{\text{cm}^2}$ \\
\hline
Material model & - & modified neo-Hookean & - \\
\hline
Shear modulus & $G$ & $80.194$ & $\frac{\text{dyn}}{\text{cm}^2}$  \\
\hline
Numerical bulk modulus & $\kappas$ & $374.239$ &
$\frac{\text{dyn}}{\text{cm}^2}$\\
\hline
Final time & $T_\text{f}$ & $100.0$  & s\\
\hline
Load time & $T_{\text{l}}$ & $40.0$ & s \\
\hline
\end{tabular}
\caption{Parameters for the compressed block benchmark (Section~\ref{subsec:compression-test}).}
\label{tb:cb-param}
\end{table}

\begin{figure}
\centering
\includegraphics[width=.7\linewidth, trim={80 0 70 0}, clip]{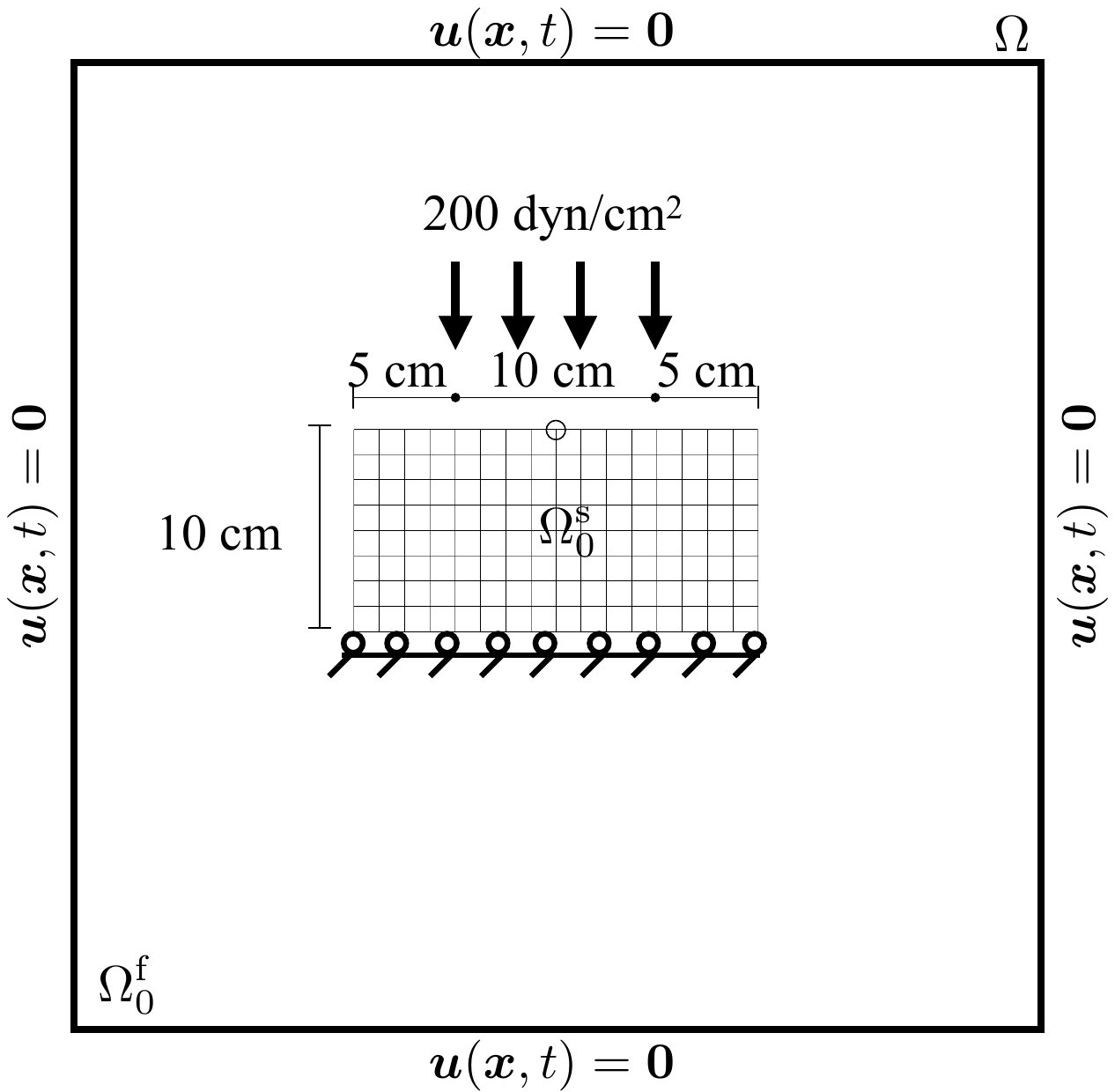}
\caption{
Specifications of the compressed block benchmark (Section~\ref{subsec:compression-test}).
The quantity of interest is the $y$-displacement measured at the encircled point.
The structure, shown here in its initial configuration and denoted by $\soliddomO$, is immersed in a fluid denoted by $\fluiddomO$.
The entire computational domain is $\Omega = \fluiddom \cup \soliddom$.
Zero fluid velocity is enforced on $\partial\Omega$.
}
\label{fig:comp-mesh}
\end{figure}

\begin{figure}
\begin{tabular}{l c c c}
&$t = 0\ \text{s}$ & $t = 20\ \text{s}$ & $t = 100\ \text{s}$  \\
\rotatebox{90}{$\quad$ Elemental}&
\includegraphics[width=.3\linewidth]{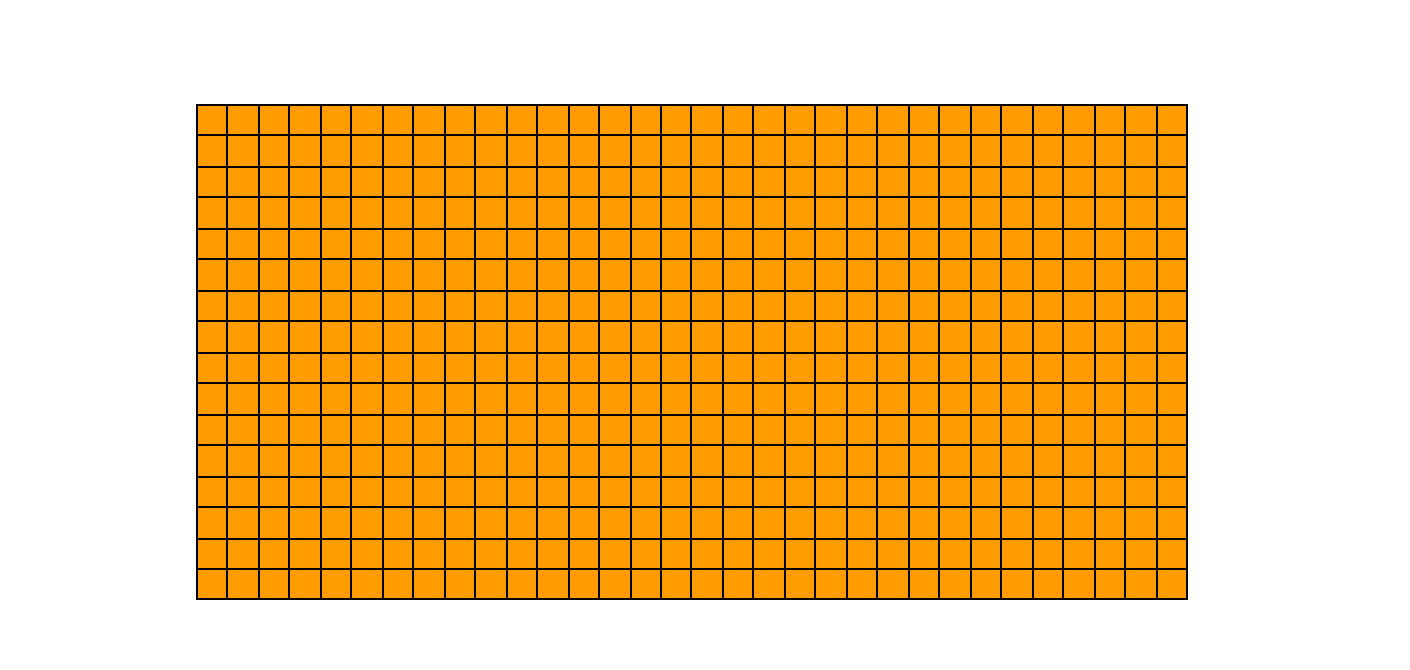}&
\includegraphics[width=.3\linewidth]{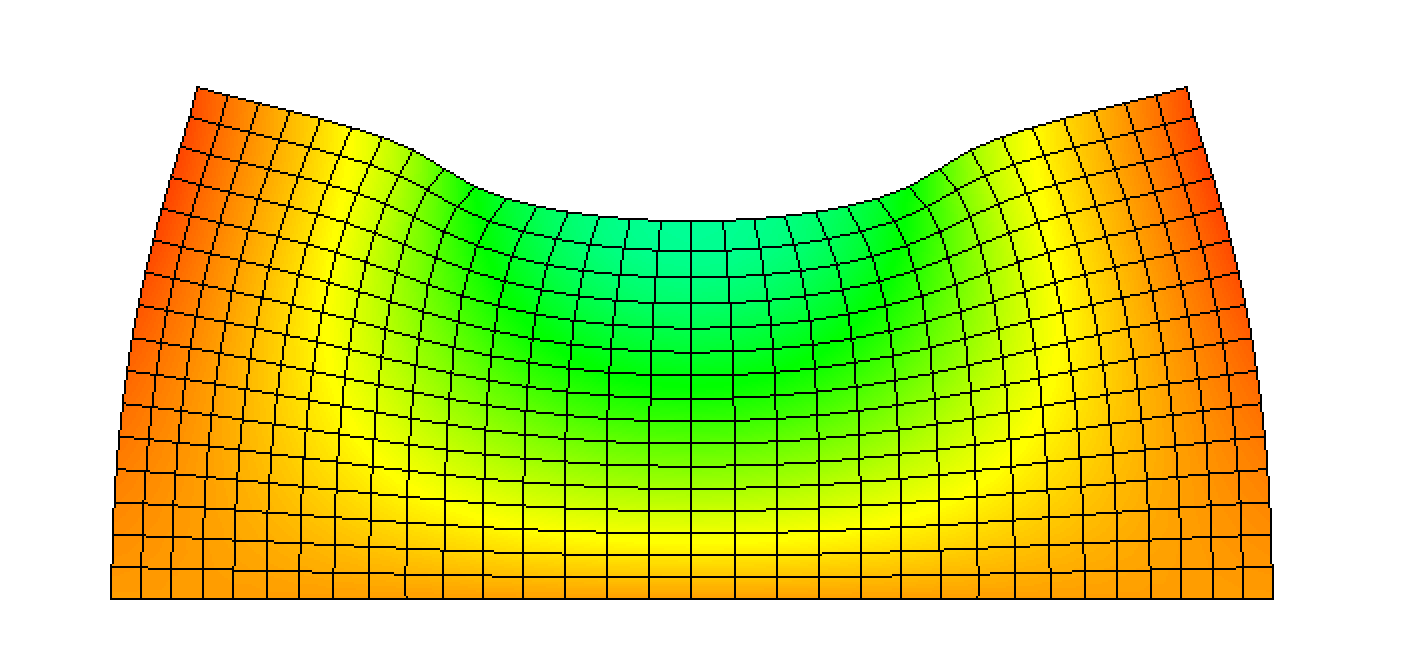}&
\includegraphics[width=.3\linewidth]{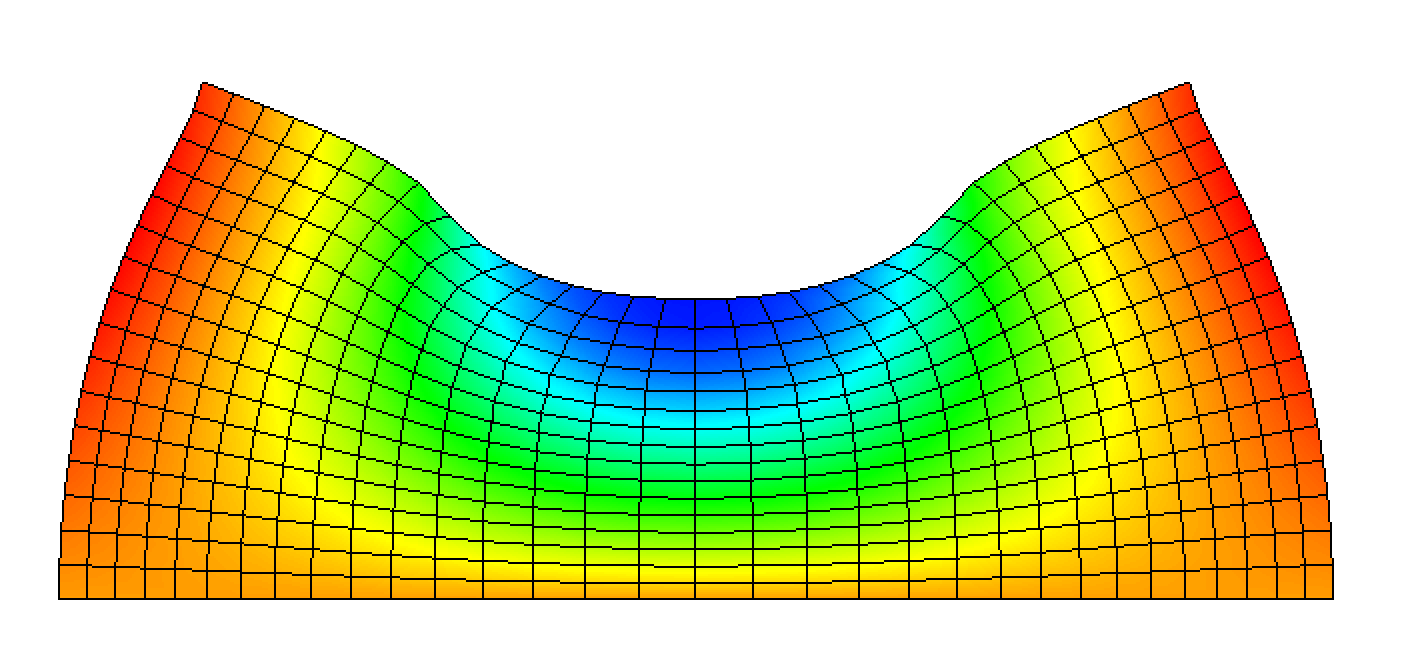}\\

\rotatebox{90}{$\quad\;$ Nodal}&
\includegraphics[width=.3\linewidth]{Figures/cb/cb_nodal_t=0.png}&
\includegraphics[width=.3\linewidth]{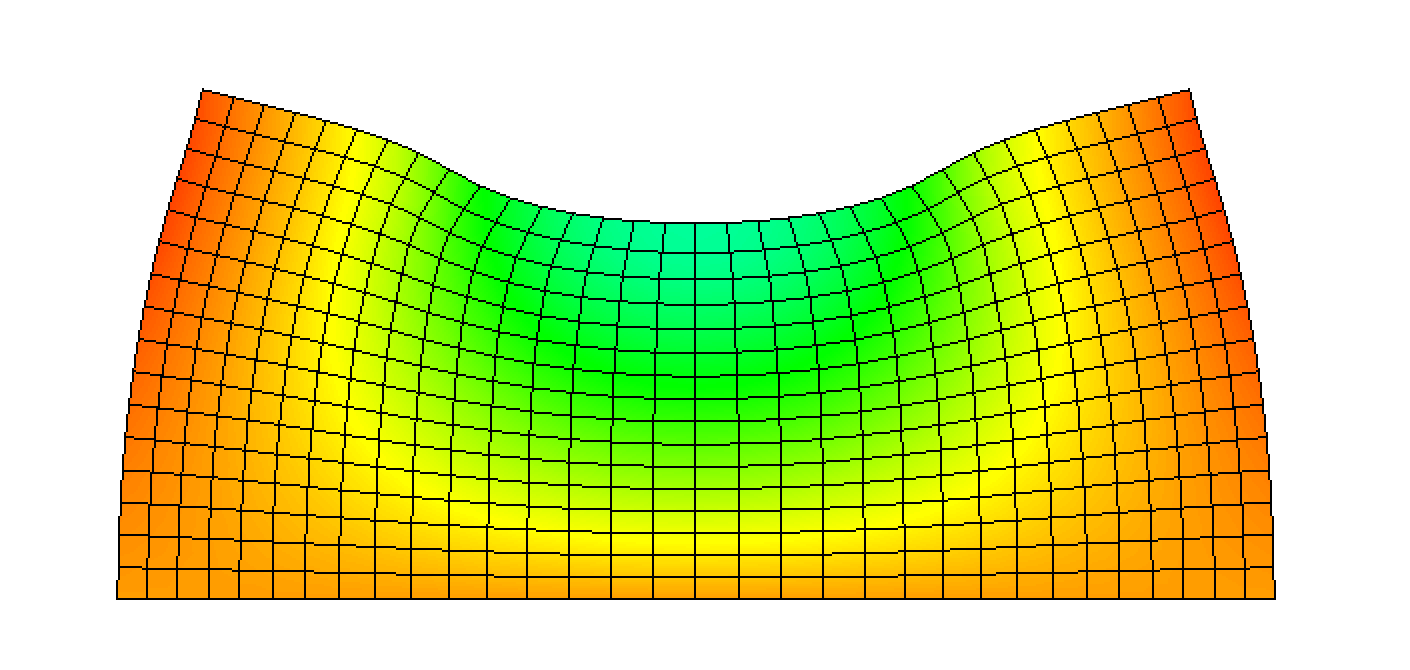}&
\includegraphics[width=.3\linewidth]{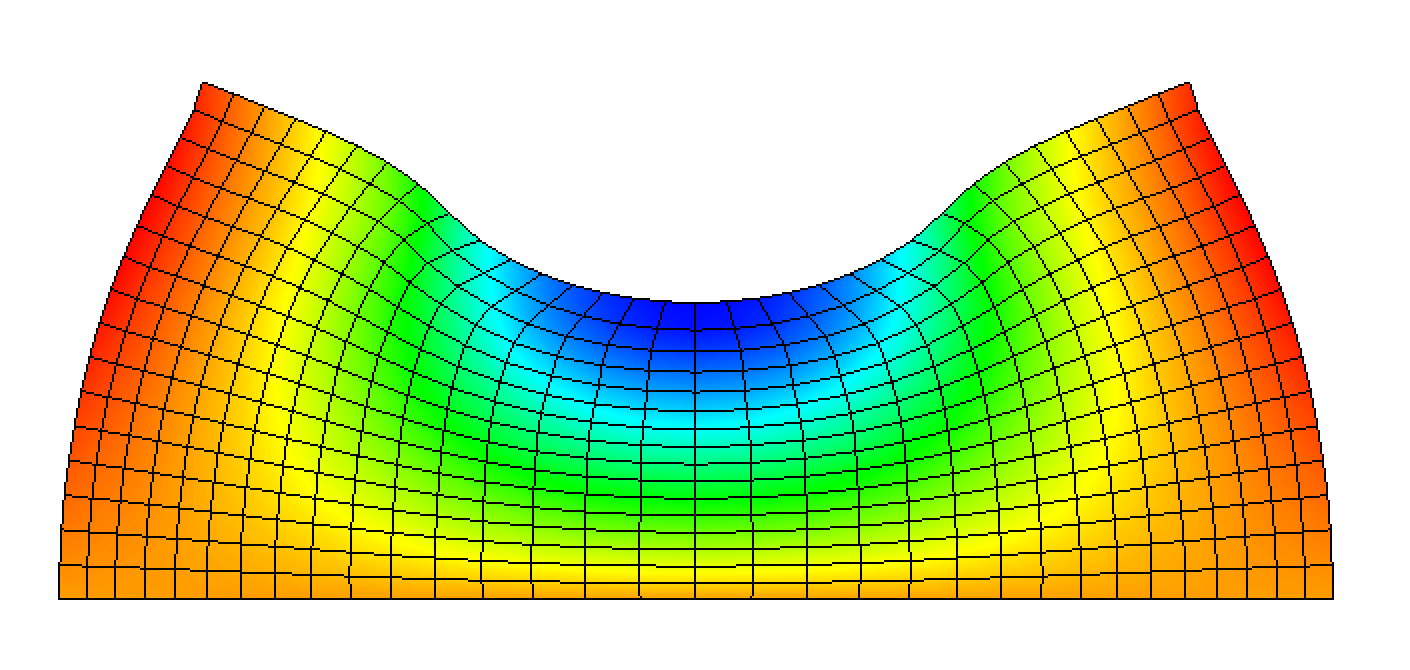}\\
& {} & $d^{\text{s}}_2$ &\\
& {} & \includegraphics[width=.3\linewidth]{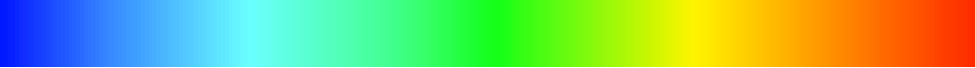}&\\
\end{tabular}
\begin{center}
  % This spacing is hard-coded so that the labels show up on both sides of the
  % colorbar and don't contribute to the calculation of column widths in the
  % above tabular environment
\hspace{0.060\linewidth}$-4.0 \ \text{cm}$ \hspace{.215\linewidth} $0.75\ \text{cm}$
\end{center}
\caption{Deformations and $y$-displacements of the compressed block benchmark for both elemental and nodal coupling at different points in time.
  Time values are start of the simulation, $0.5T_\text{l}$, and $T_\text{f}$.
  In both cases, the structure is discretized with $\Qone$ elements.
  The nodal case uses $\mfac = 0.5$ and the elemental case uses $\mfac = 1.0$.}
\label{fig:cb_def}
\end{figure}

\begin{figure}
\begin{tabular}{l c c c c}
& $\Pone$ & $\Qone$ & $\Ptwo$ & $\Qtwo$ \\

\rotatebox{90}{$\quad\;\;\mfac = 0.5$}&
\includegraphics[width=.22\linewidth, trim={0 150 50 100}, clip]{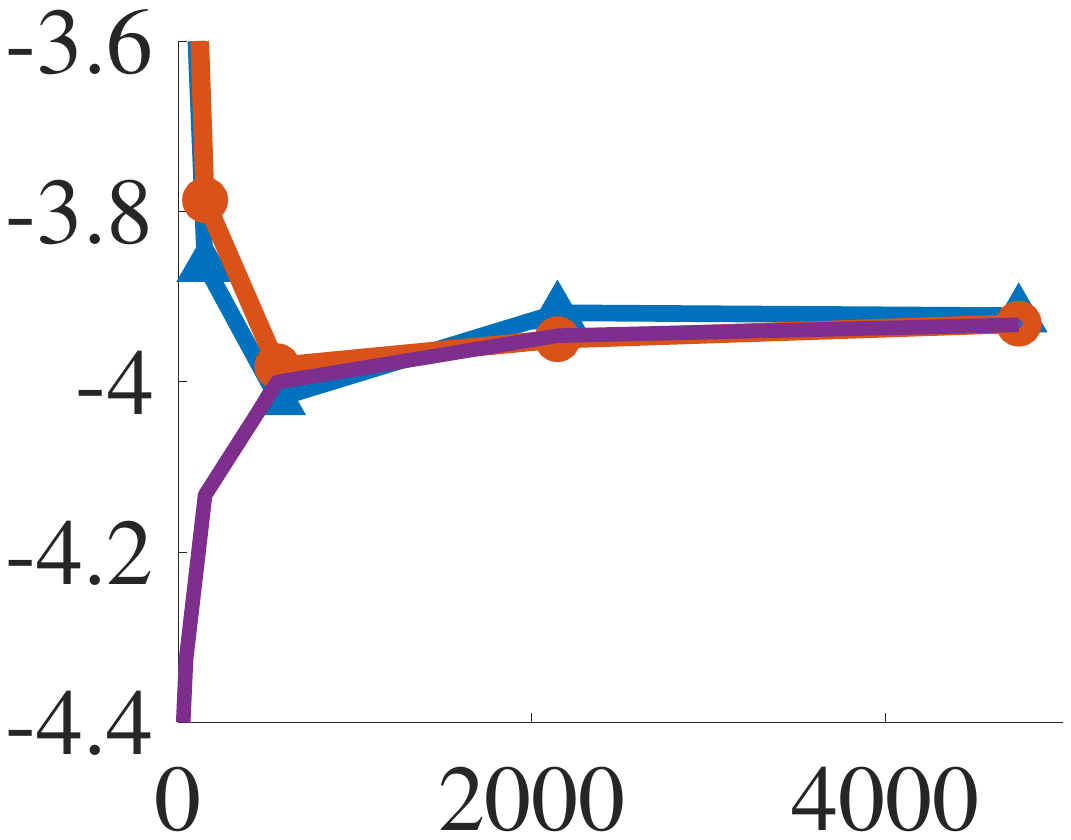}&
\includegraphics[width=.22\linewidth, trim={0 150 50 100}, clip]{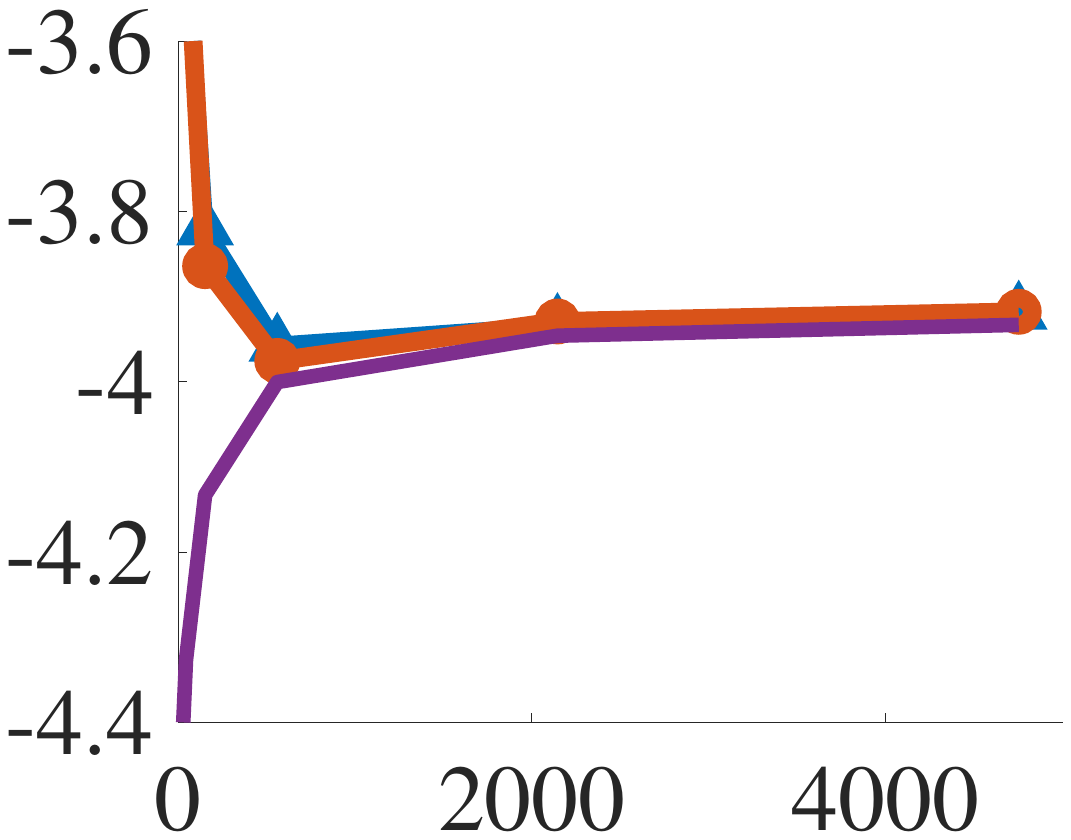}&
\includegraphics[width=.22\linewidth, trim={0 150 50 100}, clip]{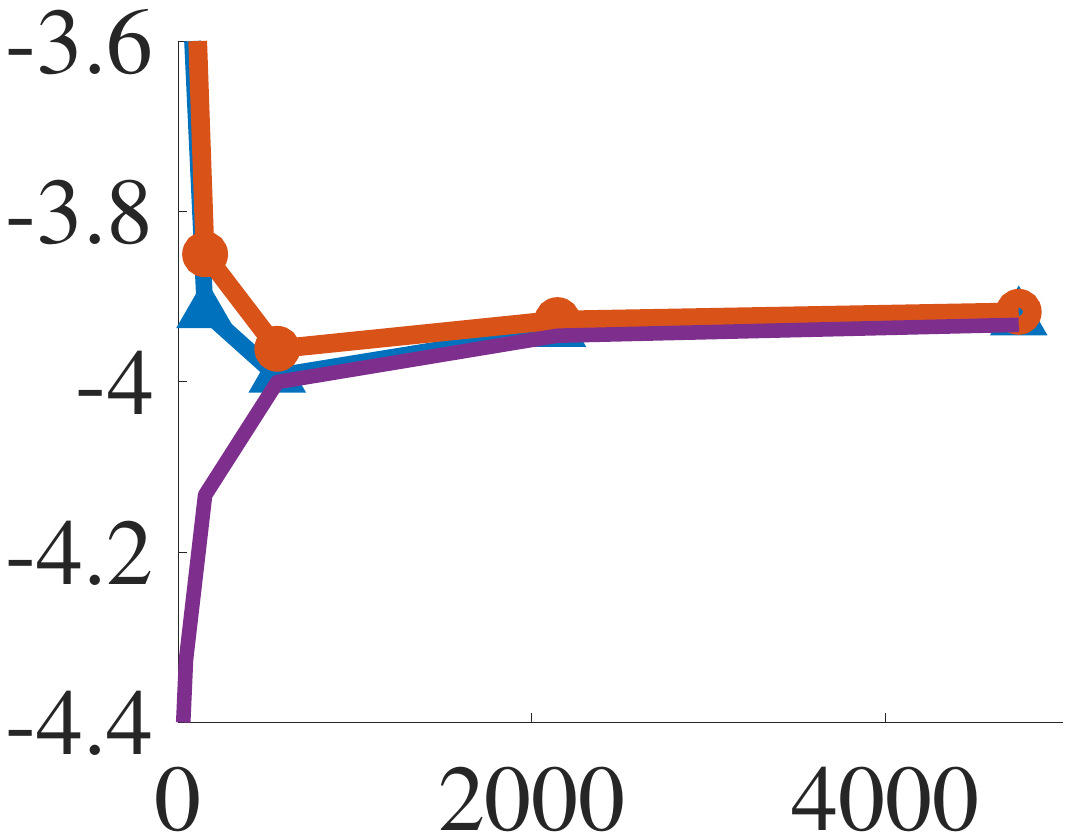}&
\includegraphics[width=.22\linewidth, trim={0 150 50 100}, clip]{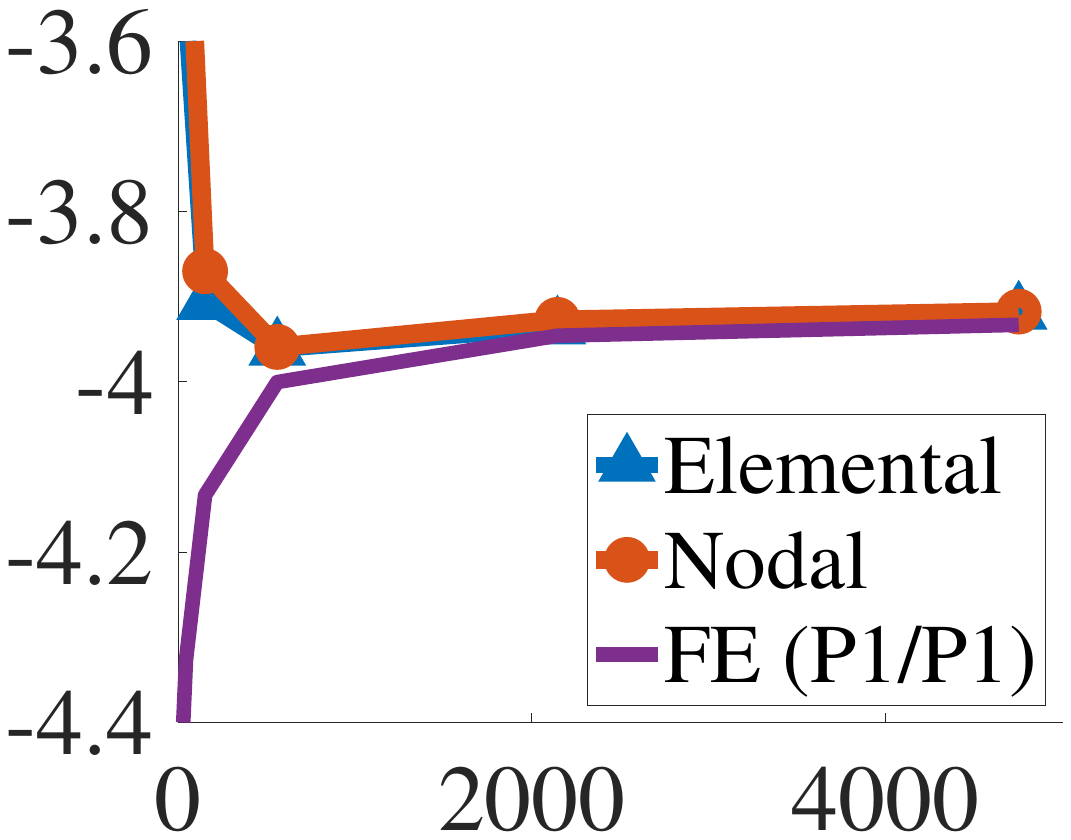}\\

\rotatebox{90}{$\quad\;\;\mfac = 0.75$}&
\includegraphics[width=.22\linewidth, trim={0 150 50 100}, clip]{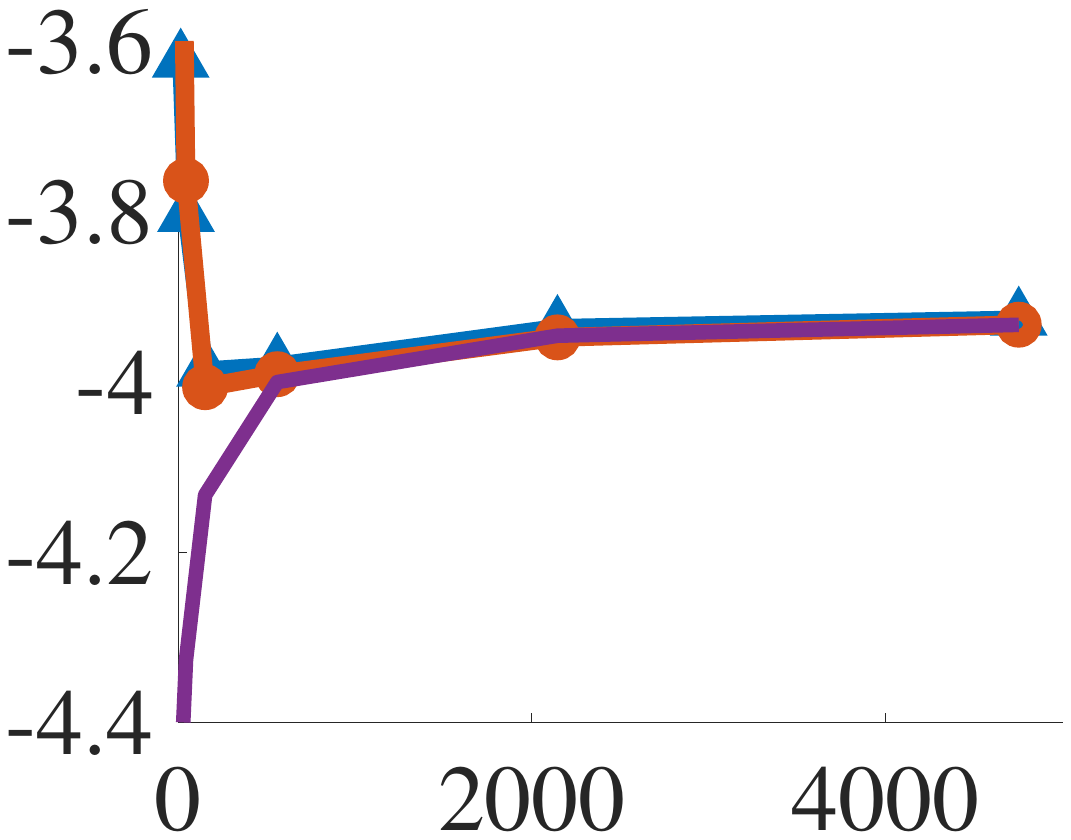}&
\includegraphics[width=.22\linewidth, trim={0 150 50 100}, clip]{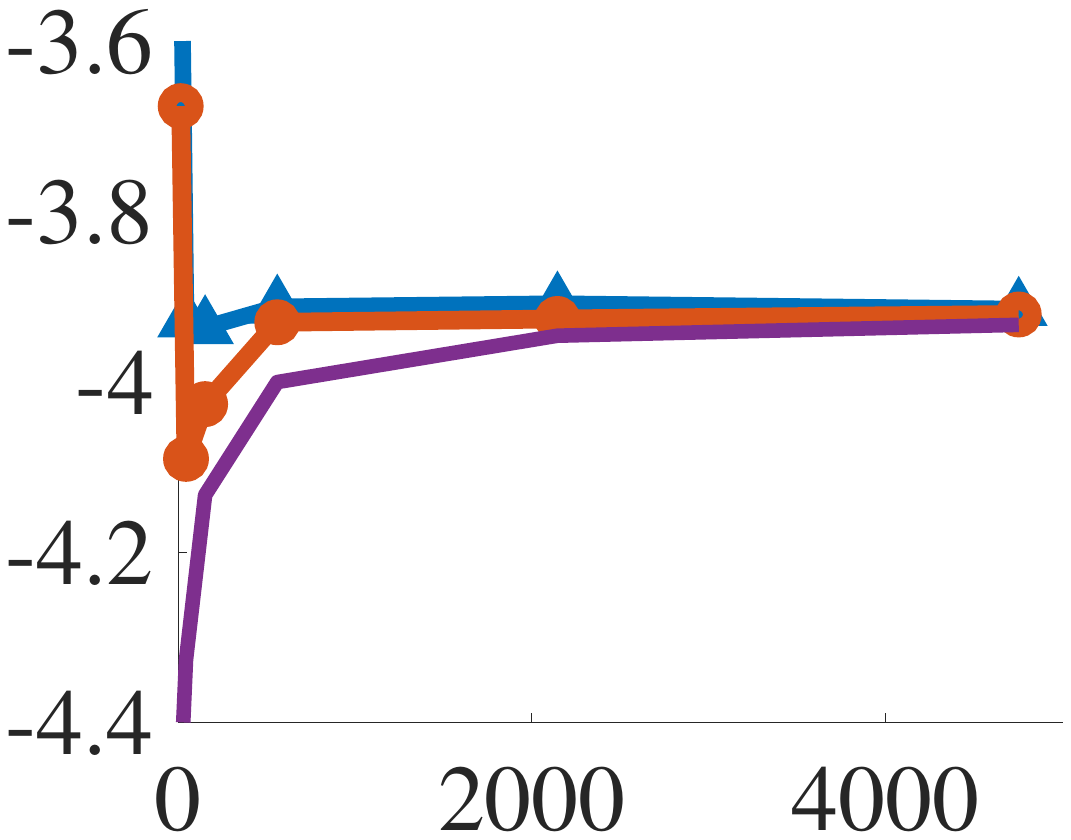}&
\includegraphics[width=.22\linewidth, trim={0 150 50 100}, clip]{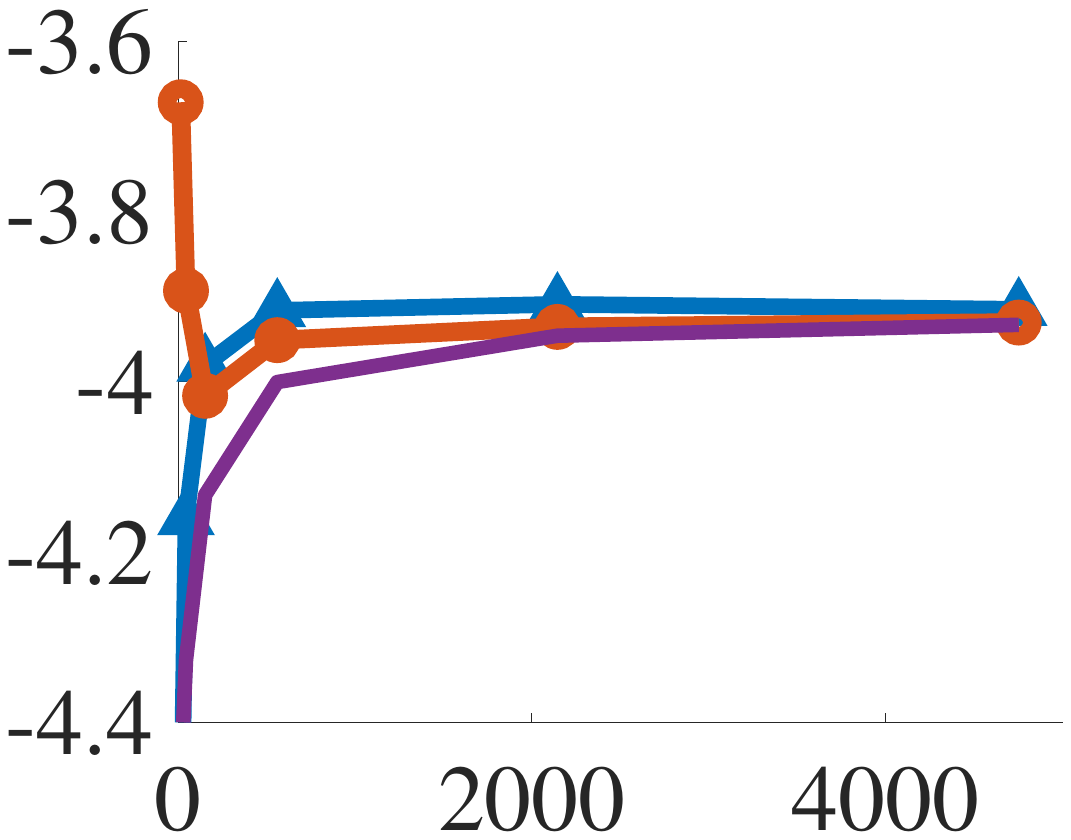}&
\includegraphics[width=.22\linewidth, trim={0 150 50 100}, clip]{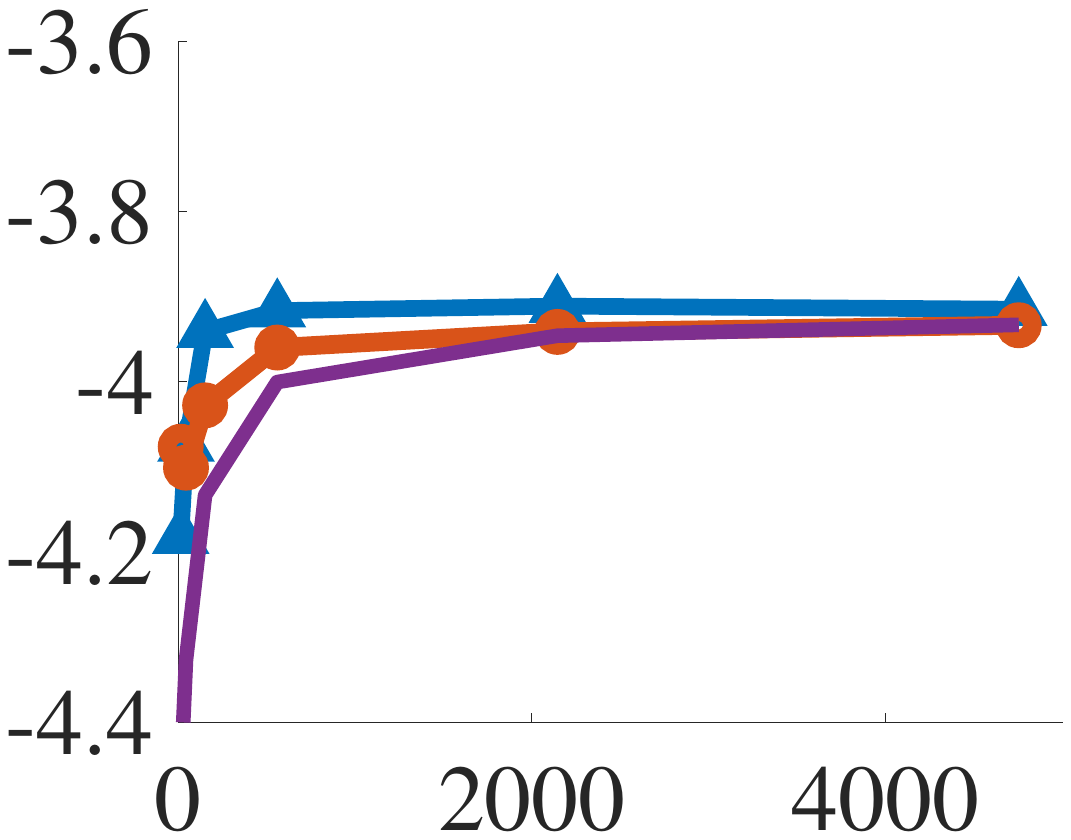}\\

\rotatebox{90}{$\quad\;\;\mfac = 1.0$}&
\includegraphics[width=.22\linewidth, trim={0 150 50 100}, clip]{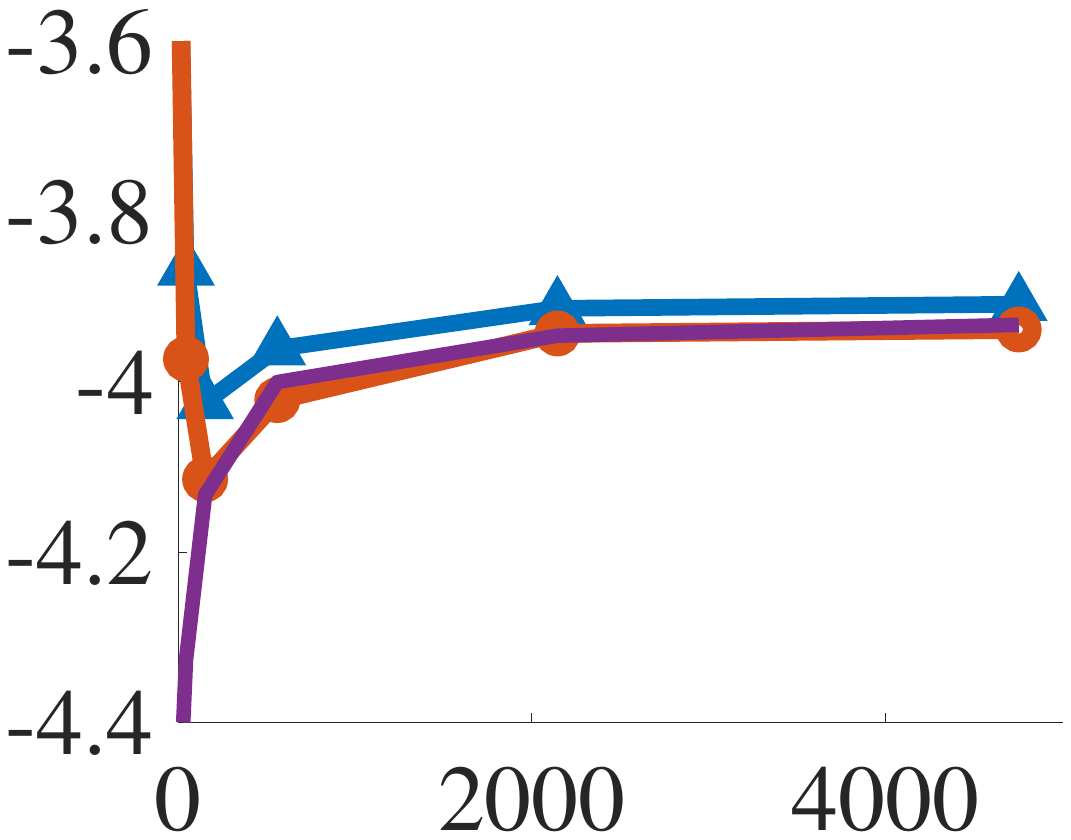}&
\includegraphics[width=.22\linewidth, trim={0 150 50 100}, clip]{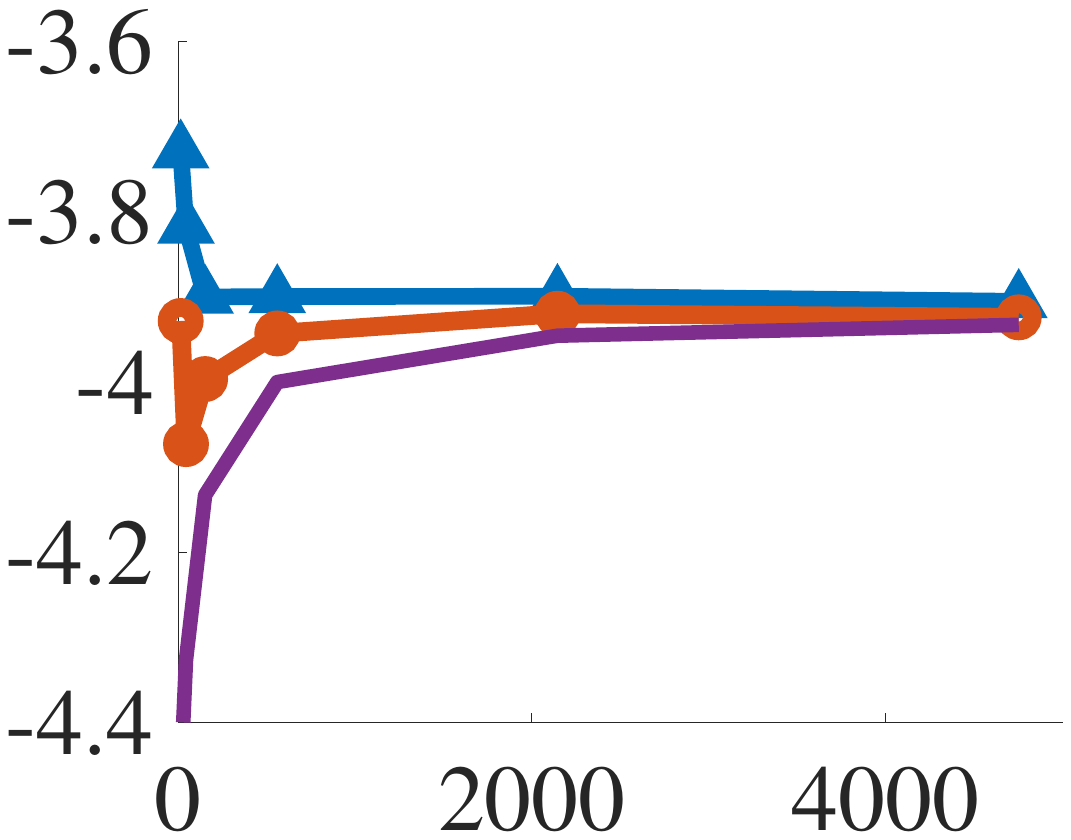}&
\includegraphics[width=.22\linewidth, trim={0 150 50 100}, clip]{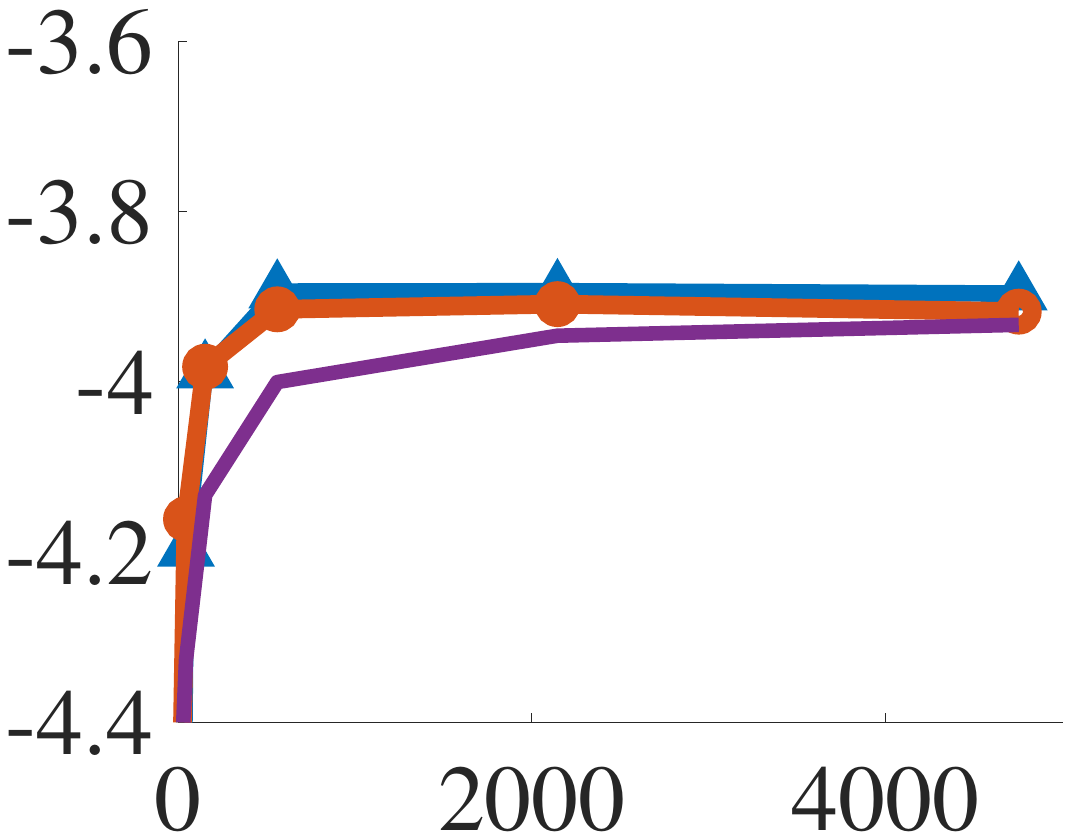}&
\includegraphics[width=.22\linewidth, trim={0 150 50 100}, clip]{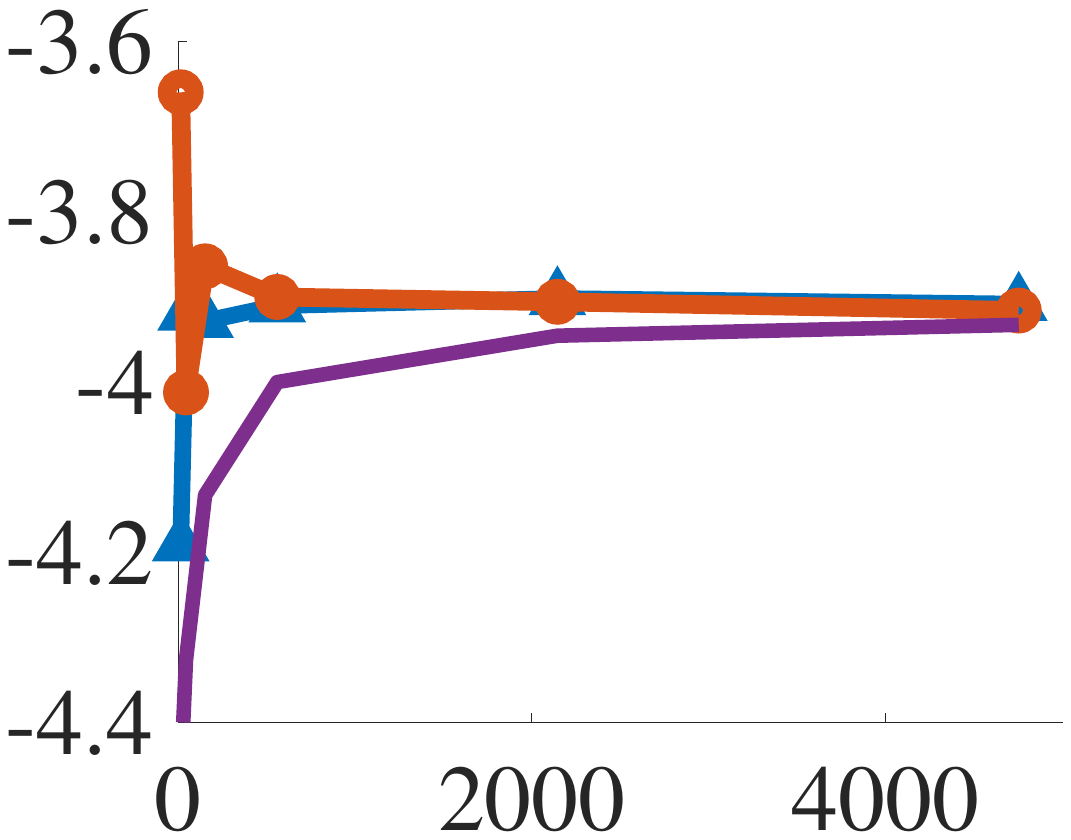}\\

\end{tabular}
\caption{The $y$-displacement ($y$-axis) in cm vs solid DoF count ($x$-axis) of the compressed block using different coupling strategies and $\mfac$ values.}
\label{fig:cb_disp_bs3}
\end{figure}

\subsubsection{Cook's Membrane}
\label{subsec:cooks-membrane}
The Cook's membrane benchmark is a classical plane strain problem involving a swept and tapered quadrilateral.
This benchmark was first proposed by Cook \etal~\cite{RDCook1974} and is commonly used to test numerical methods for incompressible elasticity.
The computational domain for this benchmark is $[0,L]\times[0,L]$, with $L = 10\ \text{cm}$.
The Cartesian grid uses $N = \mathrm{ceil}\left(M \efac \mfac \cdot \dfrac{10}{6.5}\right)$ cells in each coordinate direction, in which $M$ is the number of elements per edge in the Lagrangian mesh, $10$ is the length of the computational domain, and $6.5$ is the longest side of the structure.
The time step size is $\Delta t = 0.001 \euleriandx \ \text{s}$.
Figure~\ref{cooks} depicts the dimensions of the solid domain and the overall problem specification.
An upward loading traction is applied to the right side, and the left side is fixed in place; see Figure~$\ref{cooks}$.
Figure~\ref{cooks} shows the displacement of the uppermost right-hand corner, which is the primary quantity of interest in this benchmark.
See Table~\ref{tb:cm-param} for the benchmark's parameters.
The penalty parameter to fix the left side in place is $\kappa_{\text{S}} = 0.125 \cdot \frac{\Delta x}{\Delta t}\frac{\text{dyn}}{\text{cm}^3}$.
We gradually apply the traction to the solid boundary linearly in time so that the full load is applied at $T_{\text{l}} = 20.0$ s, and we wait until time $T_\text{f} = 50.0$ s for the structure to reach equilibrium.
All other structural boundaries have stress-free boundary conditions applied.
The number of solid DoFs ranges from $m = 25$ to $m = 4225$.
We explore the effect of using mesh factors of $\mfac = 1.0,$ $1.5,$ and $2.0$.

Figures~\ref{fig:cm_def} and \ref{fig:cm_disp_bs3} depict representative deformations and the displacement of a selected point.

As interaction points are placed further apart with nodal coupling, one might expect that case will require a relatively finer structural mesh than those used by Vadala-Roth \etal~\cite{Vadala-Roth2020} (that study used $\mfac = 2.0$ for this benchmark).
However, much like the results for the compressed block, both the elementally and nodally coupled methods converge, although the $\mfac = 2.0$ nodally coupled cases seem to converge more slowly when using elements higher than $\Pone$.
Here we have used $\mfac \geq 1$ and, notably, have shown that we may still achieve converged results with a structural mesh that is relatively coarser than the background Cartesian grid, despite the interaction points being further spaced apart.
Finally, Figure~\ref{fig:cm_def} compares the deformations, clearly indicating qualitative agreement between both cases.

\begin{table}
\centering
\begin{tabular}{| c | c | c | c | }
\hline
Density & $\rho$ & $1.0$ & $\frac{\text{g}}{\text{cm}^3}$\\
\hline
Viscosity & $\mu$ & $0.16$ & $\frac{\text{dyn} \cdot \text{s}}{\text{cm}^2}$ \\
\hline
Material model & - & modified neo-Hookean & - \\
\hline
Shear modulus & $G$ & $83.333$ & $\frac{\text{dyn}}{\text{cm}^2}$  \\
\hline
Numerical bulk modulus & $\kappas$ & $388.889$
& $\frac{\text{dyn}}{\text{cm}^2}$\\
\hline
Final time & $T_\text{f}$ & $50.0$  & s\\
\hline
Load time & $T_{\text{l}}$ & $20.0$ & s \\
\hline

\end{tabular}
\caption{Parameters for the Cook's membrane benchmark (Section~\ref{subsec:cooks-membrane}).}
\label{tb:cm-param}
\end{table}

\begin{figure}
\centering
\includegraphics[width=.7\linewidth, trim={60 0 70 0}, clip]{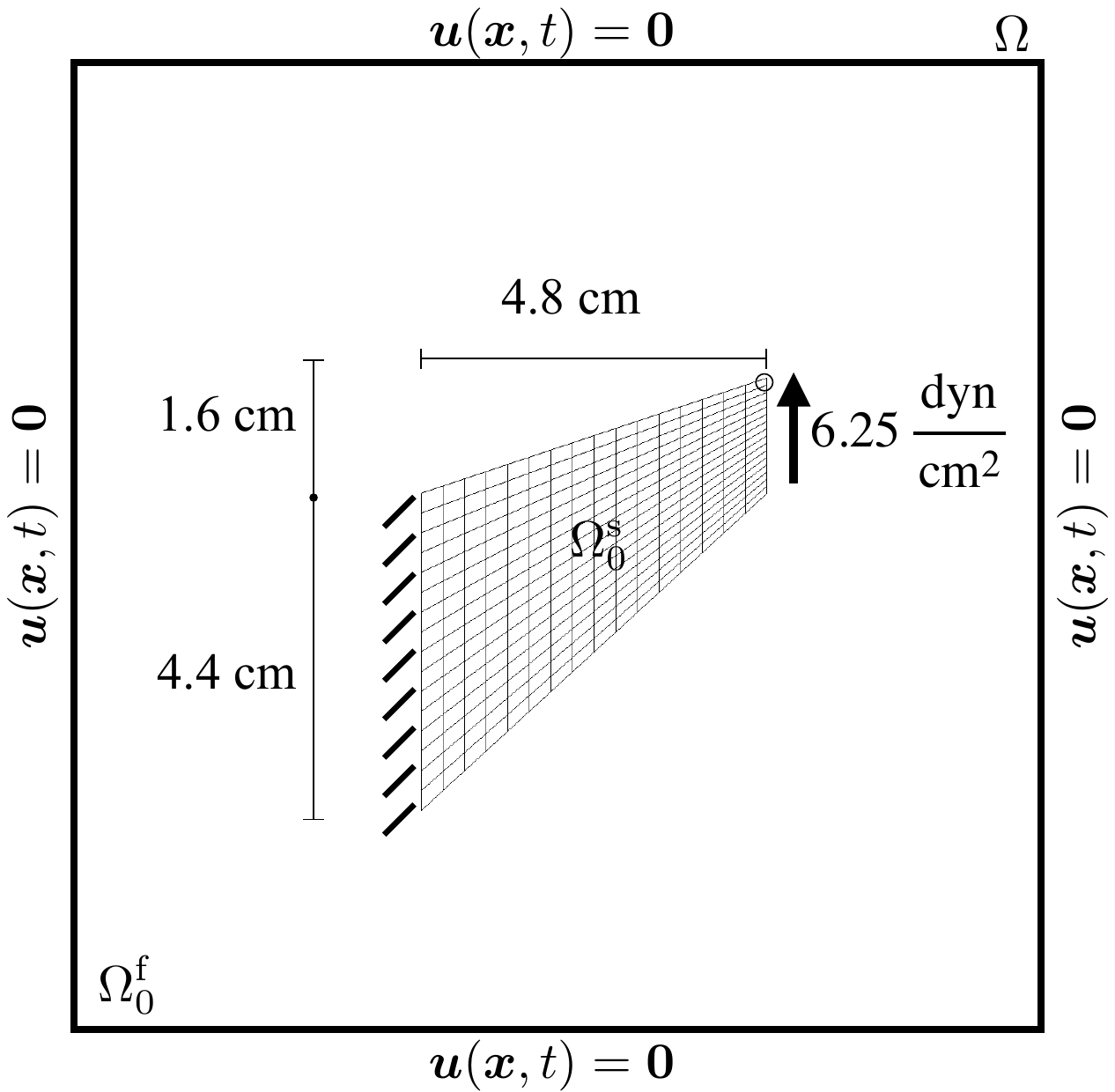}
\caption{
Specifications of the Cook's membrane benchmark (Subsection~\ref{subsec:cooks-membrane}).
The primary quantity of interest is the $y$-displacement as measured at the upper right hand corner, indicated by the circle.
The structure, shown here in its initial configuration and denoted by $\soliddomO$, is immersed in a fluid denoted by $\fluiddomO$.
The entire computational domain is $\Omega = \fluiddom \cup \soliddom$.
Zero fluid velocity is enforced on $\partial \Omega$.
}
\label{cooks}
\end{figure}

\begin{figure}
\centering
\begin{tabular}{l c c c}
&$t = 0\ \text{s}$ & $t = 9.6\ \text{s}$ & $t = 50\ \text{s}$  \\
\rotatebox{90}{$\quad\;\;$ Elemental}&
\includegraphics[width=.25\linewidth]{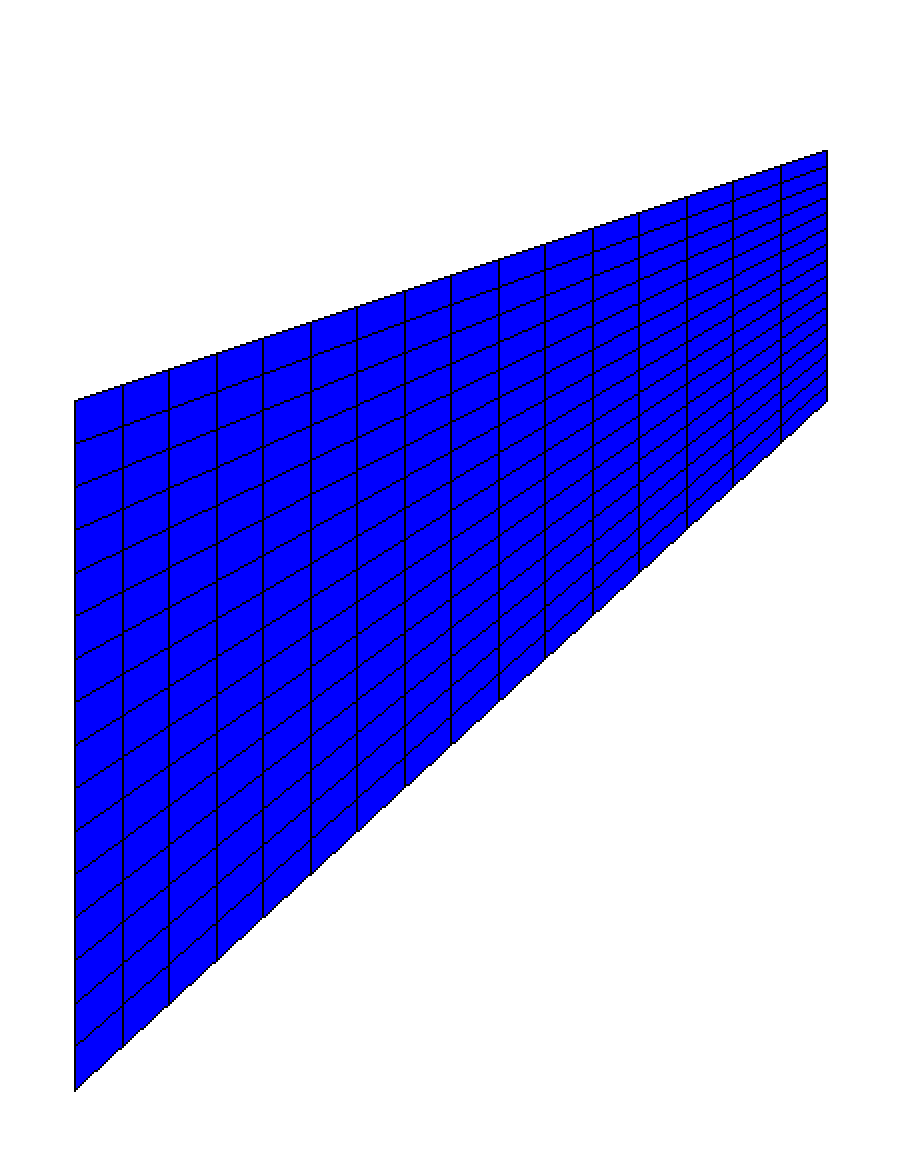}&
\includegraphics[width=.25\linewidth]{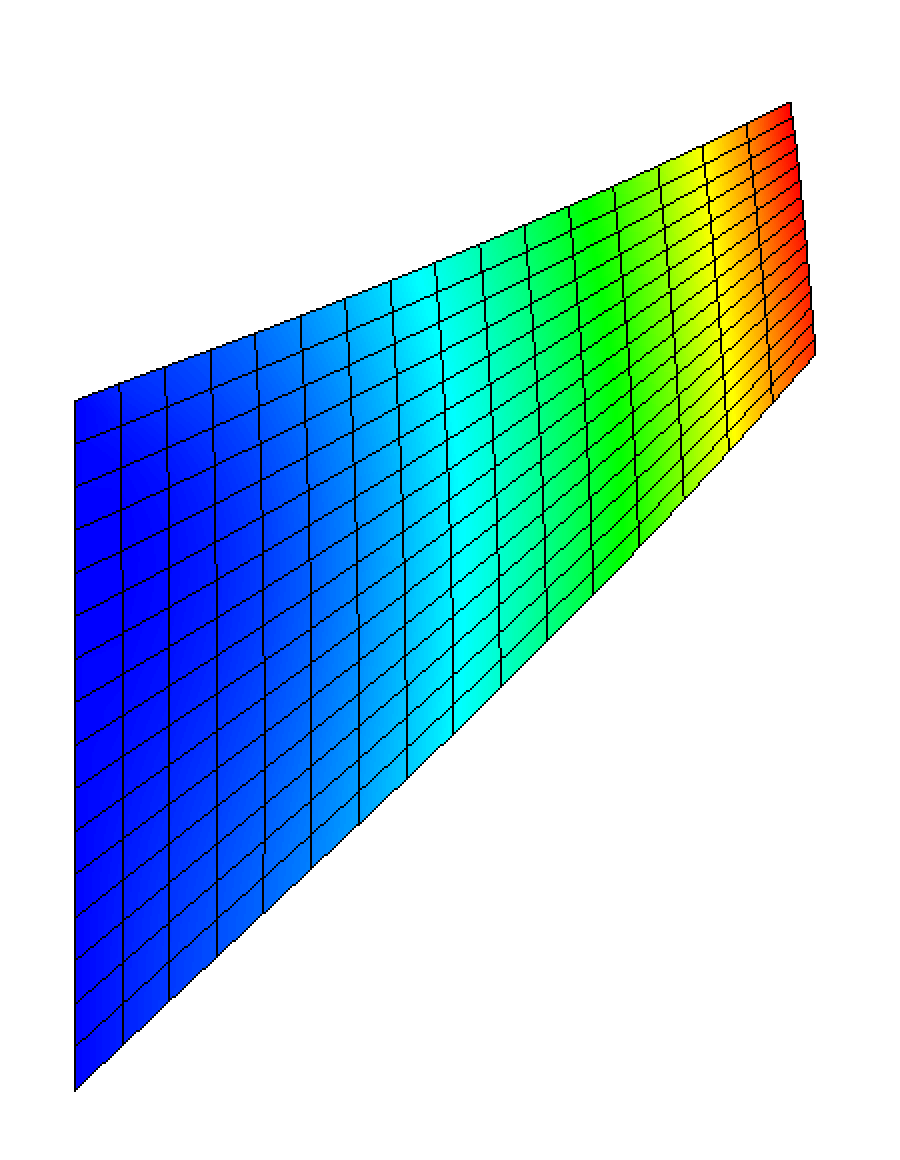}&
\includegraphics[width=.25\linewidth]{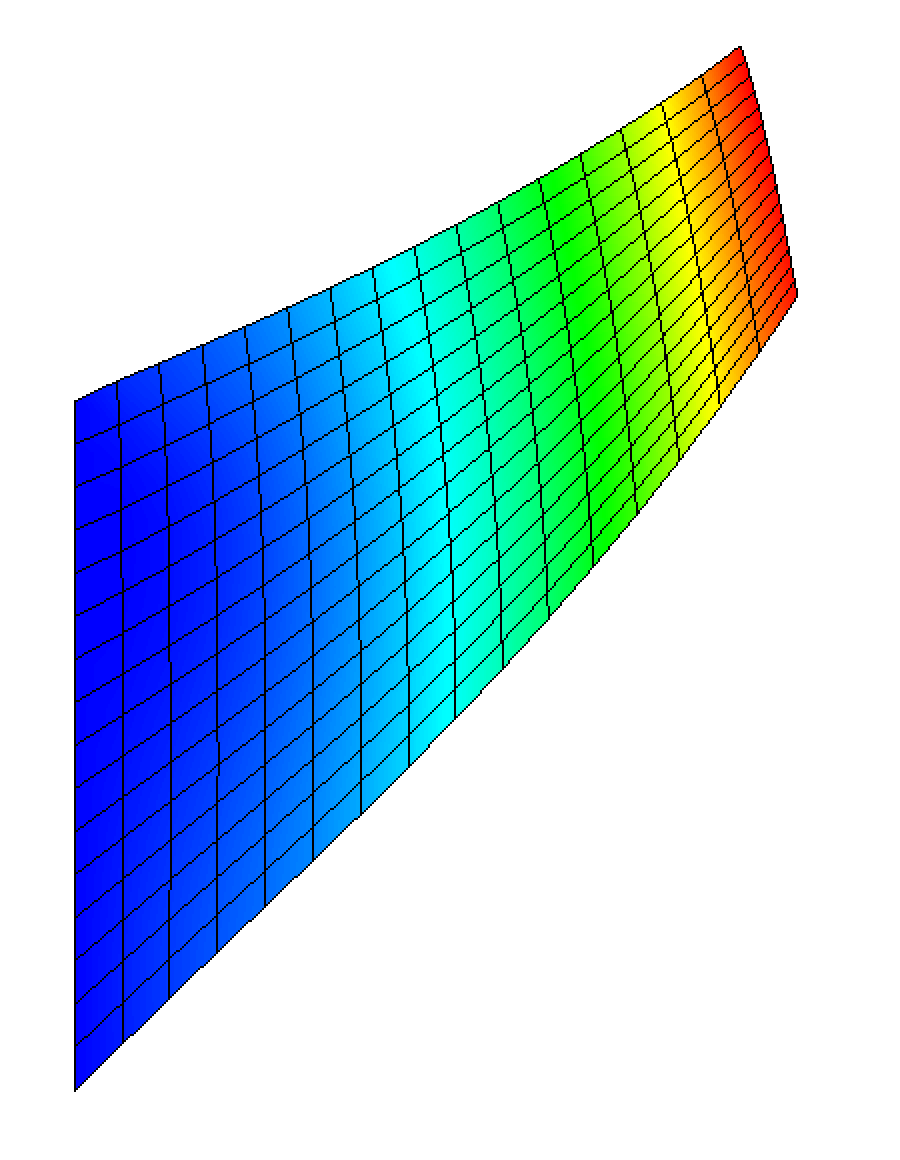}\\

\rotatebox{90}{$\quad\qquad$ Nodal}&
\includegraphics[width=.25\linewidth]{Figures/cm/cm_nodal_t=0.png}&
\includegraphics[width=.25\linewidth]{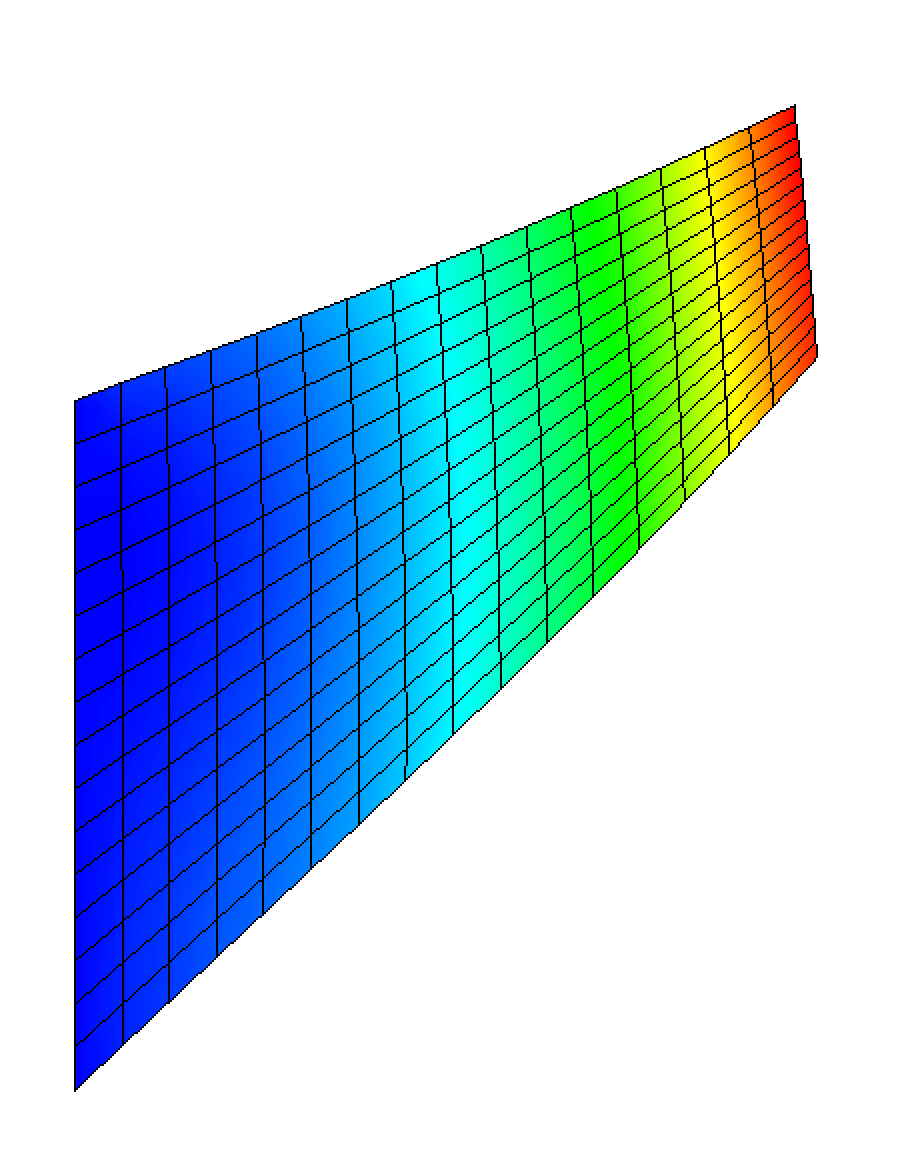}&
\includegraphics[width=.25\linewidth]{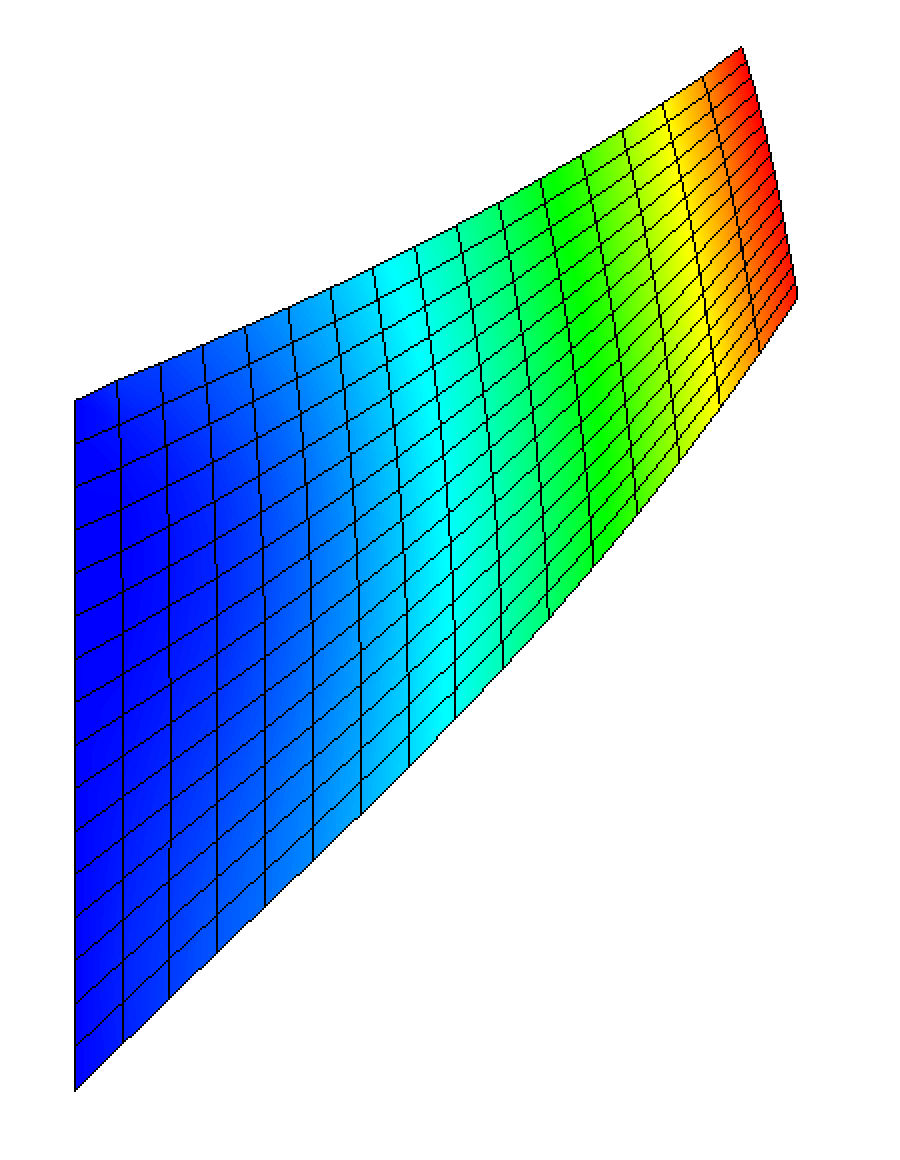}\\
&& $d^{\text{s}}_2$ &\\
&& \includegraphics[width=.3\linewidth]{Figures/color_bar.pdf}&\\
% more manual spacings
&&\hspace{0.025\linewidth}$0.0\ \text{cm}$ \hspace{.225\linewidth} $0.3\ \text{cm}$ &
\end{tabular}
\caption{Deformations and $y$-displacement of the Cook's membrane benchmark with both elemental and nodal coupling at different points in time.
  Time values are start of the simulation, (roughly) $0.5T_\text{l}$, and $T_\text{f}$.
  In both cases the structures are discretized with $\Qone$ elements.
  The nodal case uses $\mfac = 1.0$ and the elemental case uses $\mfac = 2.0$.}
\label{fig:cm_def}
\end{figure}

\begin{figure}
\begin{tabular}{l c c c c}
& $\Pone$ & $\Qone$ & $\Ptwo$ & $\Qtwo$ \\

\rotatebox{90}{$\quad\;\;\mfac = 1.0$}&
\includegraphics[width=.22\linewidth, trim={0 150 50 100}, clip]{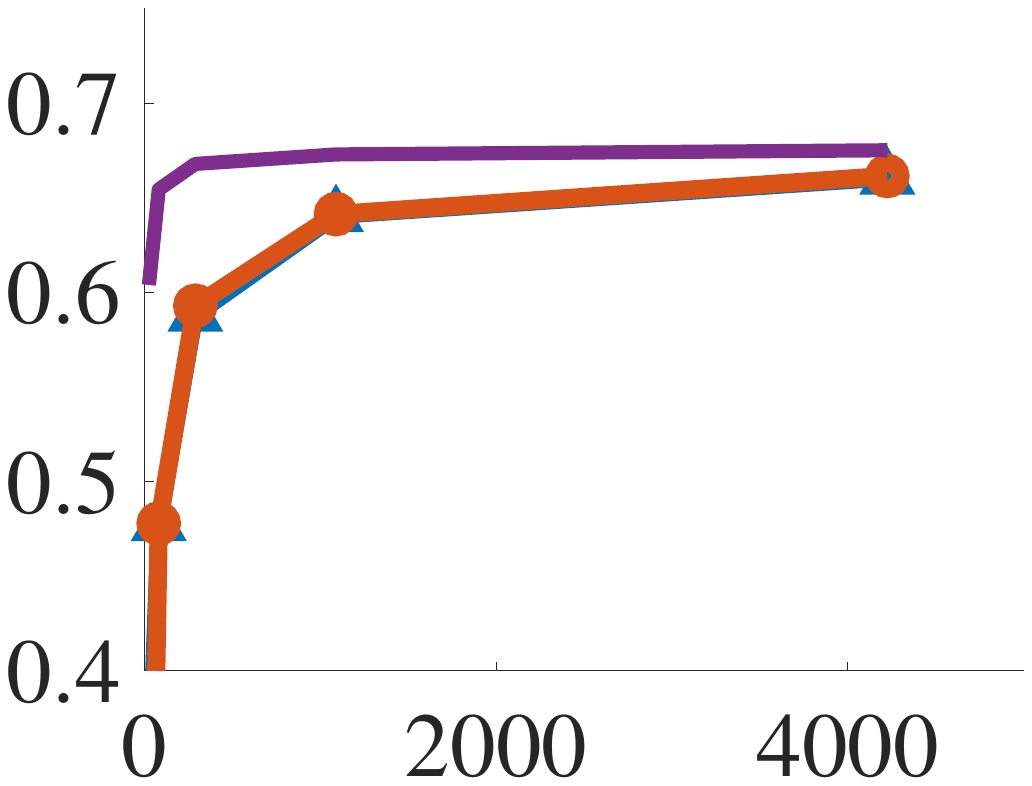}&
\includegraphics[width=.22\linewidth, trim={0 150 50 100}, clip]{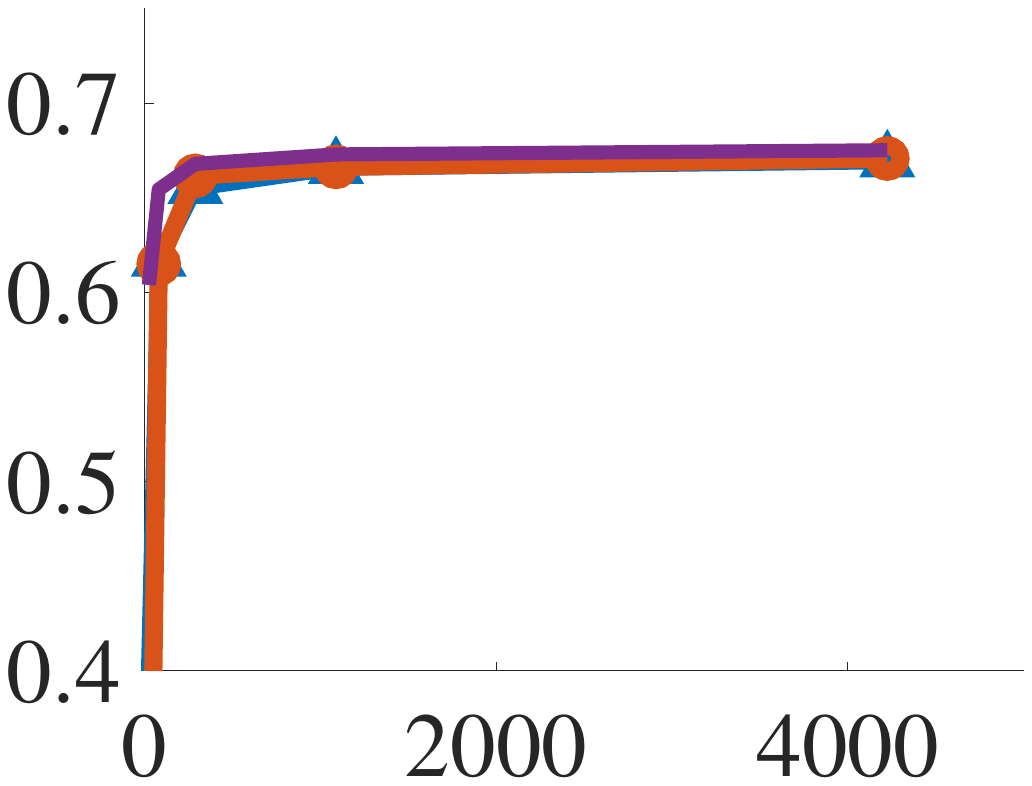}&
\includegraphics[width=.22\linewidth, trim={0 150 50 100}, clip]{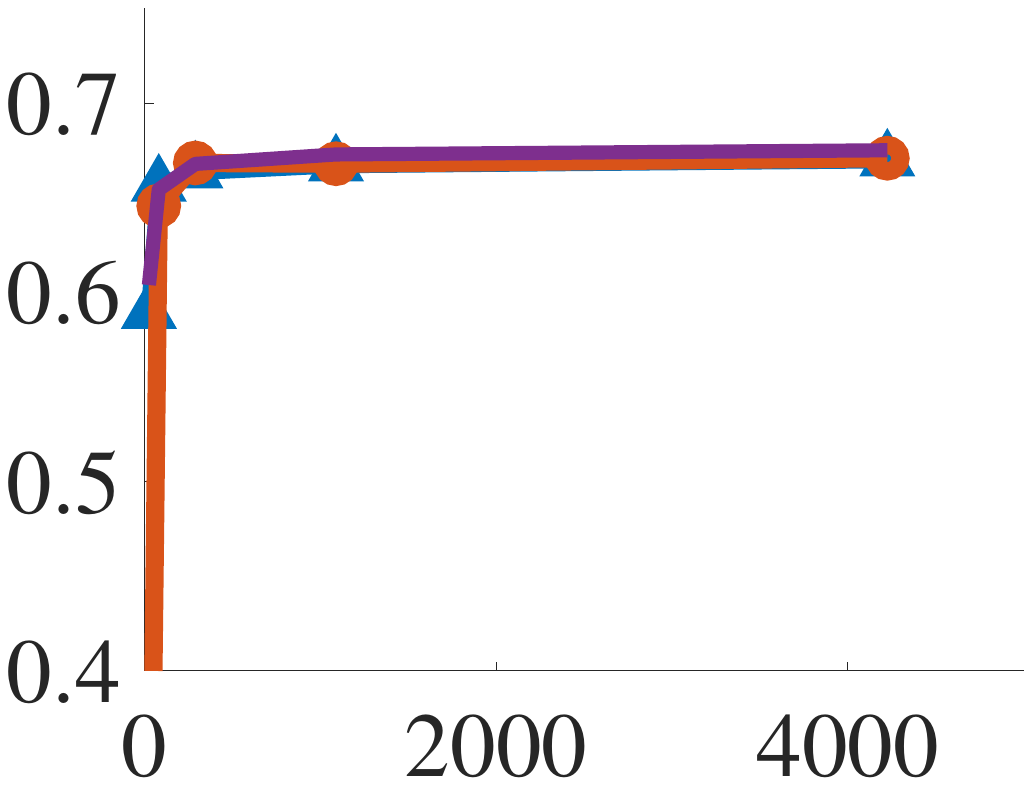}&
\includegraphics[width=.22\linewidth, trim={0 150 50 100}, clip]{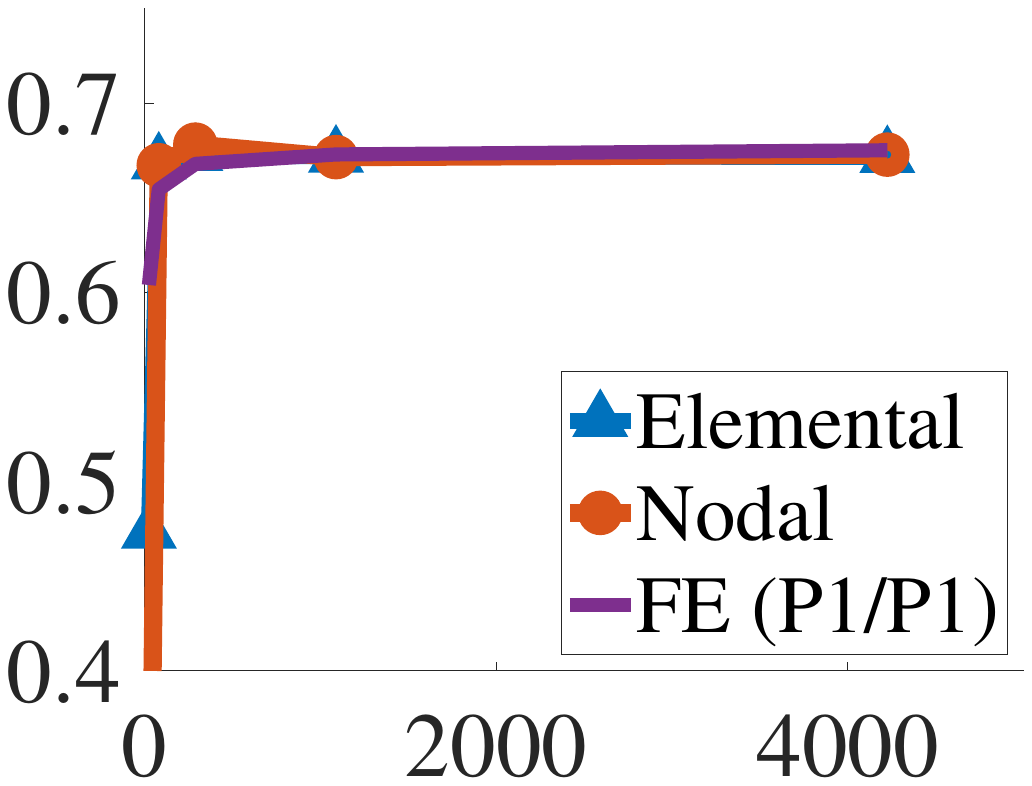}\\

\rotatebox{90}{$\quad\;\;\mfac = 1.5$}&
\includegraphics[width=.22\linewidth, trim={0 150 50 100}, clip]{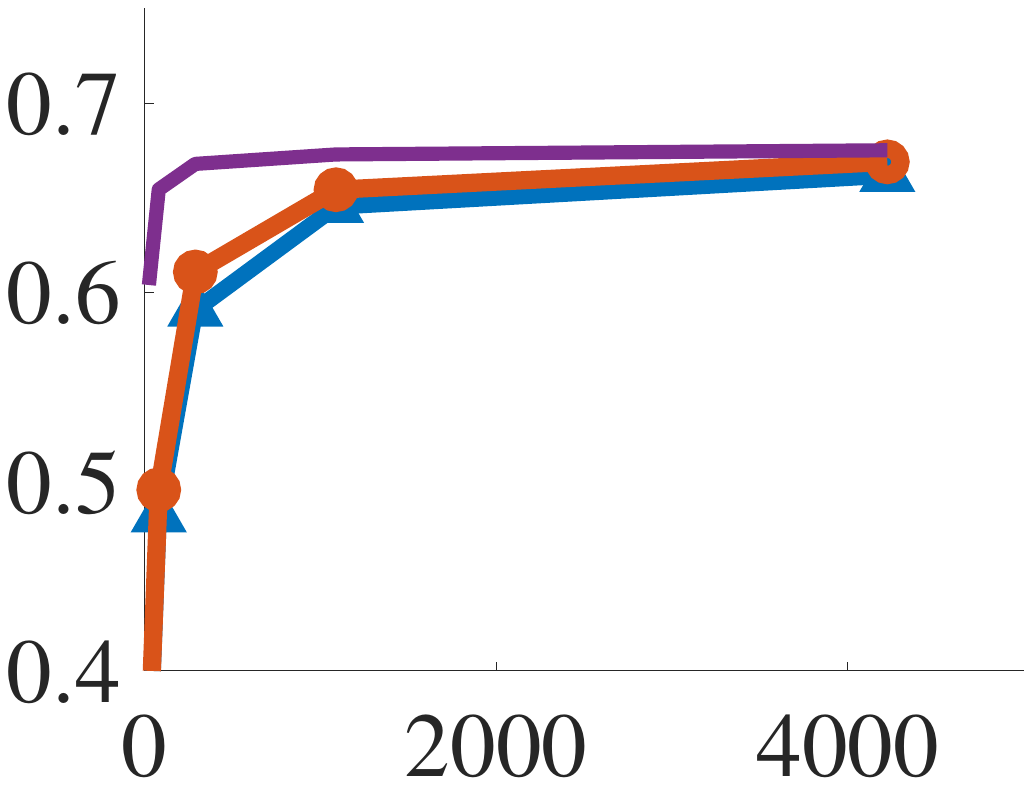}&
\includegraphics[width=.22\linewidth, trim={0 150 50 100}, clip]{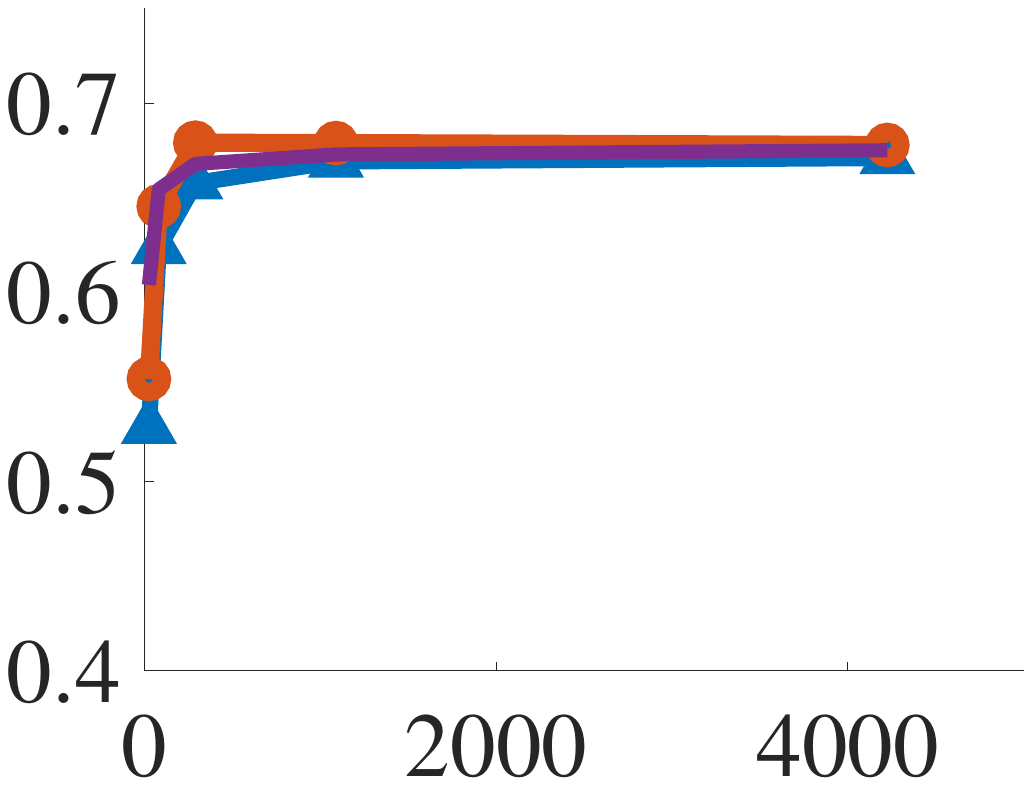}&
\includegraphics[width=.22\linewidth, trim={0 150 50 100}, clip]{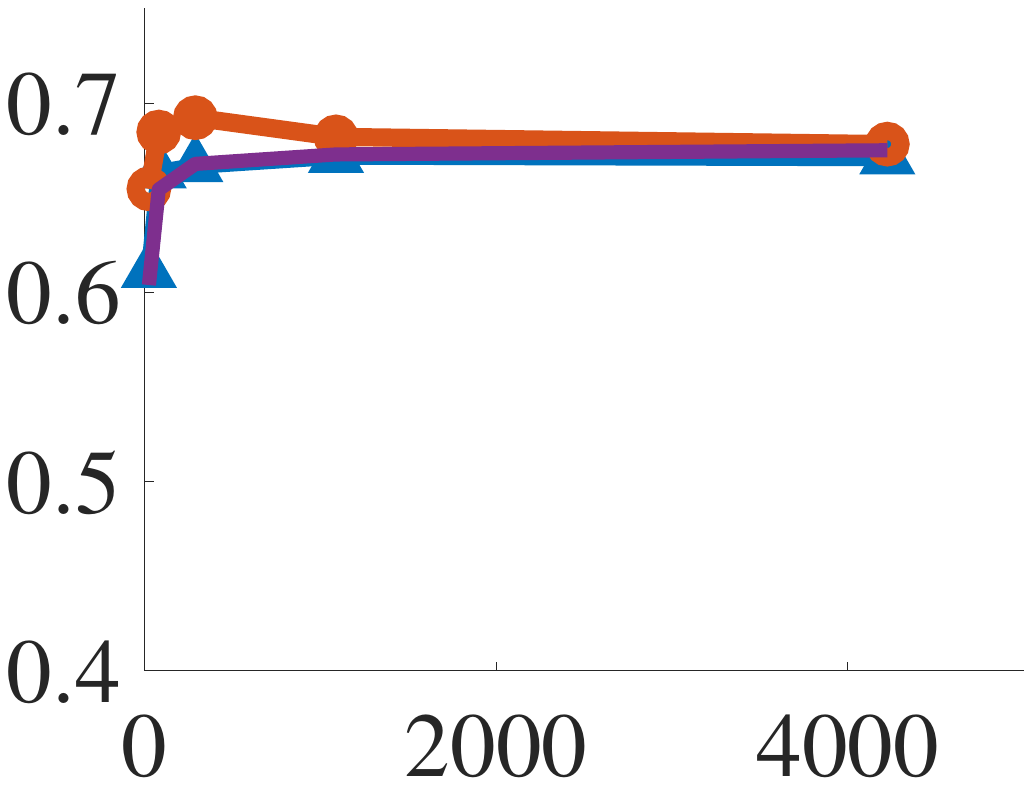}&
\includegraphics[width=.22\linewidth, trim={0 150 50 100}, clip]{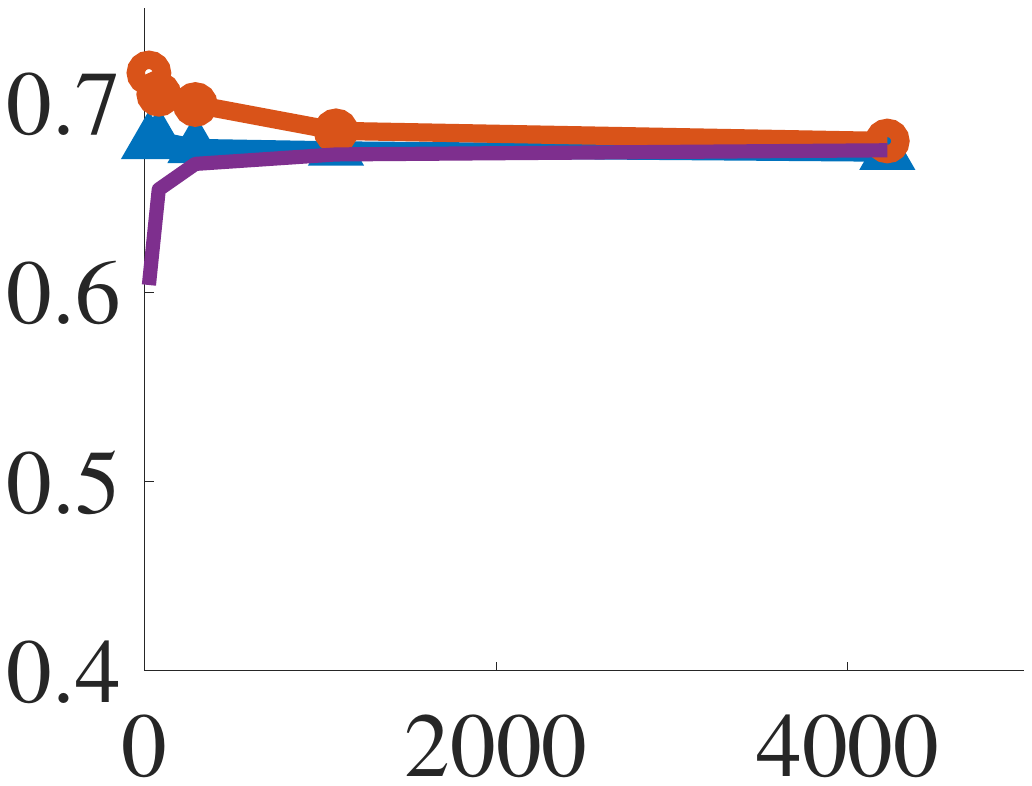}\\

\rotatebox{90}{$\quad\;\;\mfac = 2.0$}&
\includegraphics[width=.22\linewidth, trim={0 150 50 100}, clip]{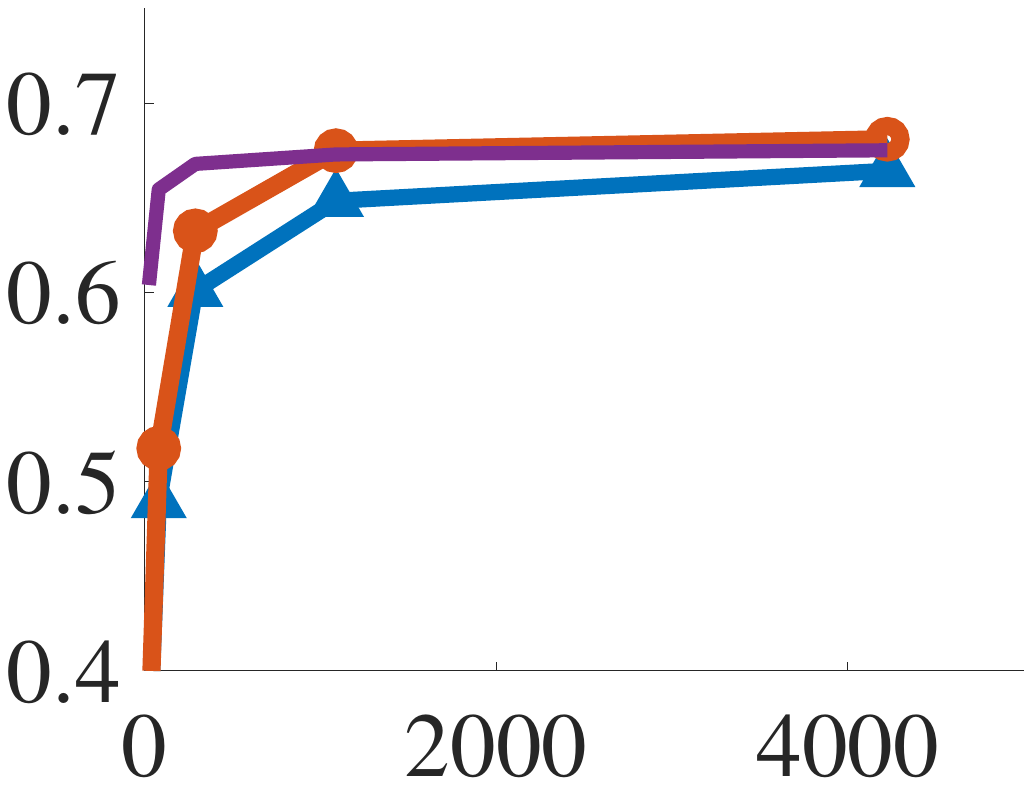}&
\includegraphics[width=.22\linewidth, trim={0 150 50 100}, clip]{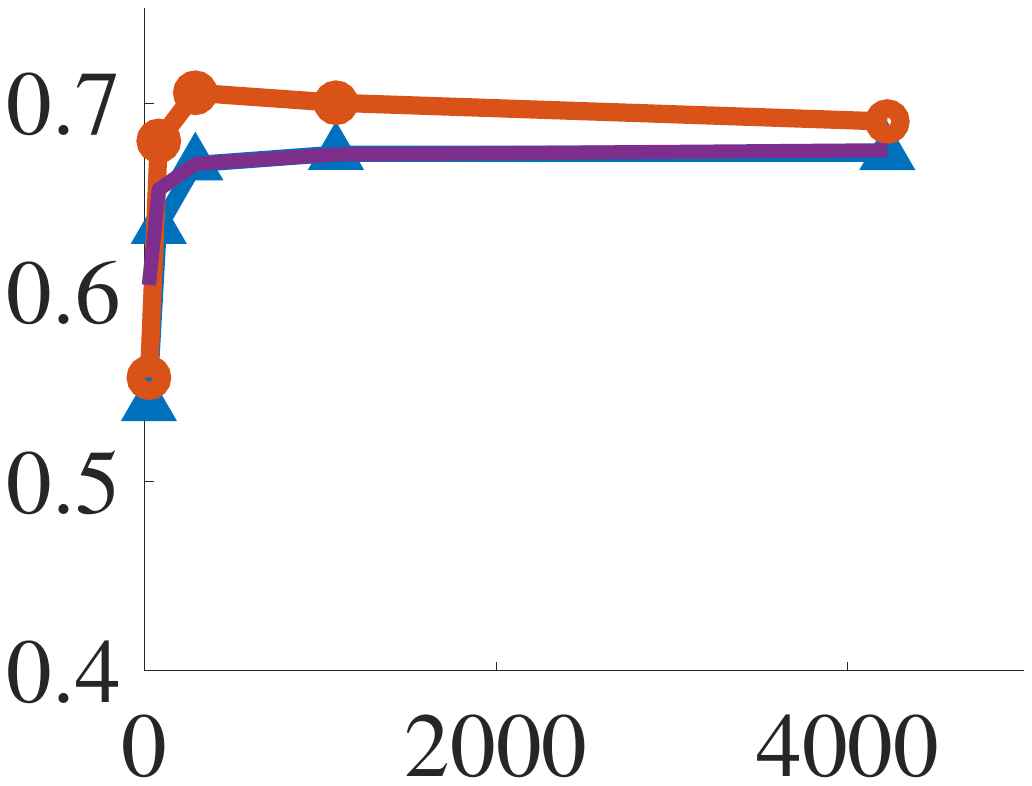}&
\includegraphics[width=.22\linewidth, trim={0 150 50 100}, clip]{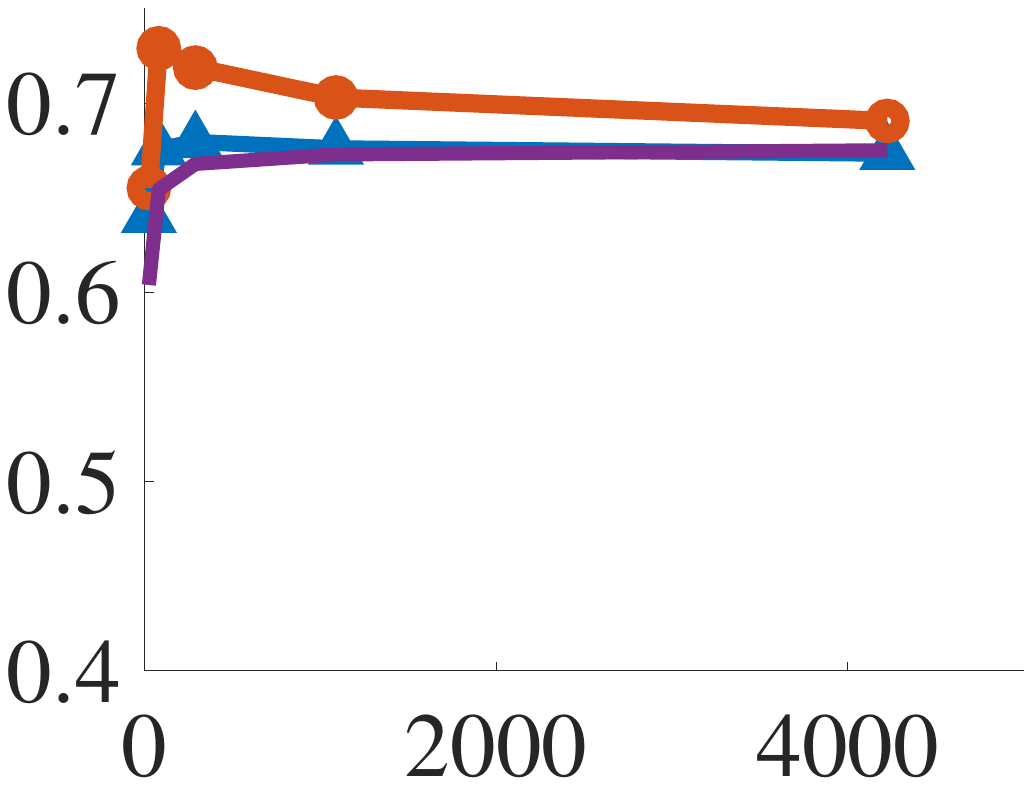}&
\includegraphics[width=.22\linewidth, trim={0 150 50 100}, clip]{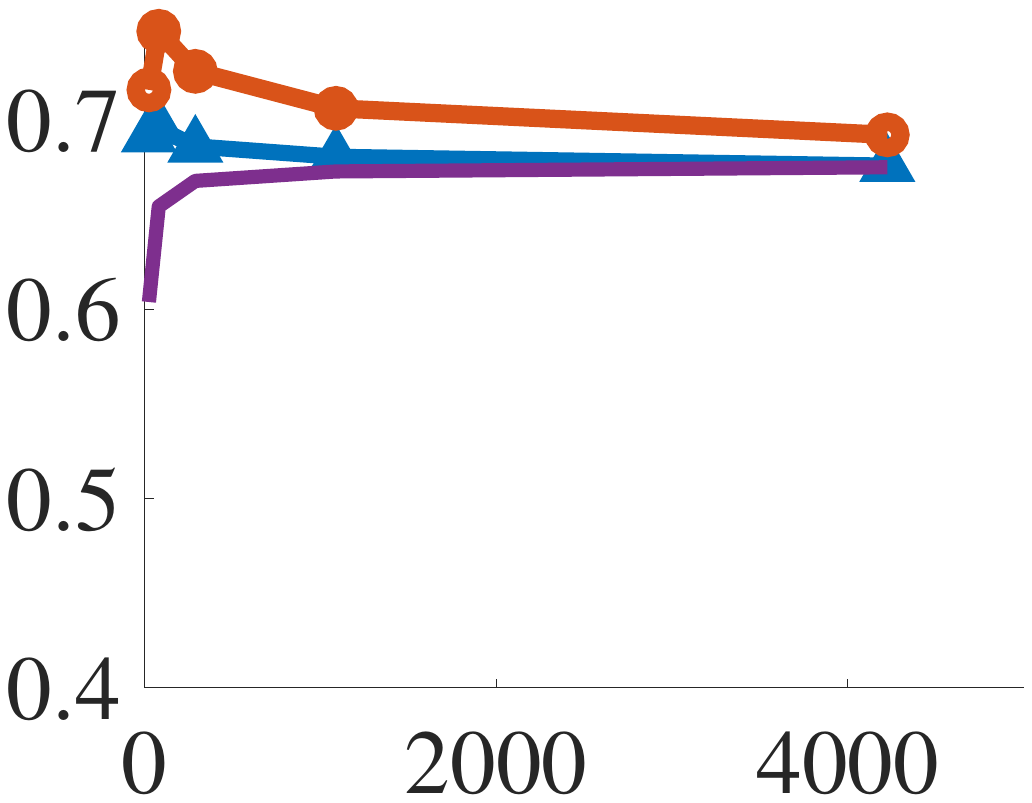}\\

\end{tabular}
\caption{The $y$-displacement ($y$-axis) in cm vs solid DoF count ($x$-axis) of the Cook's membrane for different coupling strategies and $\mfac$ values.}
\label{fig:cm_disp_bs3}
\end{figure}

\subsubsection{Torsion}
\label{Torsion}
This benchmark is based on a similar test by Bonet \etal~\cite{Bonet2015} and then modified by Vadala-Roth~\etal~\cite{Vadala-Roth2020} for an FSI framework.
It involves applying torsion to the top face of an elastic beam, while the opposite face is fixed in place; see Figure~\ref{torsion_diag}.
Zero traction boundary conditions are applied along all other faces.
The computational domain is $\Omega = [0, L]^3$ with $L = 9\ \text{cm}$.
The Cartesian grid uses $N = \mathrm{ceil}\left(M \efac \mfac \cdot \dfrac{9}{6}\right)$ cells in each coordinate direction, in which $M$ is the number of elements per edge in the Lagrangian mesh, $9$ is the length of the computational domain, and $6$ is the length of the structure.
The time step size is $\Delta t = 0.001 \euleriandx \ s$.
The structure is an elastic beam with an incompressible modified Mooney-Rivlin material model.

The torsion is imposed via displacement boundary conditions, and this face is rotated by $\theta_{\text{f}} = 2.5 \pi$.
The material and numerical parameters are listed in Table~\ref{tb:tt-param}.
A body damping parameter of $\eta_{\text{B}} = 0.2\cdot(c_1 + c_2) = 3600 \ \frac{\text{g}}{\text{cm}^3\cdot\text{s}}$ was applied on the interior of the structural domain to dampen oscillations as the structure reaches steady state.
The penalty parameter used to fix the stationary face in place is $\kappa_{\text{S}} = 1.25 \cdot \frac{\Delta x}{\Delta t} \frac{\text{dyn}}{\text{cm}^3}$.
We gradually apply the traction to the solid boundary linearly in time so that the full load is applied at $T_{\text{l}} = 2.0$ s, and we wait until time $T_\text{f} = 5.0$ s for the structure to reach equilibrium.
The number of solid DoFs range from $m = 65$ to $m = 12{,}337$.
Results are presented, only for $\mfac = 1.0$, in Figures \ref{fig:tt_def} and \ref{fig:tt_disp_bs3}.

Because of the greater computational costs of three-dimensional simulations, cases with finer Cartesian grids were omitted.
Using the intuition of the previous benchmark results and the results of Vadala-Roth~\etal~\cite{Vadala-Roth2020}, which show that $\mfac = 2.0$ yields convergent results with the standard IFED method, we expect the new methods with $\mfac = 1.0$ will converge to the benchmark solution.
Indeed, this appears to be the case, as seen in Figure \ref{fig:tt_disp_bs3}, although more structural DoFs are needed.
Figure~\ref{fig:tt_disp_bs3} shows that results from the elementally and nodally coupled methods are in close agreement, with the exception of $\Ptwo$ elements.
Here, the $\Ptwo$ elementally coupled case encountered severe time-step restrictions resulting from the mesh factor ratio.
For this reason, results for elemental coupling with $\Ptwo$ elements were omitted.
However, the nodally coupled method with $\Ptwo$ elements had no such time-step restrictions, and in fact, converges to the benchmark solution offered by the FE method.
In fact, this $\Ptwo$ case demonstrates faster convergence to the FE solution than the other cases shown here.
Figure~\ref{fig:tt_def} compares the deformations for $\Qone$ elements, again indicating close qualitative agreement between deformations computed by the two methods.

\begin{table}
\centering
\begin{tabular}{| c | c | c | c | }
\hline
Density & $\rho$ & $1.0$ & $\frac{\text{g}}{\text{cm}^3}$\\
\hline
Viscosity & $\mu$ & $0.04$ & $\frac{\text{dyn} \cdot \text{s}}{\text{cm}^2}$ \\
\hline
Material model & - & modified Mooney-Rivlin & - \\
\hline
Material constant 1 & $c_1$ & $9{,}000$ & $\frac{\text{dyn}}{\text{cm}^2}$  \\
\hline
Material constant 2 & $c_2$ & $9{,}000$ & $\frac{\text{dyn}}{\text{cm}^2}$  \\
\hline
Numerical bulk modulus & $\kappas$ & $168{,}000$
& $\frac{\text{dyn}}{\text{cm}^2}$\\
\hline
Final time & $T_\text{f}$ & $5.0$  & s\\
\hline
Load time & $T_{\text{l}}$ & $2.0$ & s \\
\hline
\end{tabular}
\caption{Parameters for the torsion benchmark (Section~\ref{Torsion}).}
\label{tb:tt-param}
\end{table}

\begin{figure}
\centering
\includegraphics[width=0.6\linewidth]{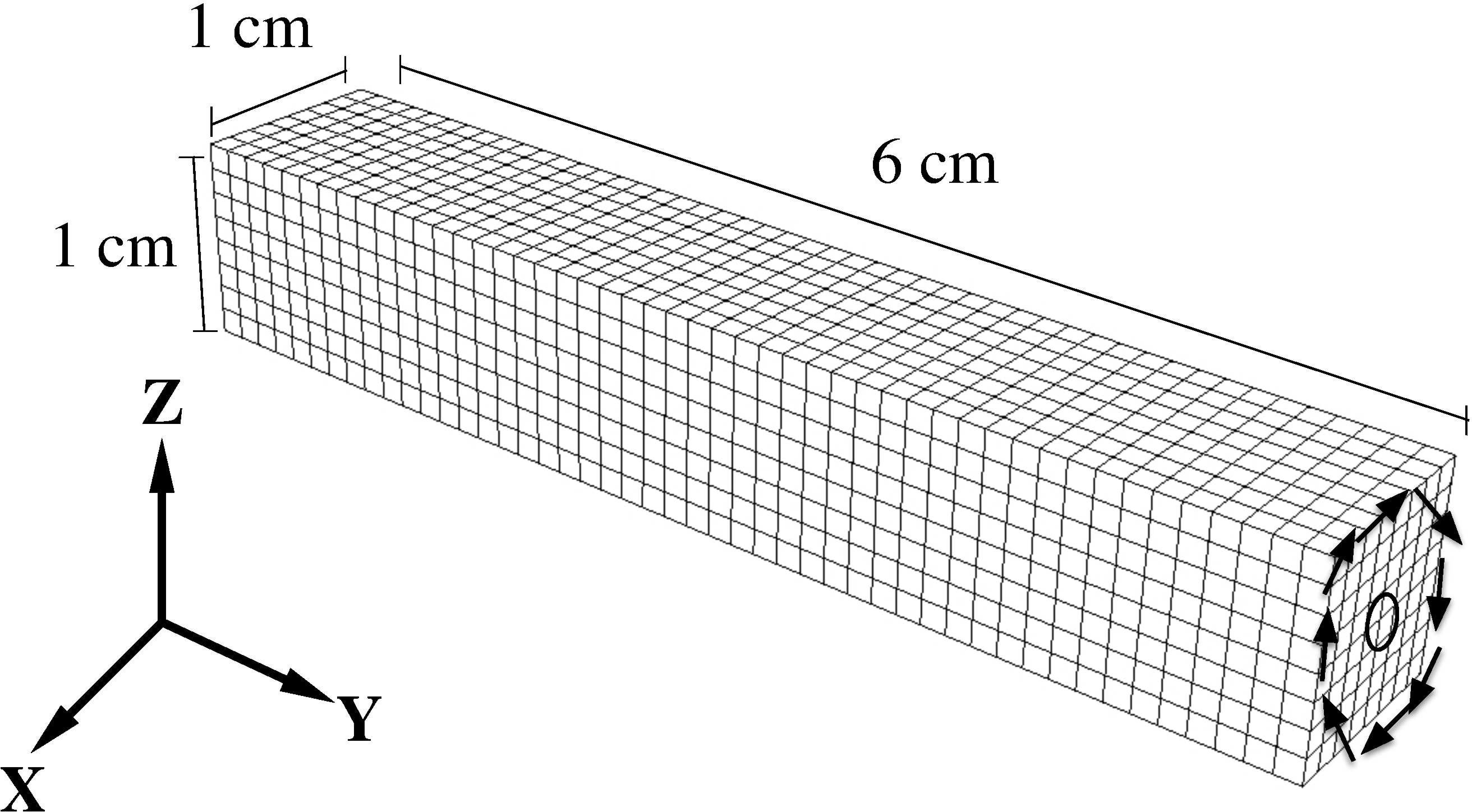}
\caption{
Specifications of the torsion benchmark (Section~\ref{Torsion}).
The face opposite the applied torsion is kept fixed.
The quantity of interest is the $y$-displacement as measured at the encircled area.
To simplify the diagram of this three dimensional benchmark, we omit the computational domain in this figure describing the problem setup.
In the IFED model, however, the structure is contained within a computational domain with dimensions $\Omega = [0,L]^3$ and $L = 9 \ \text{cm}$, and the solid mesh is placed in the center of this domain.
Zero fluid velocity is enforced on $\partial \Omega$.
}
\label{torsion_diag}
\end{figure}

\begin{figure}
\centering
\begin{tabular}{l c c c}
&$t = 0\ \text{s}$ & $t = 1.05\ \text{s}$ & $t = 5\ \text{s}$  \\
\rotatebox{90}{$\quad\;\;$ Elemental}&
\includegraphics[width=.25\linewidth]{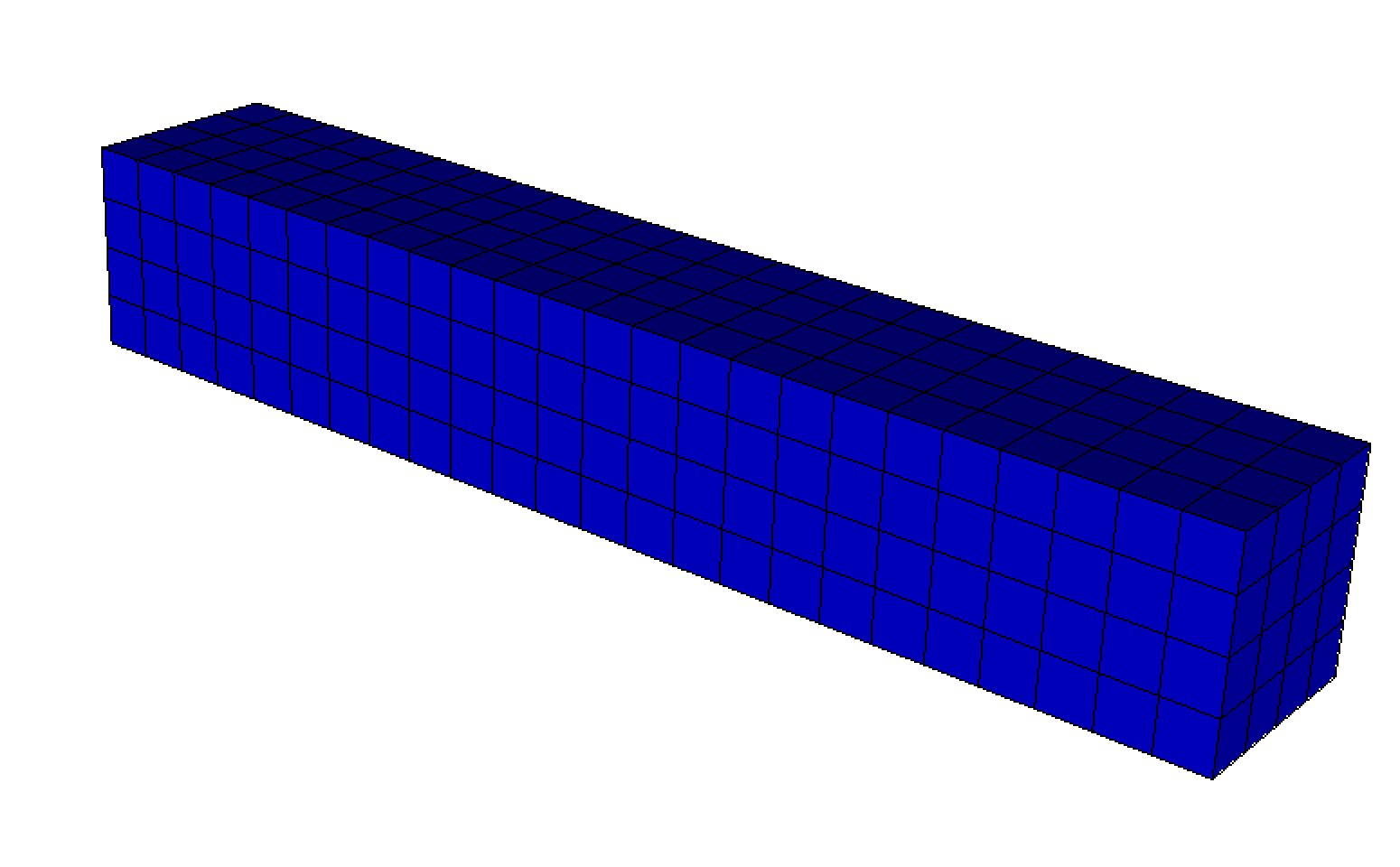}&
\includegraphics[width=.25\linewidth]{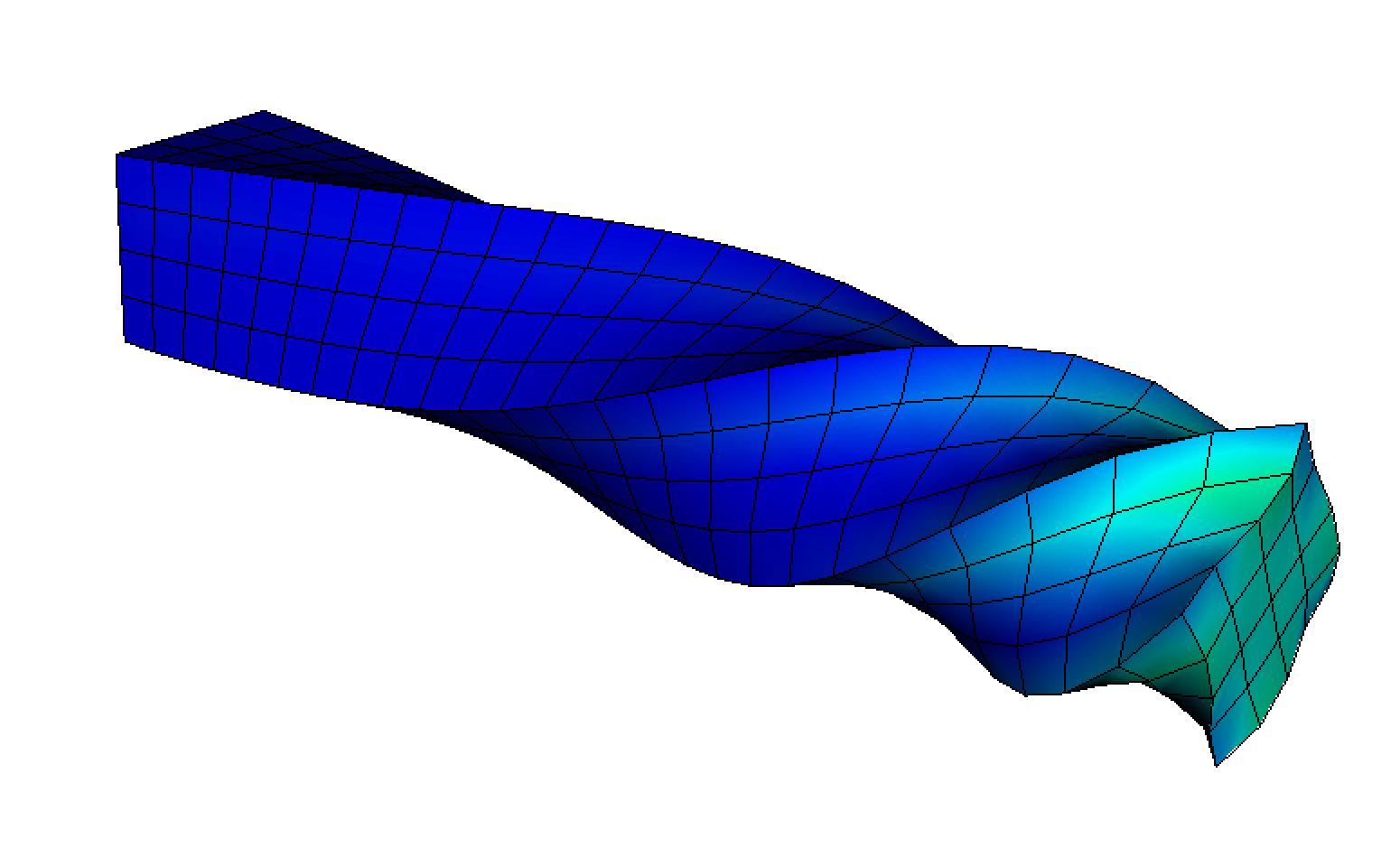}&
\includegraphics[width=.25\linewidth]{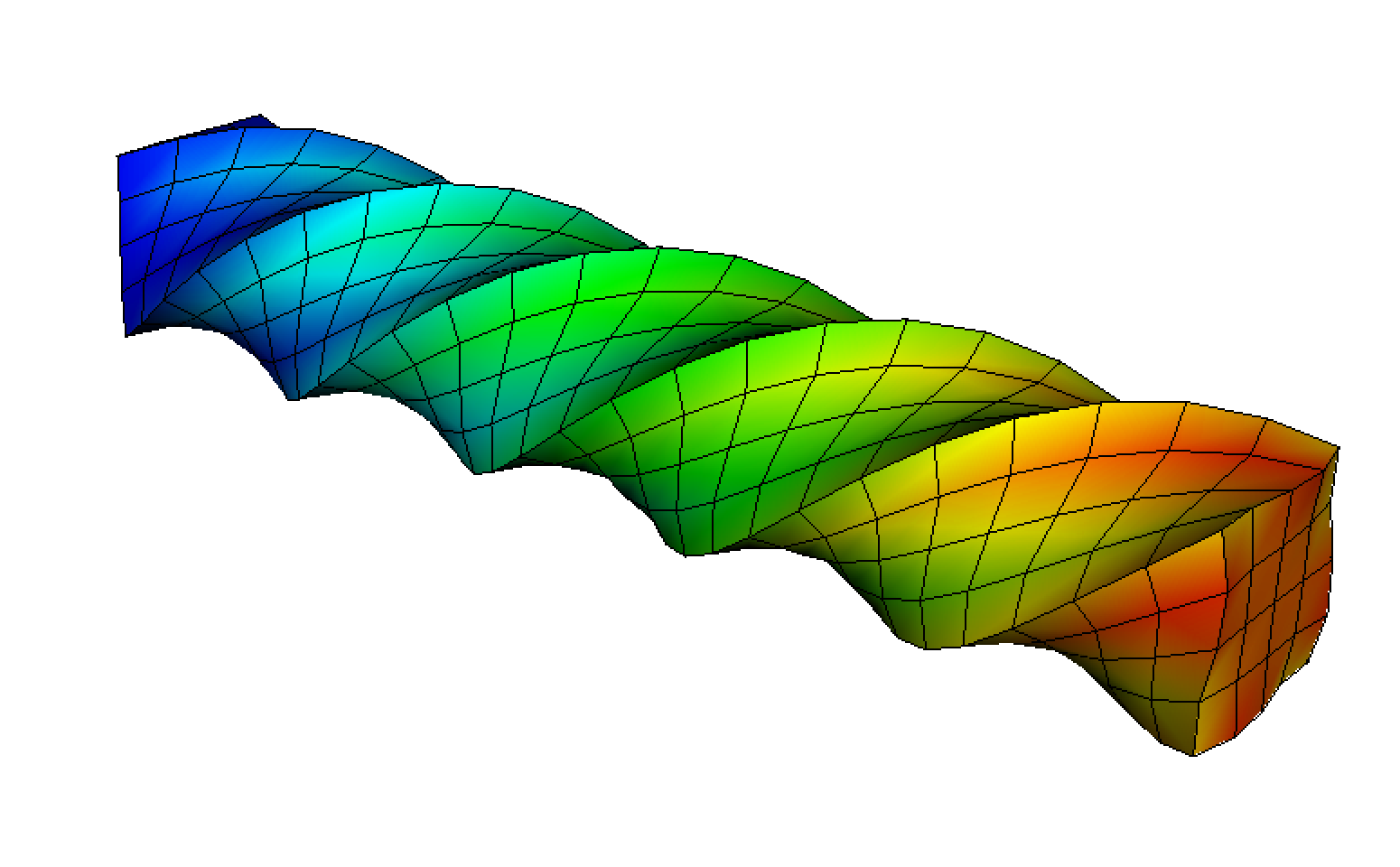}\\

\rotatebox{90}{$\quad\qquad$ Nodal}&
\includegraphics[width=.25\linewidth]{Figures/tt/tt_nodal_t=0.png}&
\includegraphics[width=.25\linewidth]{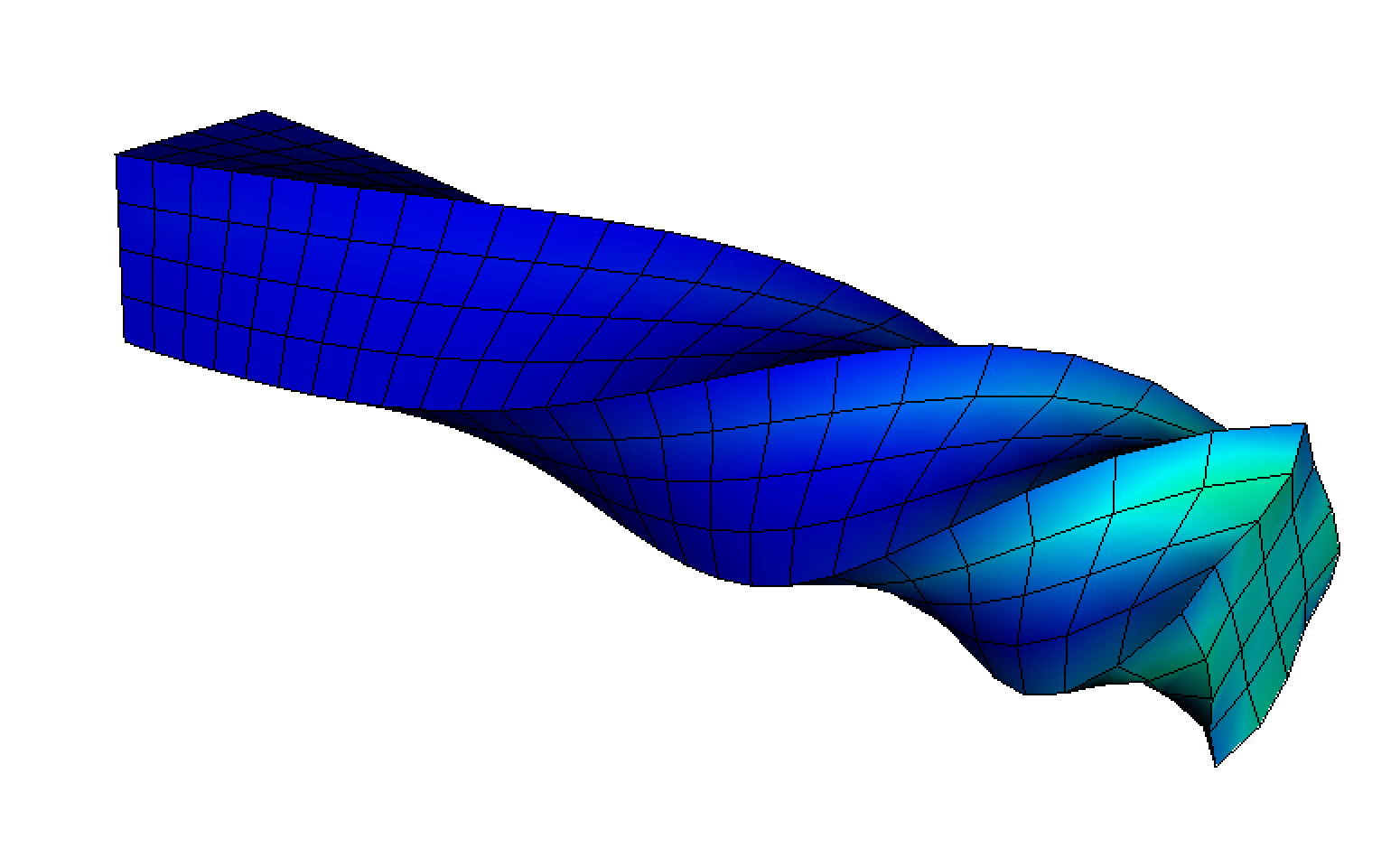}&
\includegraphics[width=.25\linewidth]{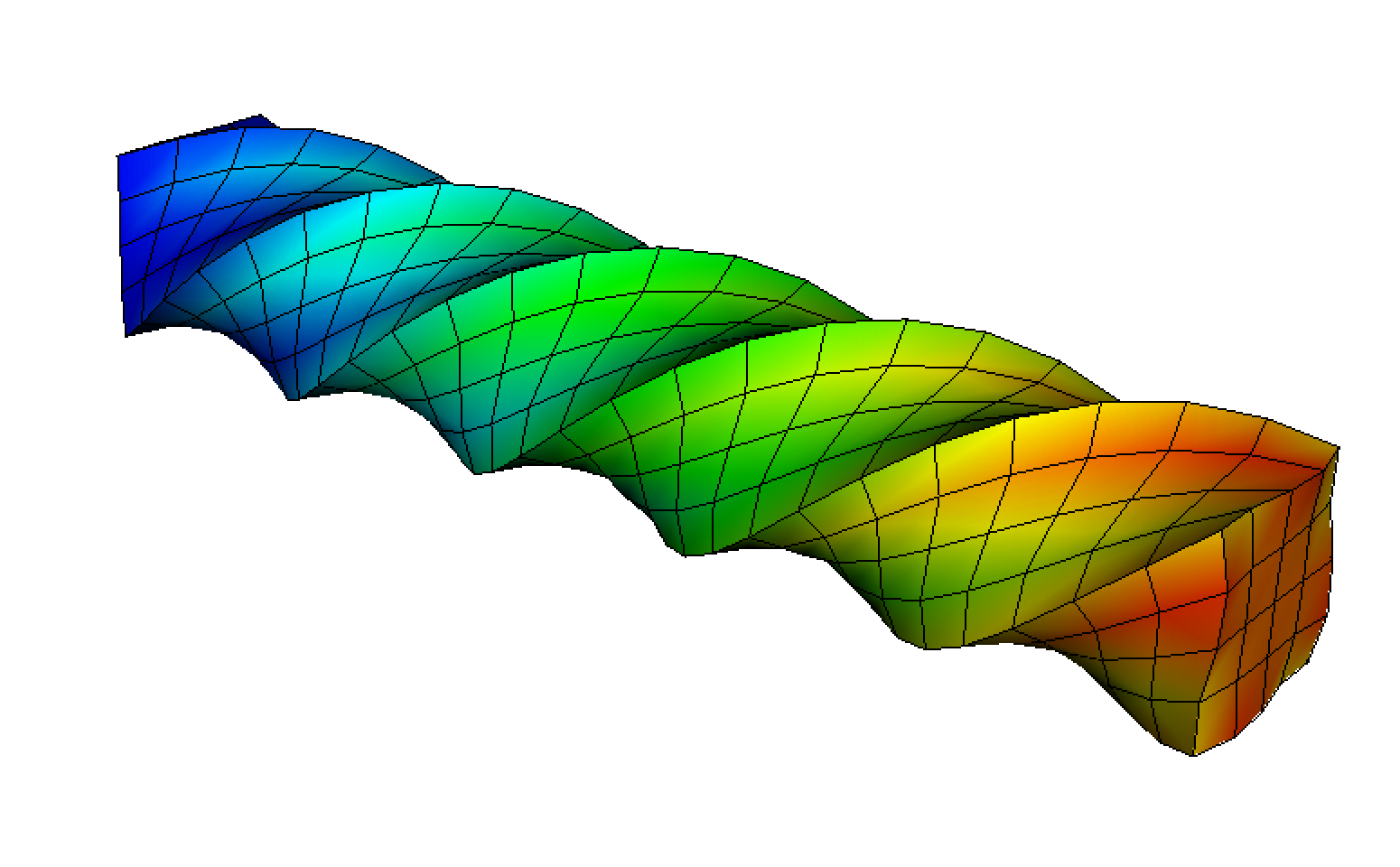}\\

&& $d^{\text{s}}_2$ &\\
&& \includegraphics[width=.3\linewidth]{Figures/color_bar.pdf}&\\
% more manual spacings
&&\hspace{0.025\linewidth}$0.0\ \text{cm}$ \hspace{.225\linewidth} $0.35\ \text{cm}$ &
\end{tabular}
\caption{Deformations of the torsion benchmark and $y$-displacement for both elemental and nodal coupling at different points in time.
  Time values are start of the simulation, (roughly) $0.5T_\text{l}$, and $T_\text{f}$.
  In both cases the structures are discretized with $\Qone$ elements and use $\mfac = 1.0$.}
\label{fig:tt_def}
\end{figure}

\begin{figure}
\begin{tabular}{c c c c}
$\Pone$ & $\Qone$ & $\Ptwo$ & $\Qtwo$ \\
\includegraphics[width=.22\linewidth, trim={0 150 50 100}, clip]{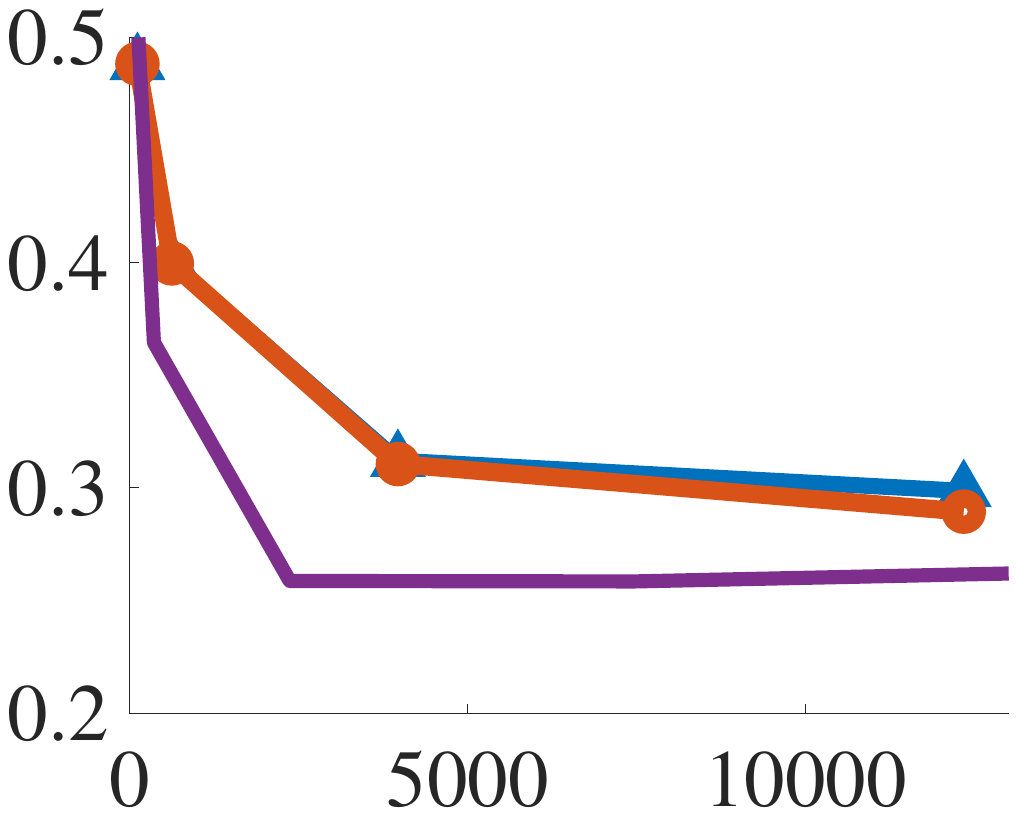}&
\includegraphics[width=.22\linewidth, trim={0 150 50 100}, clip]{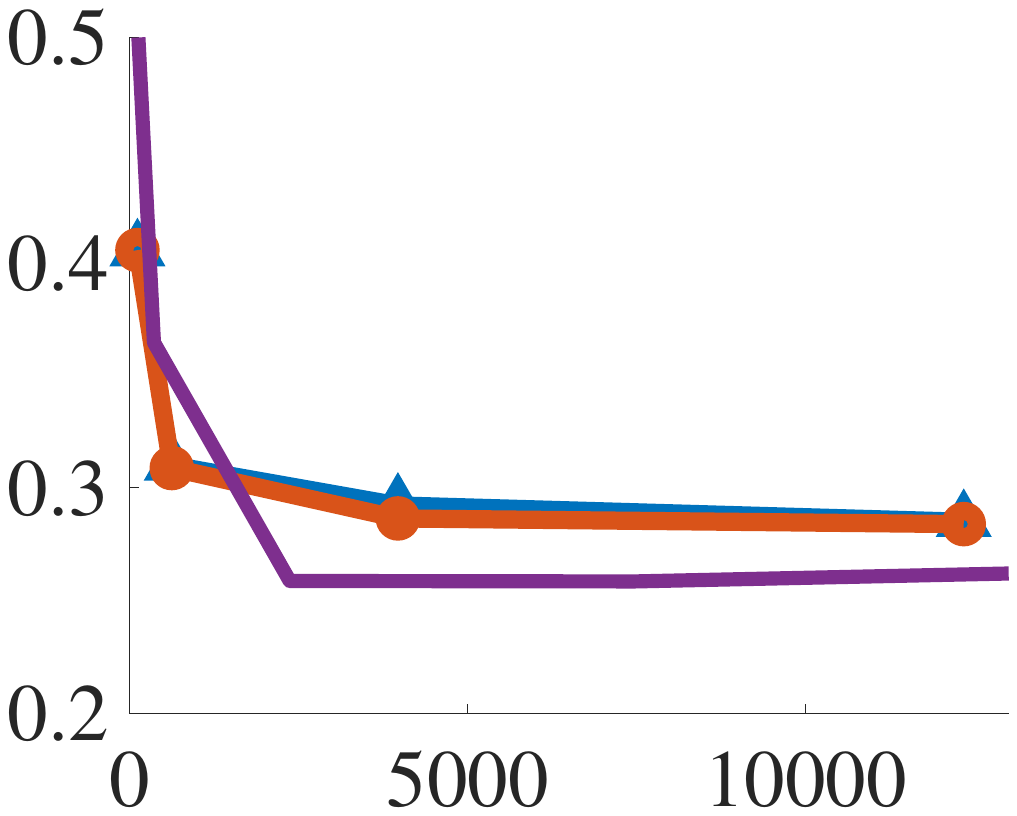}&
\includegraphics[width=.22\linewidth, trim={0 150 50 100}, clip]{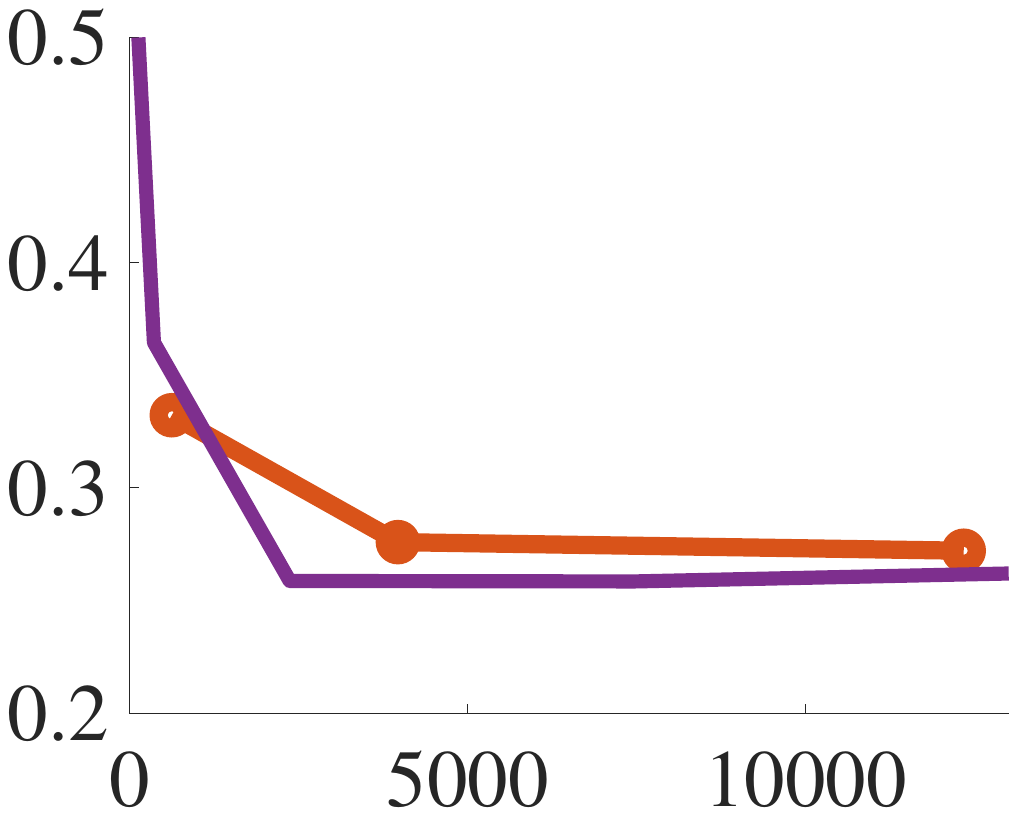}&
\includegraphics[width=.22\linewidth, trim={0 150 50 100}, clip]{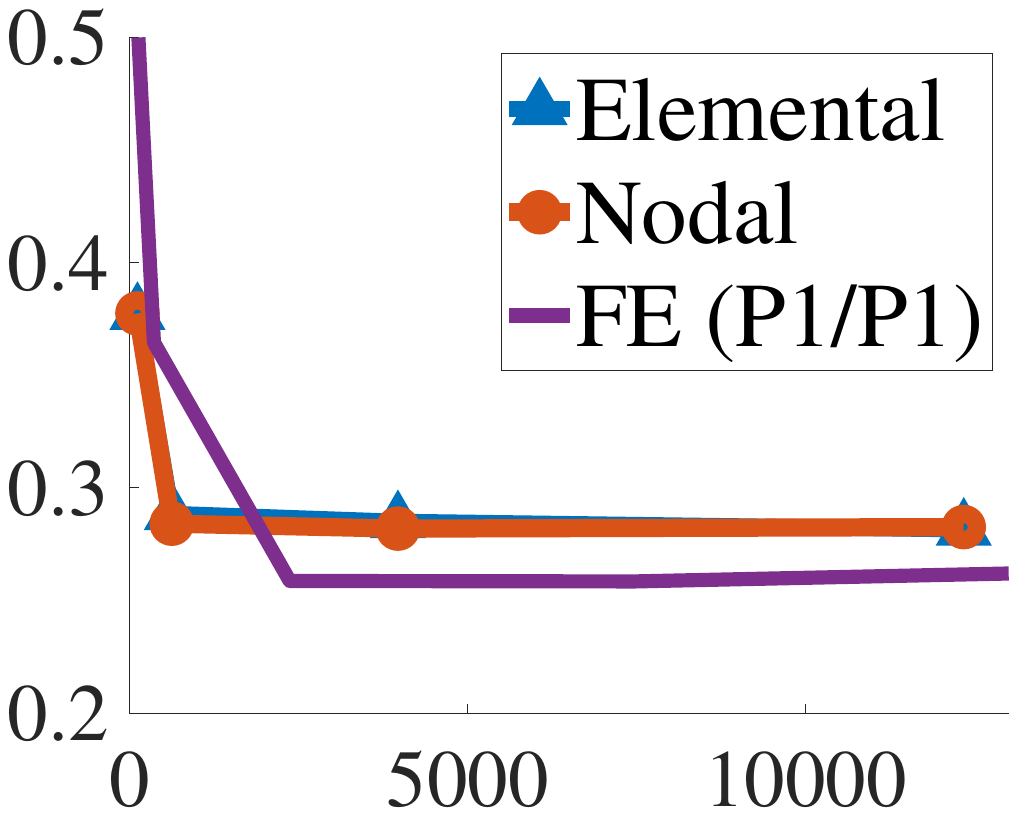}\\

\end{tabular}
\caption{The $y$-displacement ($y$-axis) in cm vs solid DoF count ($x$-axis) of the torsion benchmark using different coupling strategies and $\mfac = 1.0$ values.
 Note for the $\Ptwo$ case that the results with elemental coupling are omitted because these cases demonstrated severe time-step
  restrictions.}
\label{fig:tt_disp_bs3}
\end{figure}

\subsubsection{Hessenthaler's Three-dimensional FSI Benchmark}
\label{FSI-benchmark}
This benchmark involves the FSI of an elastic beam and enclosure geometry, depicted in Figure~\ref{fig:fsi}, with two parabolic inflow regions and one larger outflow region.
It was first introduced as an experiment and then used as a benchmark for a monolithic arbitrary Lagrangian-Eulerian (ALE) FSI scheme by Hessenthaler \etal~\cite{Hessenthaler2017a, Hessenthaler2017b}.
We focus on the ``Phase I'' part of the benchmark, which is a steady state benchmark.
The elastic beam is less dense than the surrounding fluid and is thus subject to upward buoyancy forces.
To adapt the problem to the IB framework, the enclosure and beam are placed in a cube-shaped computational domain $[0,L]\times [L/2,L/2] \times [0,L]$, in which $L = 11.9$ cm.
The structure is an elastic beam with an incompressible modified neo-Hookean material model.
The Cartesian grid uses $N = 24$ or $N = 32$ cells in each coordinate direction and $3$ levels with a refinement ratio of $2$, resulting in an effective fine-grid spacing of $4 N$.

The time step size is $0.02 \euleriandx / 63.0 \ \text{s}$ in which $0.02$ is a CFL-like stability restriction and $63.0 \frac{\text{cm}}{s}$ is the largest anticipated velocity.
Additionally, the enclosure has a thickness of $w = 0.1$ cm, whereas the original problem featured an enclosure with no thickness (although the 3D printed enclosure in the experiment had thickness of $w = 29 \, \mu \text{m}$~\cite{Hessenthaler2017a}).

The inflow boundary conditions were unidirectional in the z-direction and had peak values given by
\begin{equation}
    u_3 =
    \begin{cases}
        63.0 \, \frac{\text{cm}}{\text{s}}\cdot (12t^2 - 16t^3) & y > 0, t < 0.5 \, \text{s}, \\
        61.5 \, \frac{\text{cm}}{\text{s}}\cdot (12t^2 - 16t^3) & y < 0, t < 0.5 \, \text{s}, \\
        63.0 \, \frac{\text{cm}}{\text{s}} & y > 0, t \geq 0.5 \, \text{s}, \\
        61.5 \, \frac{\text{cm}}{\text{s}} & y > 0, t \geq 0.5 \, \text{s},
    \end{cases}
   \label{eq:fsi-bc}
\end{equation}
Zero fluid traction is applied on the outflow region.
Because of fluid incompressibility, fluid will strictly flow out of the domain at this boundary.
To correct for any spurious inflow we use a penalty technique introduced by Bodony~\cite{Bodony2006}.
The original work uses techniques described in the work of Bazilevs \etal~\cite{Bazilevs2009}, in which a fluid traction is imposed to counteract flow into the domain at the outflow boundary.
Zero velocity boundary conditions are specified for the remaining components and boundaries.

In the original work, different densities are used for the solid and fluid, with the fluid being denser than the solid (see Table~\ref{tb:fsi-param}).
This leads to a buoyancy body force of the form $\Fb_{\text{g}} = (\rhos - \rhof) \cdot (0,g,0)$, in which $\rhos$ is the solid density, $\rhof$ is the fluid density, and $g = -980.665 \, \frac{\text{cm}}{\text{s}^s}$ is the acceleration due to gravity.
Figure~\ref{fig:fsi} shows the benchmark specifications.
Damping parameters for the interior of the structural domains are $\eta_{\text{B}_1} = \frac{\rhos}{\Delta t} \frac{\text{g}}{\text{cm}^3\cdot\text{s}}$ for the elastic beam and $\eta_{\text{B}_2} = 0.002\cdot\frac{\rhos}{\Delta t} \frac{\text{g}}{\text{cm}^3\cdot\text{s}}$ for the rigid enclosure.
The penalty parameter for the body force to enforce rigidity of the enclosure is $\kappa_{\text{B}} = 30.0 \cdot \frac{\Delta x}{\left(\Delta t\right)^2} \frac{\text{dyn}}{\text{cm}^4}$, and the penalty parameter for the surface traction to fix the beam to the enclosure is $\kappa_{\text{S}} = 0.3 \cdot \frac{\Delta x}{\left(\Delta t\right)^2}\frac{\text{dyn}}{\text{cm}^3}$.
We gradually increase the parabolic inflow velocity in time according to Equation \eqref{eq:fsi-bc} so that the full velocity is applied at $T_{\text{l}} = 0.5$ s, and we wait until time $T_\text{f} = 30.0$ s for the structure to reach equilibrium.
We use a constant density fluid solver with density $\rhof$, in which the momentum term $\rhof \frac{D\ub}{Dt}(\xb,t)$ in equation \eqref{eq:ns} accounts for the momentum of whichever material is present at $\xb$, and only account for the difference in densities through the buoyancy force.

Figures~\ref{fig:fsi_def} and \ref{fig:fsi_disp} depict this benchmark's results.
The elementally coupled results use $\mfac = 2.0$ and the nodally coupled results use $\mfac = 1.0$.
We report two different Cartesian discretizations for this case since nodal coupling is generally more computationally efficient than elemental coupling.
In Figure~\ref{fig:fsi_disp}, notice that the nodal case with $N=32$, corresponding to $N=32$ cells in each dimension on the coarsest grid level, shows very good agreement with the experimental results.
Both the elemental and nodal cases with $N=24$ yield small deviations from the experimental results, with the elemental results being slightly closer to the experimental data.

\begin{table}
\centering
\begin{tabular}{| c | c | c | c | }
\hline
Fluid density & $\rhof$ & $1.1633$ & $\frac{\text{g}}{\text{cm}^3}$\\
\hline
Solid density & $\rhos$ & $1.058.3$ & $\frac{\text{g}}{\text{cm}^3}$\\
\hline
Viscosity & $\mu$ & $0.125$ & $\frac{\text{dyn} \cdot \text{s}}{\text{cm}^2}$ \\
\hline
Material model & - & modified neo-Hookean& - \\
\hline
Shear modulus & $G$ & $610,000$ & $\frac{\text{dyn}}{\text{cm}^2}$  \\
\hline
Numerical bulk modulus & $\kappas$ & $248{,}666.667$
& $\frac{\text{dyn}}{\text{cm}^2}$\\
\hline
Final time & $T_\text{f}$ & $30.0$  & s\\
\hline
Load time & $T_{\text{l}}$ & $0.5$ & s \\
\hline
\end{tabular}
\caption{
Parameters for Hessenthaler's FSI benchmark.
Note that the $\rhof$ is used for inertial terms (e.g., $\rhof \frac{D\ub}{Dt}$) for both the solid and fluid, and the density $\rhos$ is only used for the buoyancy force term $(\rhos - \rhof)g$ in the solid region.
}
\label{tb:fsi-param}
\end{table}

\begin{figure}
\centering
\begin{tabular}{c c c}
\includegraphics[width=.35\linewidth]{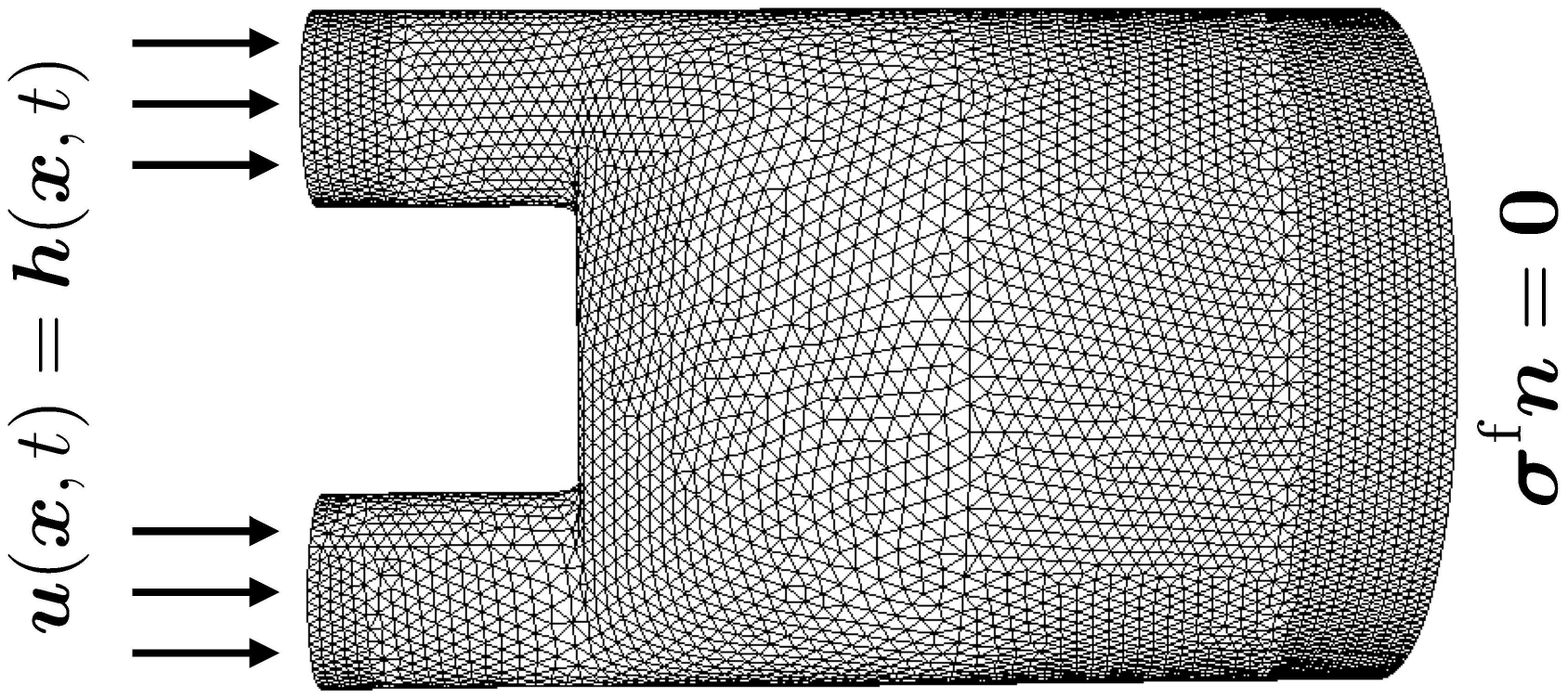} &
\includegraphics[width=.35\linewidth]{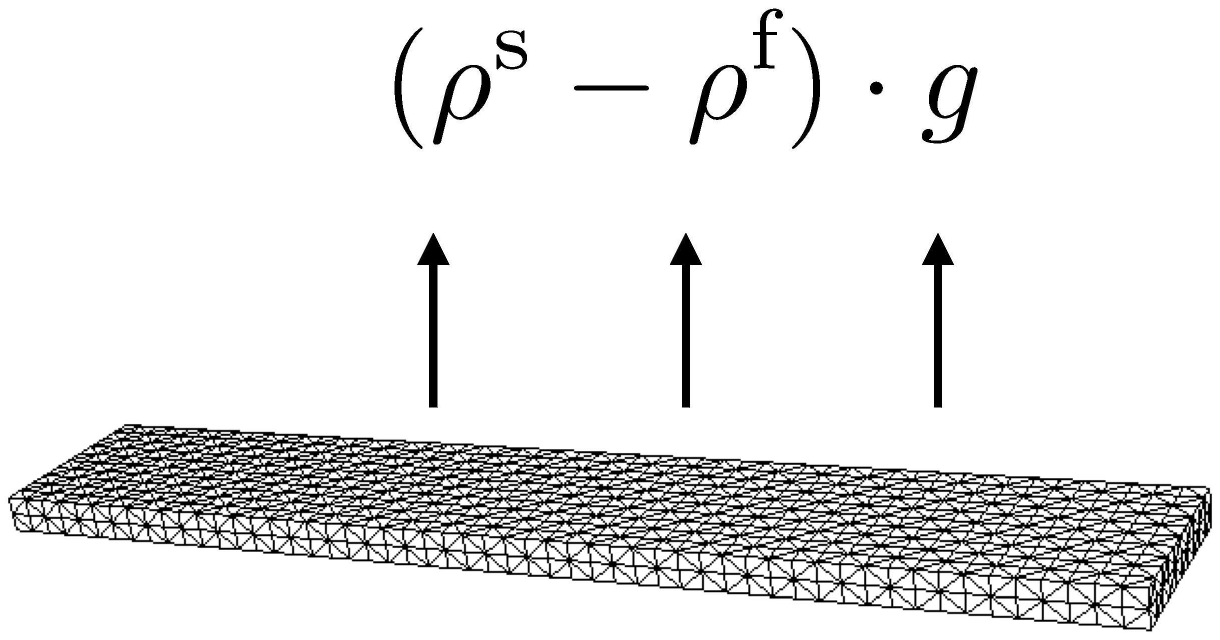} &
\includegraphics[width=.25\linewidth]{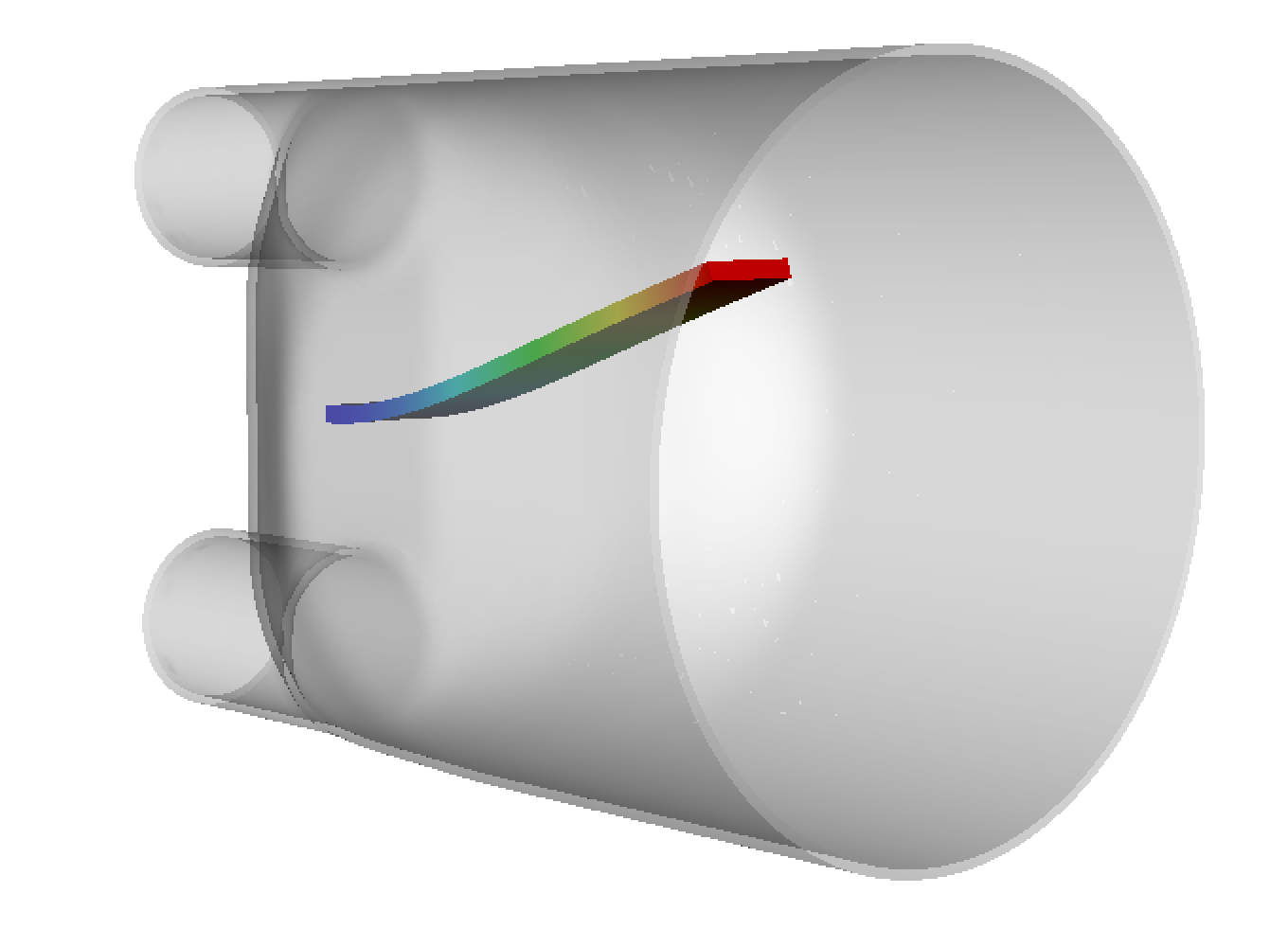} \\
(a) & (b) & (c)\\
\end{tabular}
\caption{Panels (a) and (b) provide the specifications of Hessenthaler's FSI benchmark.
  The parabolic inflow profiles are applied to the two circle regions on the left, and the larger region on the right is the outflow region.
  Panel (c) shows the deformation of the elastic beam.}
\label{fig:fsi}
\end{figure}

\begin{figure}
\centering
\begin{tabular}{l c c c}
&$t = 0\ \text{s}$ & $t = 0.75\ \text{s}$ & $t = 30\ \text{s}$  \\
\rotatebox{90}{$\quad\;\;$ Elemental}&
\includegraphics[width=.25\linewidth]{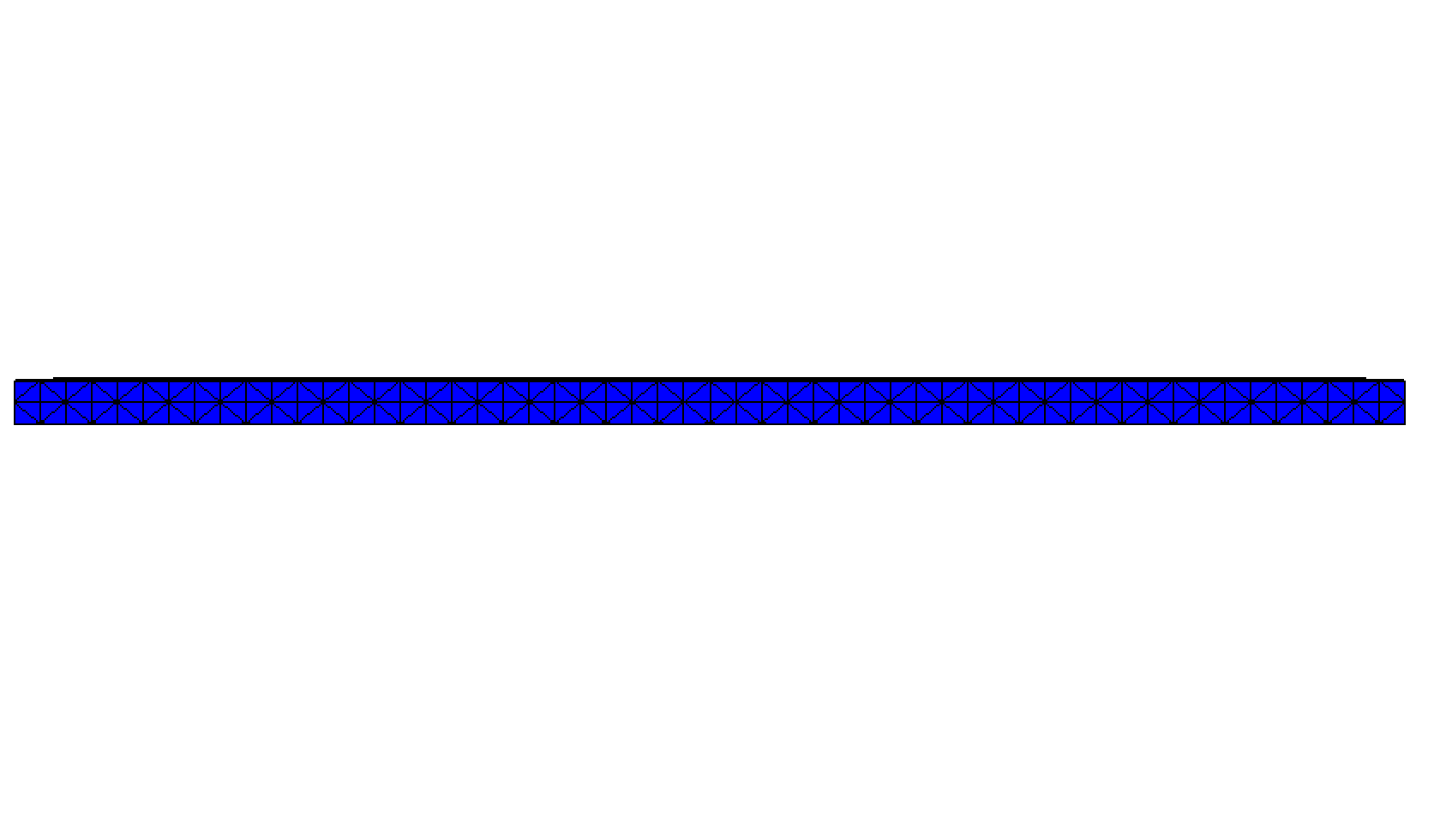}&
\includegraphics[width=.25\linewidth]{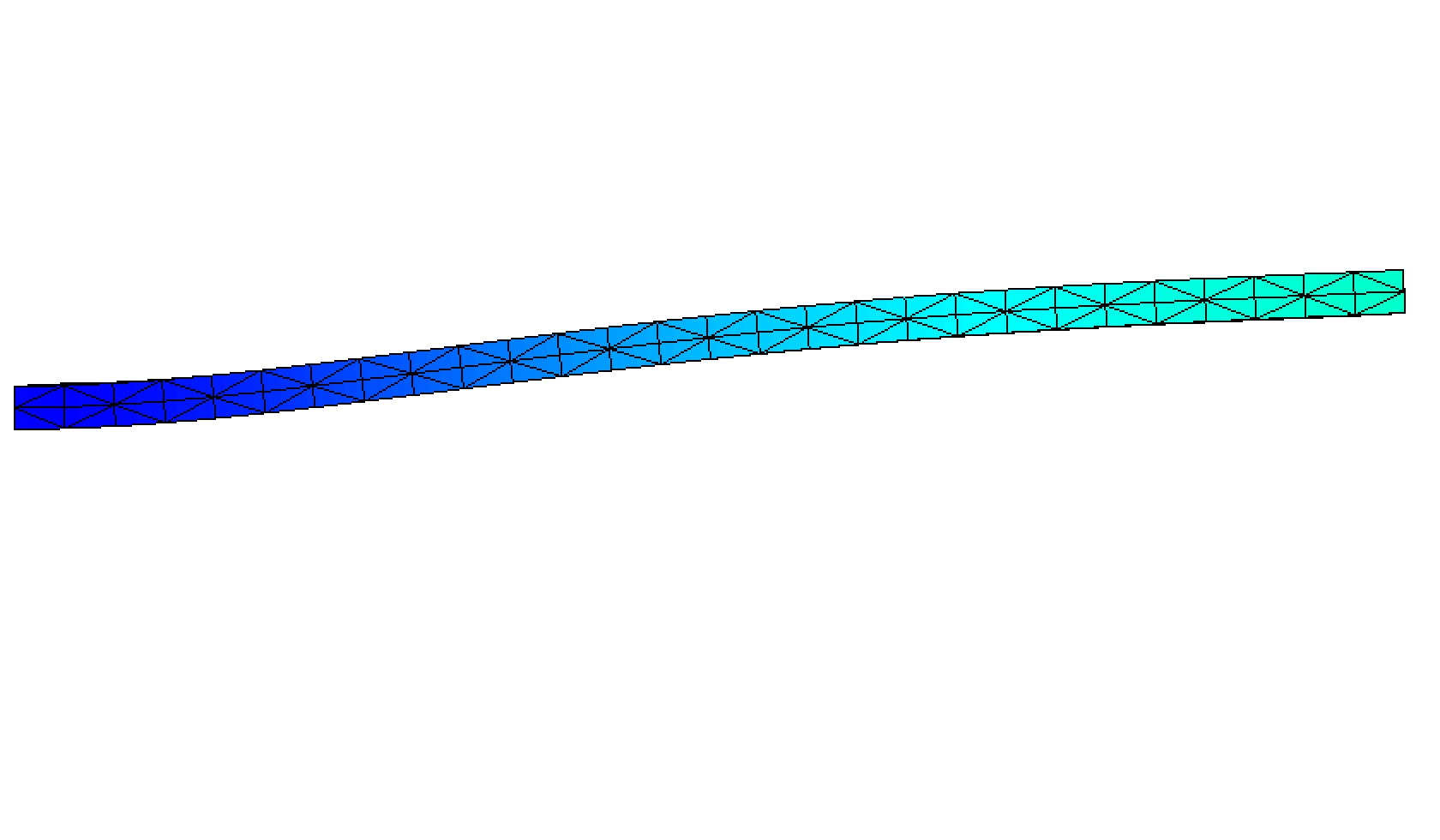}&
\includegraphics[width=.25\linewidth]{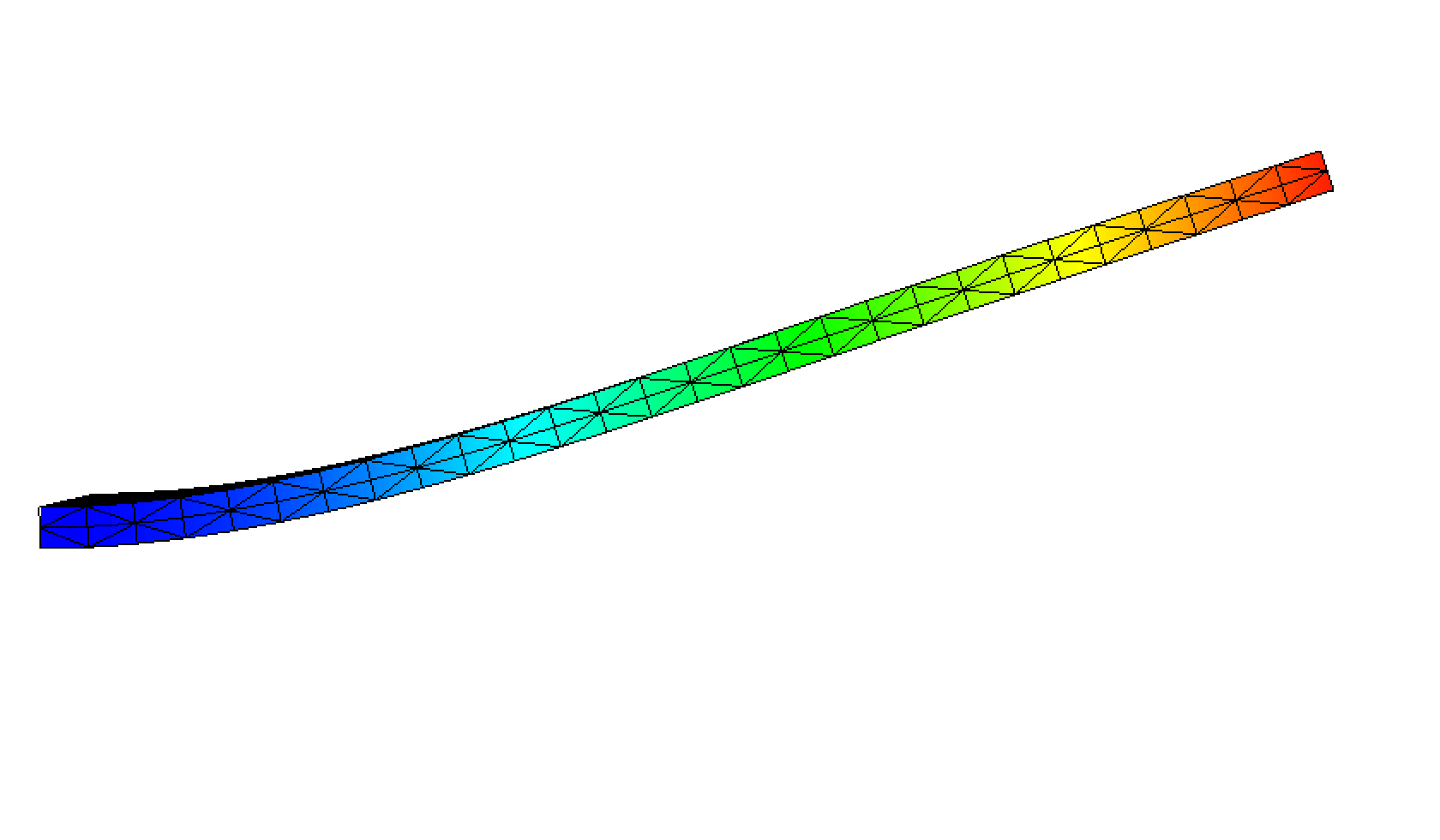}\\

\rotatebox{90}{$\quad\qquad$ Nodal}&
\includegraphics[width=.25\linewidth]{Figures/fsi/fsi_nodal_t=0.png}&
\includegraphics[width=.25\linewidth]{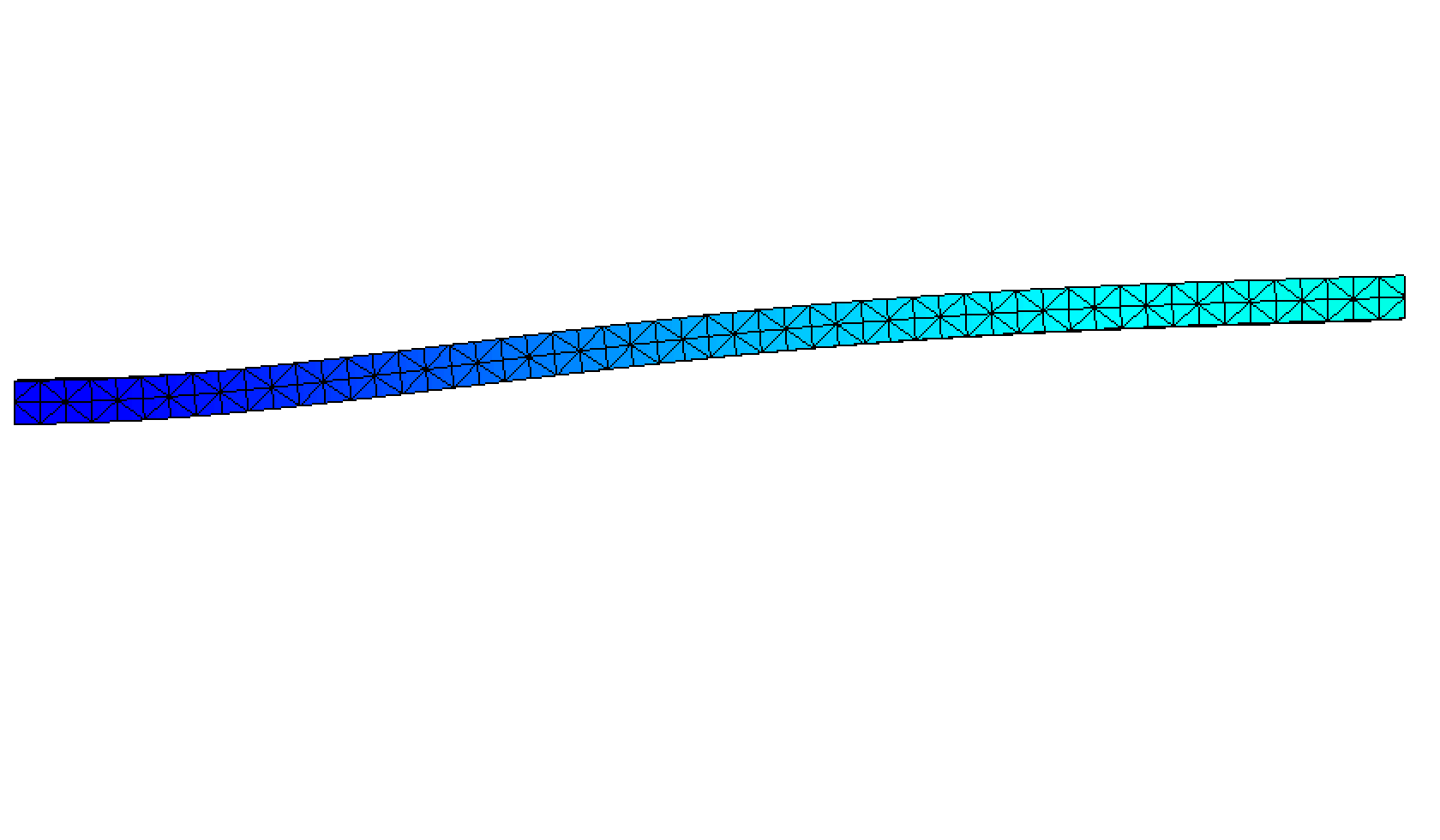}&
\includegraphics[width=.25\linewidth]{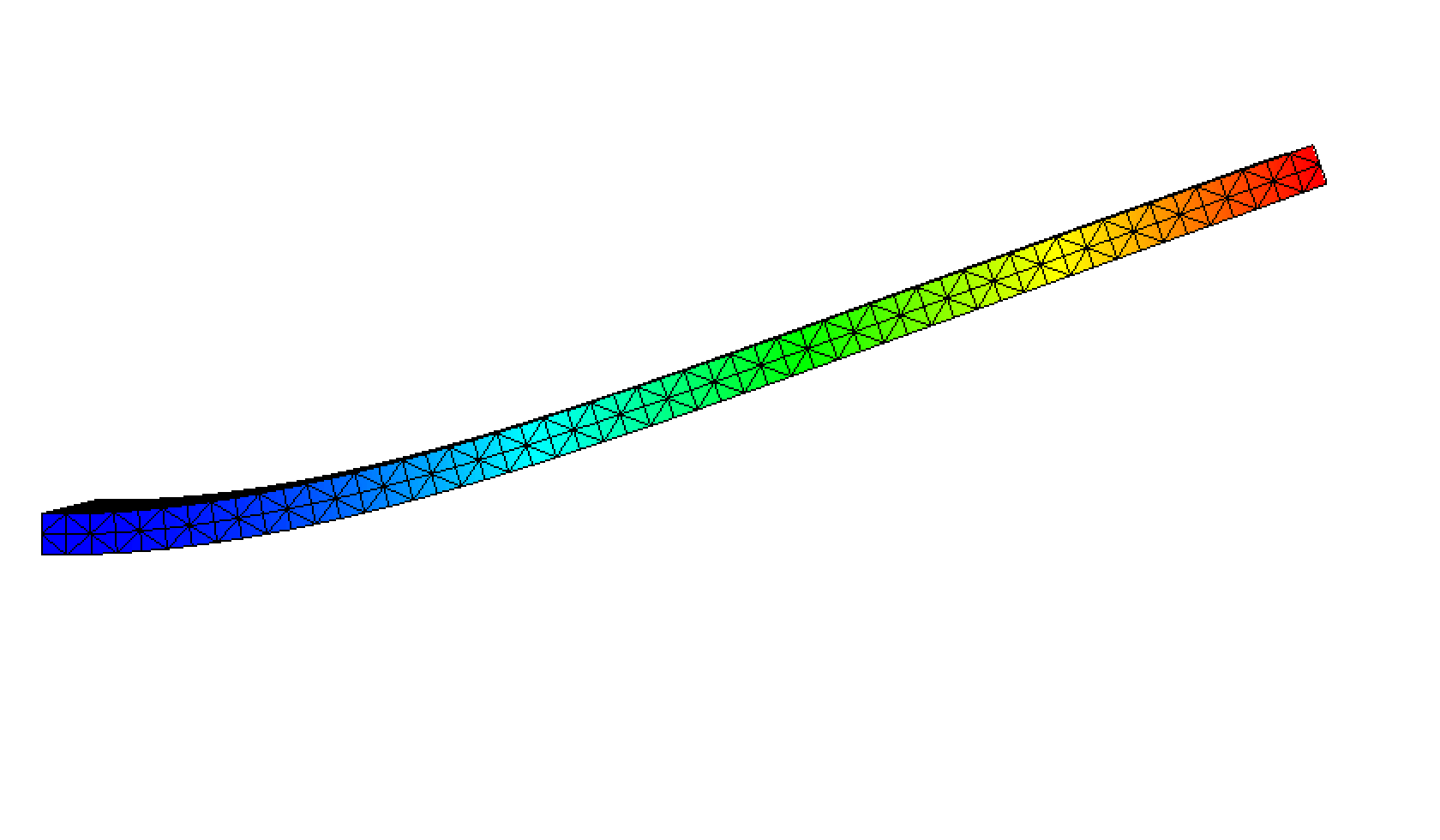}\\

&& $d^{\text{s}}_2$ &\\
&& \includegraphics[width=.3\linewidth]{Figures/color_bar.pdf}&\\
% more manual spacings
&&\hspace{0.025\linewidth}$0.0\ \text{cm}$ \hspace{.225\linewidth} $1.8\ \text{cm}$ &
\end{tabular}
\caption{Deformations and $y$-displacement of Hessenthaler's FSI benchmark with both elemental and nodal coupling at different points in time.
  Time values are start of the simulation, $1.5T_\text{l}$, and $T_\text{f}$.
  In both cases, the beams are discretized with $\Qtwo$ elements (however, the elements look like tetrahedra due to a visualization artifact).
  The nodal case uses $\mfac = 1.0$ and the elemental case uses $\mfac = 2.0$.
  Both use the same Cartesian grid with $N=24$ cells in each dimension on the coarsest AMR level.}
\label{fig:fsi_def}
\end{figure}

\begin{figure}
\centering
\begin{tabular}{l c}
\rotatebox{90}{$\qquad\qquad y-\text{displacement (mm)}$} &
\includegraphics[width=.45\linewidth, trim={50 200 50 100}, clip]{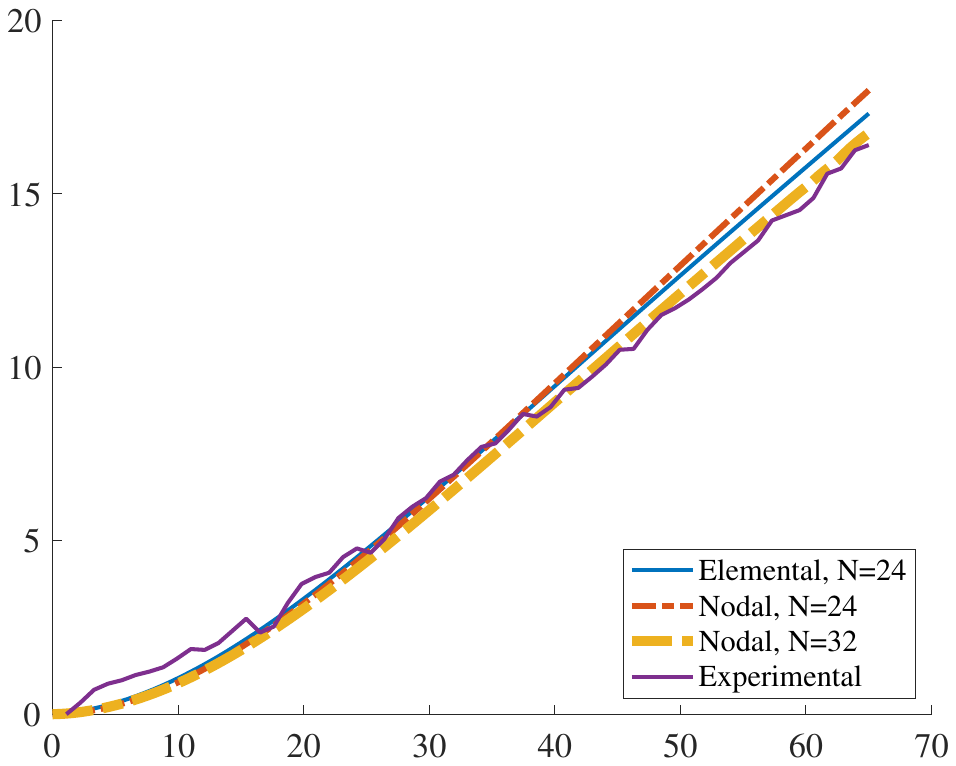}\\
&$z-\text{position (mm)}$
\end{tabular}

\caption{The displacement of the centerline of the elastic beam in Hessenthaler's FSI benchmark for $\Qtwo$ elements.
  $N=24$ corresponds to a fluid grid with three refinement levels and a coarsest level with $N=24$ cells in each dimension.
  $N=32$ is similarly defined.
  Elemental results use $\mfac = 2.0$ and nodal results use $\mfac = 1.0$.}
\label{fig:fsi_disp}
\end{figure}

\subsection{Dynamic Benchmarks}
This subsection examines several benchmarks that do not exhibit steady-state behavior.

\subsubsection{Modified Turek-Hron}
\label{Turek-Hron}
The Turek-Hron benchmark is an FSI benchmark introduced by Turek and Hron~\cite{Turek2007} and was first implemented with an IB method by Roy \etal~\cite{Roy2015}.
Contrary to both these instances, our formulation involves an incompressible fluid and structure.
%% \DRW{I verified that Roy et al do use a compressible formulation for Turek-Hron (its in the compressible section.)}
Additionally, we model the beam with a modified neo-Hookean model for the elastic component and with the same viscosity model that is used for the Newtonian fluid.
This stands in contrast to the original work, which used a Saint-Venant Kirchhoff material and no viscosity model, resulting in purely elastic and compressible deformations.

The Cartesian grid is $6N \times N$ with $N = 32$ or $N = 64$ cells with $3$ levels.
Hence, the Cartesian grid has an effective fine grid resolution of $4 N$.
The time step size is $0.02 \euleriandx / 2 \ \text{s}$ in which $0.02$ is a CFL-like stability restriction and $2 \frac{\text{m}}{\text{s}}$ is the largest anticipated velocity.
The modifications used here follow those introduced in the work of Lee, in which a modified incompressible Turek-Hron benchmark was used to study the grid-dependence of IB kernel functions~\cite{LeeThesis, Lee2021}.
This benchmark involves an elastic beam that is attached to a rigid disk, and unlike the previous benchmarks, motion is driven by fluid forces rather than an external traction on the solid boundary.
We use a domain of $[0,L] \times [0, H]$, with $L = 2.46$ m and $H = 0.41$ m, and place the rigid disk so that its center is at $(2,2)$.
Table~\ref{tb:th-param} details the remaining parameters, and the configuration is depicted in Figure~\ref{fig:th}.
The penalty parameters to impose rigidity and to fix the elastic beam to the rigid cylinder are $\kappa_{\text{B}} = 3031.25 \cdot \frac{\Delta x}{\left(\Delta t\right)^2} \frac{\text{dyn}}{\text{cm}^4}$ and $\kappa_{\text{S}} =  3031.25 \cdot \frac{\Delta x}{\left(\Delta t\right)^2} \frac{\text{dyn}}{\text{cm}^3}$.
We run the simulation until time $T_\text{f} = 12.0$ s to allow for the periodic oscillations of the structure to become fully developed.

We use a parabolic inflow boundary condition with a max inflow velocity of $u_{\text{max}} = 3.0 \, \frac{\text{m}}{\text{s}}$.
This inflow condition was applied via $\hb(\xb,t) = (h_1(\xb,t), 0)$, with
\begin{equation}
    h_1(\xb,t) = u_{\text{max}} \frac{y(H - y)}{(H/2)^2} \cdot
    \begin{cases}
      \half \left(1 - \cos\left(\frac{\pi}{2} t\right) \right),
      \ t < 2 \\
      1, \ t \geq 2.
    \end{cases}
\end{equation}
This specification of the problem corresponds most closely to the ``FSI3'' case in the original work of Turek and Hron~\cite{Turek2007}, which was designed to benchmark the periodic oscillations of the immersed structure when subjected to a relatively high inflow velocity.

The rigid disk is discretized with $\Ptwo$ surface elements in all cases.
Figure \ref{fig:th_def} shows the deformations of elastic beams for both elemental and nodal coupling at three points in time.
We investigate elemental and nodal coupling with $\mfac = 0.5, 1.0,$ and $2.0$.
Tables~\ref{tb:th-mean-amp1}--\ref{tb:th-mean-amp2} show the $x$ and $y$ displacements of the encircled point in Figure \ref{fig:th} and the associated Strouhal numbers.
All results are for the $N = 64$ cases and are between times $t = 10.0$ s and $t = 12.0$ s.
$\mfac = 1.0$ yields the best agreement between the elementally and nodally coupled cases for all element types.
Both elemental and nodal coupling results stay within the same ranges with the exception of $\mfac = 2.0$.
This benchmark highlights the extent to which the elementally coupled method may be less sensitive to $\mfac$ changes than the nodally coupled method; see, for instance, the range of values in fourth columns in Tables \ref{tb:th-mean-amp1}--\ref{tb:th-mean-amp2}.
However, the difference in $\mfac$ sensitivity is less stark when we omit $\mfac = 2.0$.
This is expected because $\mfac = 2.0$ corresponds to a relatively coarser structural discretization, and, as demonstrated in Section \ref{subsec:mfac}, the nodally coupled method may start to experience gaps in the Cartesian grid force density with $\mfac \geq 2.0$ in shear-dominated cases whereas the elementally coupled method does not.

\begin{table}
\centering
\begin{tabular}{| c | c | c | c | }
\hline
Density & $\rho$ & $1{,}000$ & $\frac{\text{kg}}{\text{m}^3}$\\
\hline
Viscosity & $\mu$ & $1.0$ & $\text{Pa} \cdot \text{s}$ \\
\hline
Material model & - & modified neo-Hookean & - \\
\hline
Shear modulus & $G$ & $2.0$ & $\text{MPa}$  \\
\hline
Numerical bulk modulus & $\kappas$ & $18.78 - 75.12$
& $\text{MPa}$\\
\hline
Final time & $T_\text{f}$ & $12.0$  & s\\
\hline
Load time & $T_{\text{l}}$ & $2.0$ & s \\
\hline
\end{tabular}
\caption{Parameters for the Turek-Hron benchmark.}
\label{tb:th-param}
\end{table}

\begin{figure}
\centering
\includegraphics[width=.7\linewidth]{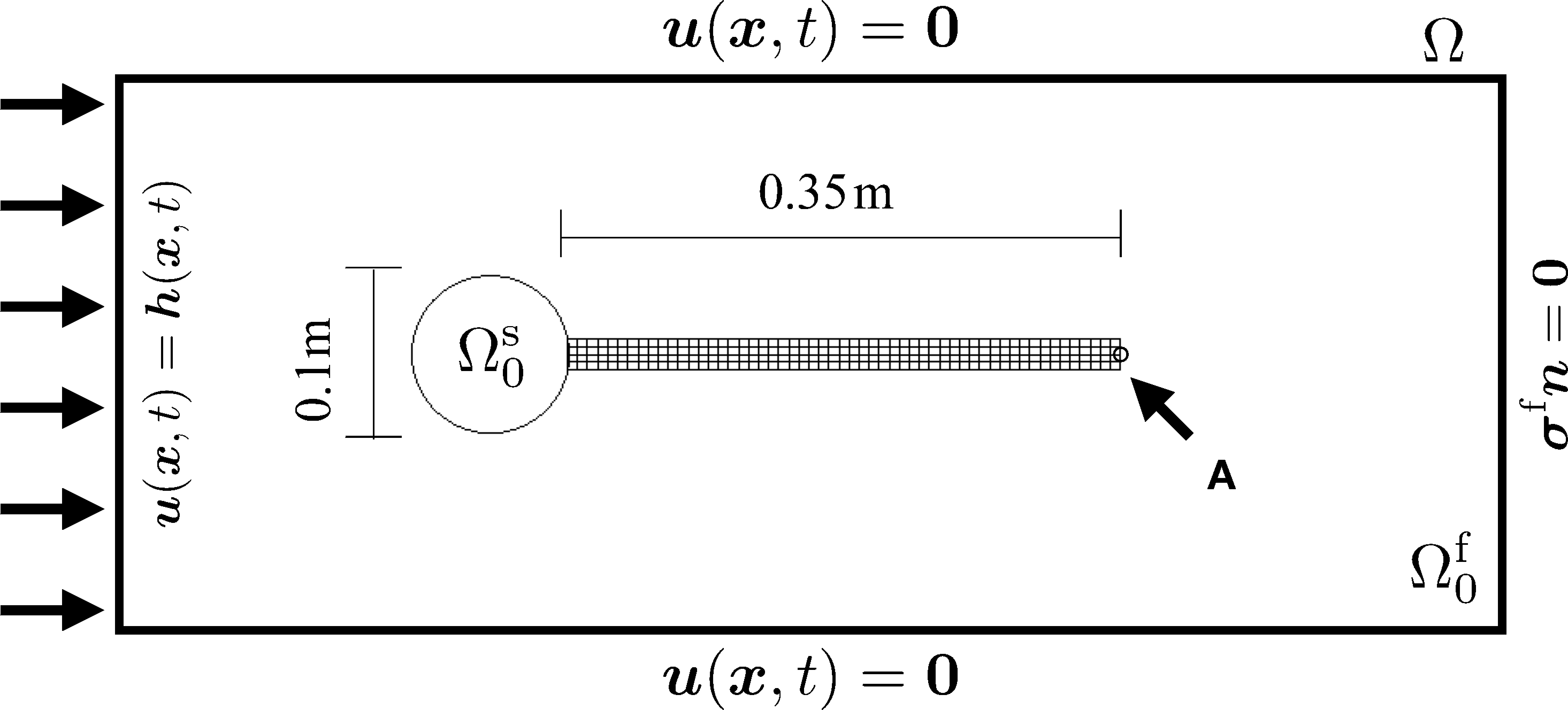}
\caption{Specifications of the Turek-Hron benchmark (Section~\ref{Turek-Hron}). Note that the elastic beam (length $0.35 \ \text{m}$) is not depicted to scale with respect to the computational domain (length $2.46 \ \text{m}$).}
\label{fig:th}
\end{figure}

\begin{figure}
\centering
\begin{tabular}{l c c c}
&$t = 0$ & $t = 6$ & $t = 12$  \\
\rotatebox{90}{$\quad\;\;$ Elemental}&
\includegraphics[width=.25\linewidth]{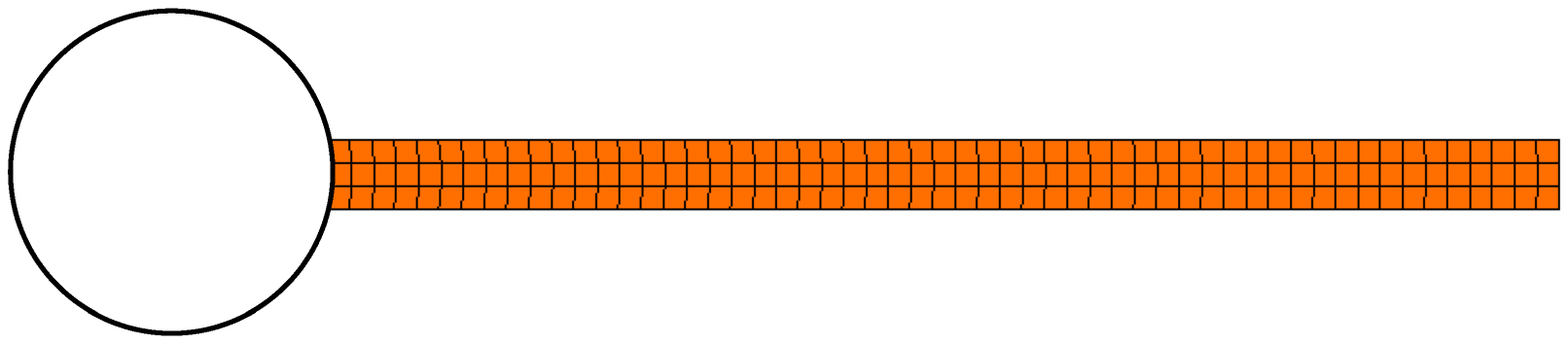}&
\includegraphics[width=.25\linewidth]{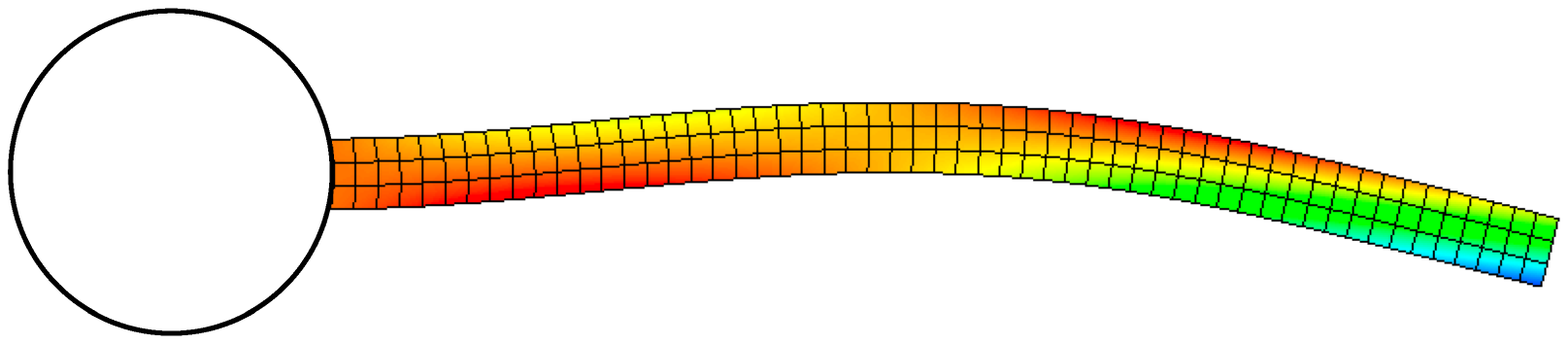}&
\includegraphics[width=.25\linewidth]{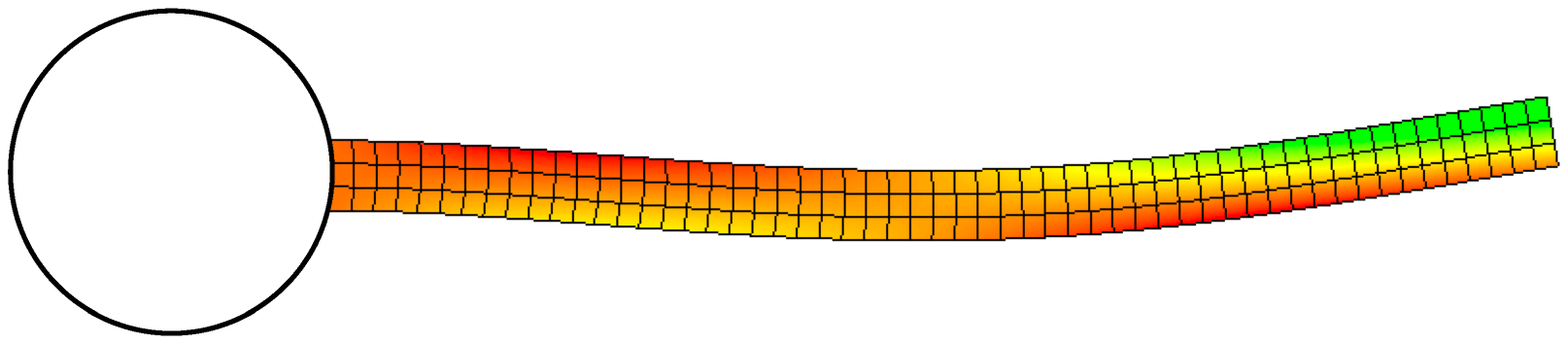}\\

\rotatebox{90}{$\quad\qquad$ Nodal}&
\includegraphics[width=.25\linewidth]{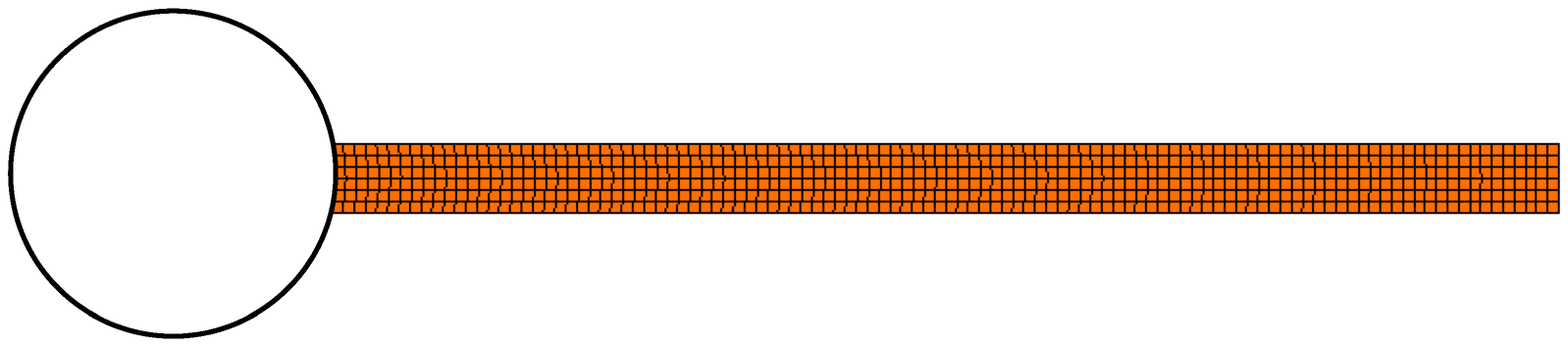}&
\includegraphics[width=.25\linewidth]{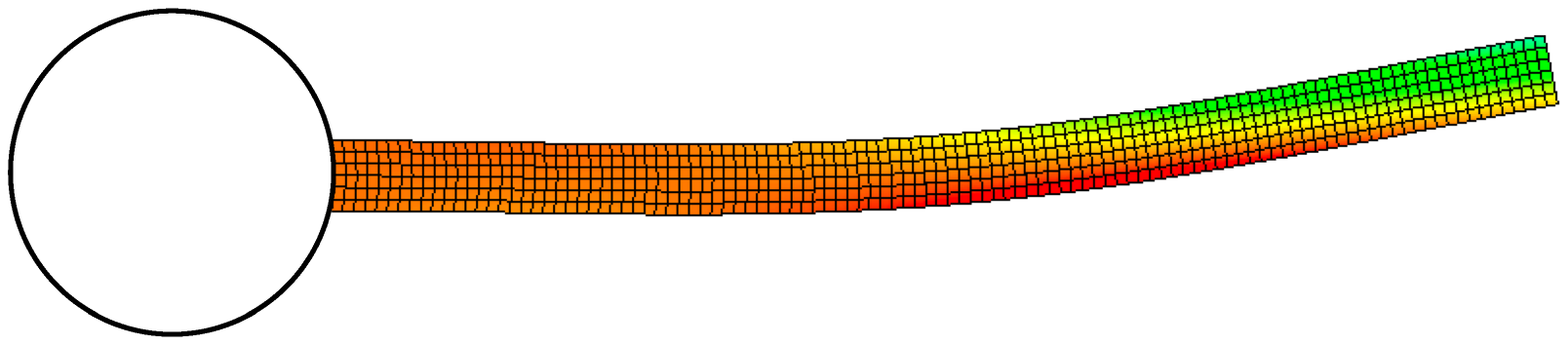}&
\includegraphics[width=.25\linewidth]{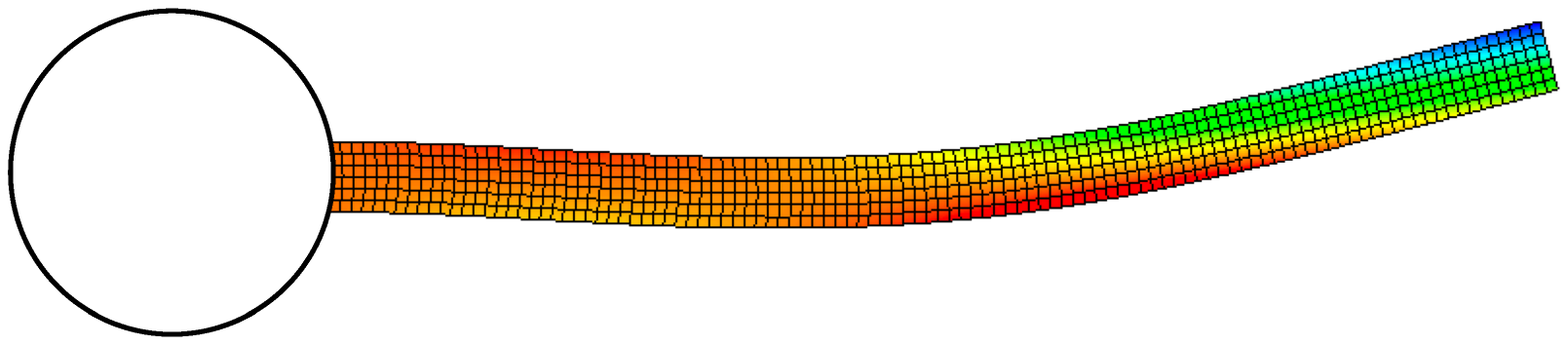}\\

&& $d^{\text{s}}_2$ &\\
&& \includegraphics[width=.3\linewidth]{Figures/color_bar.pdf}&\\
% more manual spacings
&&\hspace{0.01\linewidth}$-0.007\ \text{m}$ \hspace{.2\linewidth} $0.001\ \text{m}$ &
\end{tabular}
\caption{Deformations of the Turek-Hron benchmark's beam for both elemental and nodal coupling at different points in time.
  Time values are start of the simulation, $0.5T_\text{f}$, and $T_\text{f}$.
  In both cases the structures are discretized with $\Qone$ elements.
  The nodal case uses $\mfac = 1.0$ and the elemental case uses $\mfac = 2.0$.
  Both use the same Cartesian grid denoted by $N=32$.}
\label{fig:th_def}
\end{figure}

\begin{table}
\centering
\begin{tabular}{ c | c | c | c | c | c | c }
\hline
& $\mfac$ & Element Type & $d^{\text{s}}_1 \ (\text{m})$ & $d^{\text{s}}_2 \ (\text{m})$ & $\text{St}_1$
& $\text{St}_2$ \\
\hline
Elemental Coupling &0.5 & $\Pone$ & $-2.6479 \pm 2.5397$ & $1.6122 \pm 33.3406$
& 11.1523 & 5.5739\\
&& $\Qone$ & $-2.6297 \pm 2.5221 $& $1.5580 \pm 33.2858$ & 11.1487 & 5.5731\\
&& & & & & \\

 &1.0 & $\Pone$ &$-2.7138 \pm 2.5694$ & $1.6693 \pm 33.2625$ & 11.1433 & 5.5694\\
&& $\Qone$ &$-2.6922 \pm 2.5694$ & $1.6906  \pm 33.1611$ & 11.1104 & 5.5570\\
  && & & & & \\
  
 &2.0 & $\Pone$ &$-2.7512 \pm 2.5610$ & $1.4424 \pm 32.6893$ & 11.2140 & 5.6050\\
&& $\Qone$ &$-2.7417 \pm 2.5955$ & $1.4551 \pm 33.0260$ & 11.1983 & 5.5972\\
  && & & & & \\
  
 \hline
Nodal Coupling &0.5 & $\Pone$ & $-2.4947 \pm 2.3843$ & $1.5455 \pm 32.4282$ & 11.1428
& 5.5750\\
&& $\Qone$ &$-2.4754 \pm 2.3897$ & $1.5027 \pm 32.5401$ & 11.1335 & 5.5708\\
&& & & & & \\

&1.0 & $\Pone$ &$-2.6020 \pm 2.5115$ & $1.6490 \pm 32.9147$ & 11.1175 & 5.5566\\
&& $\Qone$ &$-2.6386 \pm 2.5036$ & $1.6589 \pm 32.8664$ & 11.0973 & 5.5500\\
  && & & & & \\
  
&2.0 & $\Pone$ &$-3.0115 \pm 2.9064$ & $1.6368 \pm 35.3819$ & 11.3609 & 5.6776\\
&& $\Qone$ &$-2.9812 \pm 2.8743$ & $1.6399 \pm 35.3095$ & 11.3575 & 5.6758\\
  && & & & & \\
\hline
\end{tabular}
\caption{
Mean, amplitude, and Strouhal number for the Turek-Hron benchmark for the $x$ and $y$-components of displacement for $N = 64$ and first-order elements.
}
\label{tb:th-mean-amp1}
\end{table}

\begin{table}
\centering
\begin{tabular}{ c | c | c | c | c | c | c }
\hline
& $\mfac$ & Element Type & $d^{\text{s}}_1 \ (\text{m})$ & $d^{\text{s}}_2 \ (\text{m})$ & $\text{St}_1$
& $\text{St}_2$ \\
\hline
Elemental Coupling &0.5 & $\Ptwo$ &$-2.6590 \pm 2.5284$ & $1.5582 \pm 33.3757$ & 11.1457
& 5.5705\\
&& $\Qtwo$ &$-2.6115 \pm 2.5199$ & $1.5382 \pm 33.2676$ & 11.1423 & 5.5692\\
&& & & & & \\

&1.0& $\Ptwo$ &$-2.8010 \pm 2.6465$ & $1.7158 \pm 33.9103$ & 11.0641 & 5.5338\\
&& $\Qtwo$ &$-2.7439 \pm 2.5986$ & $1.6472 \pm 33.5674$ & 11.0820 & 5.5390\\
  && & & & & \\
  
&2.0& $\Ptwo$ &$-2.7693 \pm 2.6691$ & $1.4170 \pm 34.2502$ & 11.0678 & 5.5346\\
&& $\Qtwo$ &$-2.7550 \pm 2.5825$ & $1.4621 \pm 33.7282$ & 11.0806 & 5.5415\\
  && & & & & \\
  
 \hline
 Nodal Coupling &0.5 & $\Ptwo$ &$-2.5218 \pm 2.4237$ & $1.4997 \pm 32.7361$ & 11.1141
 & 5.5594\\
&& $\Qtwo$ &$-2.4639 \pm 2.3917$ & $1.4751 \pm 32.6360$ & 11.1185 & 5.5633\\
&& & & & & \\

&1.0 & $\Ptwo$ & $-2.6676 \pm  2.5625$ & $1.7285 \pm 33.3968$ & 11.0485 & 5.5261\\
 && $\Qtwo$ &$-2.6027 \pm 2.5220$ & $1.6815 \pm 33.1388$ & 11.0695 & 5.5324\\
  && & & & & \\
  
&2.0 & $\Ptwo$ & $-3.1716 \pm 3.0097$ & $1.6551 \pm 36.7702$ & 11.1673 & 5.5814\\
 && $\Qtwo$ &$-3.0919 \pm 2.9460$ & $1.6406 \pm 36.3013$ & 11.2296 & 5.6124\\
  && & & & & \\
\hline
\end{tabular}
\caption{
Mean, amplitude, and Strouhal number for the Turek-Hron benchmark for the $x$ and $y$-components of displacement for $N = 64$ and second-order elements.
}
\label{tb:th-mean-amp2}
\end{table}

\subsubsection{FSI Model of Bioprosthetic Heart Valve Dynamics}
\label{BHV}
We can use our findings in this study to simulate bioprosthetic heart valve (BHV) dynamics in a pulse-duplicator as described by Lee \etal.~\cite{Lee2020, LeeJTCVS}.
This model consists of a three-dimensional aortic test section of an experimental pulse-duplicator; see Figure~\ref{fig:BHV_comparison}(a).
The test section is coupled to three-element Windkessel models that provide realistic downstream loading conditions and the upstream driving conditions for the aortic test section.
A combination of normal traction and zero tangential velocity boundary conditions are used at the inlet and outlet to couple the reduced-order models to the detailed description of the flow within the test section.

The computational domain is $5.05$ cm $\times$ $10.1$ cm $\times$ $5.05$ cm.
The simulations use a three-level locally refined grid with a refinement ratio of two between levels and an $N/2 \times N \times N/2$ coarse grid with $N = 64$, which corresponds to $N = 256$ at the finest level.
The time step size is $\Delta t = 5.0 \times 10^{-6} \ \text{s}$ and is systematically reduced as needed to maintain stability throughout the simulation.

Solid wall boundary conditions are imposed on the remaining boundaries of the computational domain.
This benchmark's parameters are listed in Table~\ref{tb:bhv}
Structural models use $\Ptwo$ tetrahedral elements for the BHV leaflets and $\Pone$ tetrahedral elements for the aortic test section, which are generated using Trelis (Computational Simulation Software, LLC, American Fork, UT, USA).
We use a piecewise-linear kernel for the aortic test section and a three-point B-spline kernel for the valve leaflets as regularized delta functions.

\begin{figure}
\centering
\includegraphics[width=.7\linewidth]{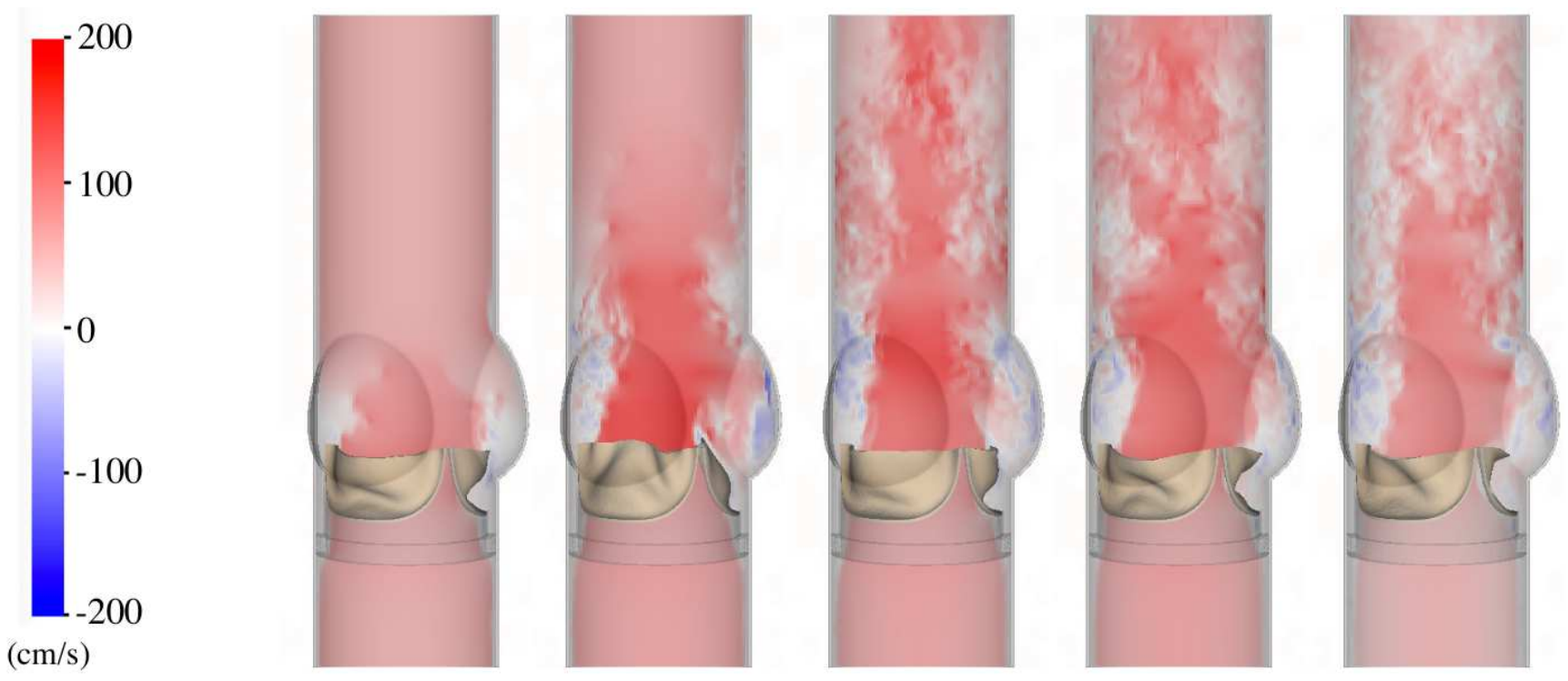}\\
(a)\\
\begin{tabular}{c c c}
\includegraphics[width=.3\linewidth]{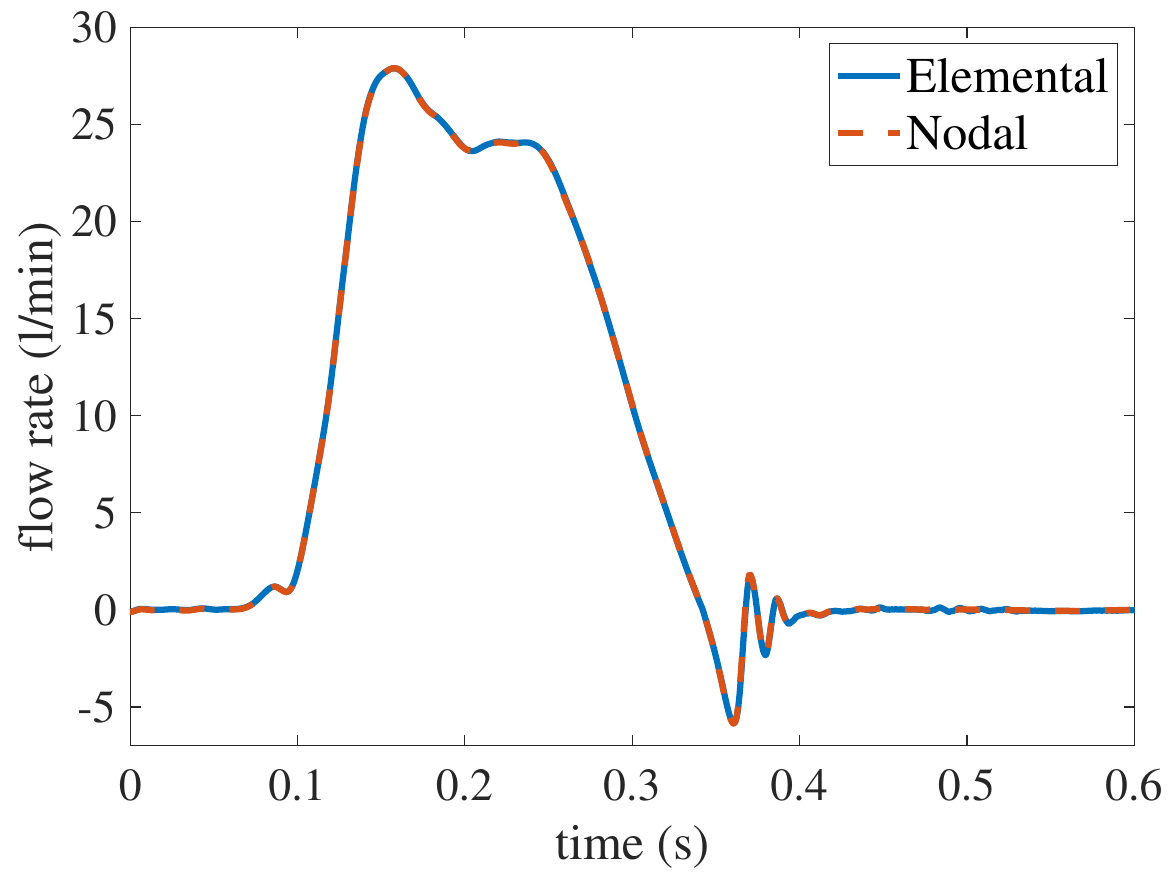} &
\includegraphics[width=.3\linewidth]{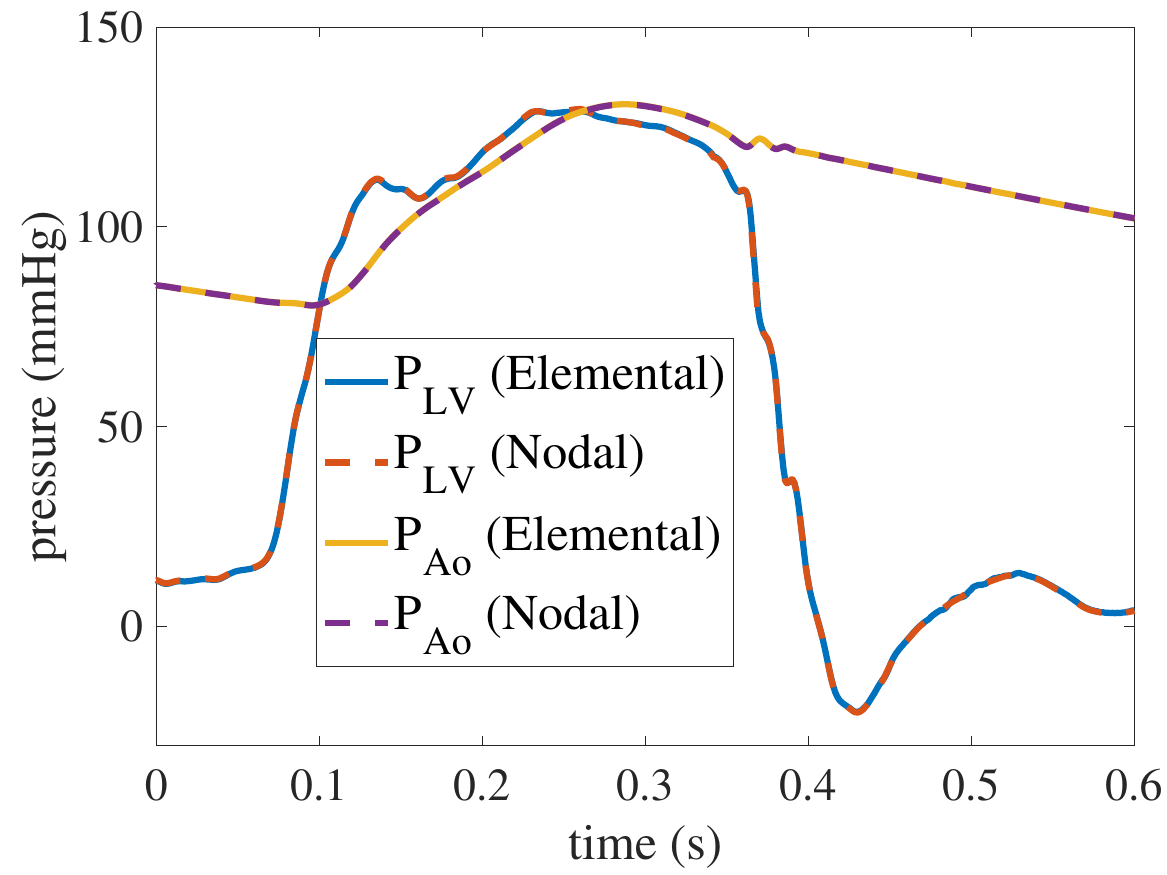} &
\includegraphics[width=.3\linewidth]{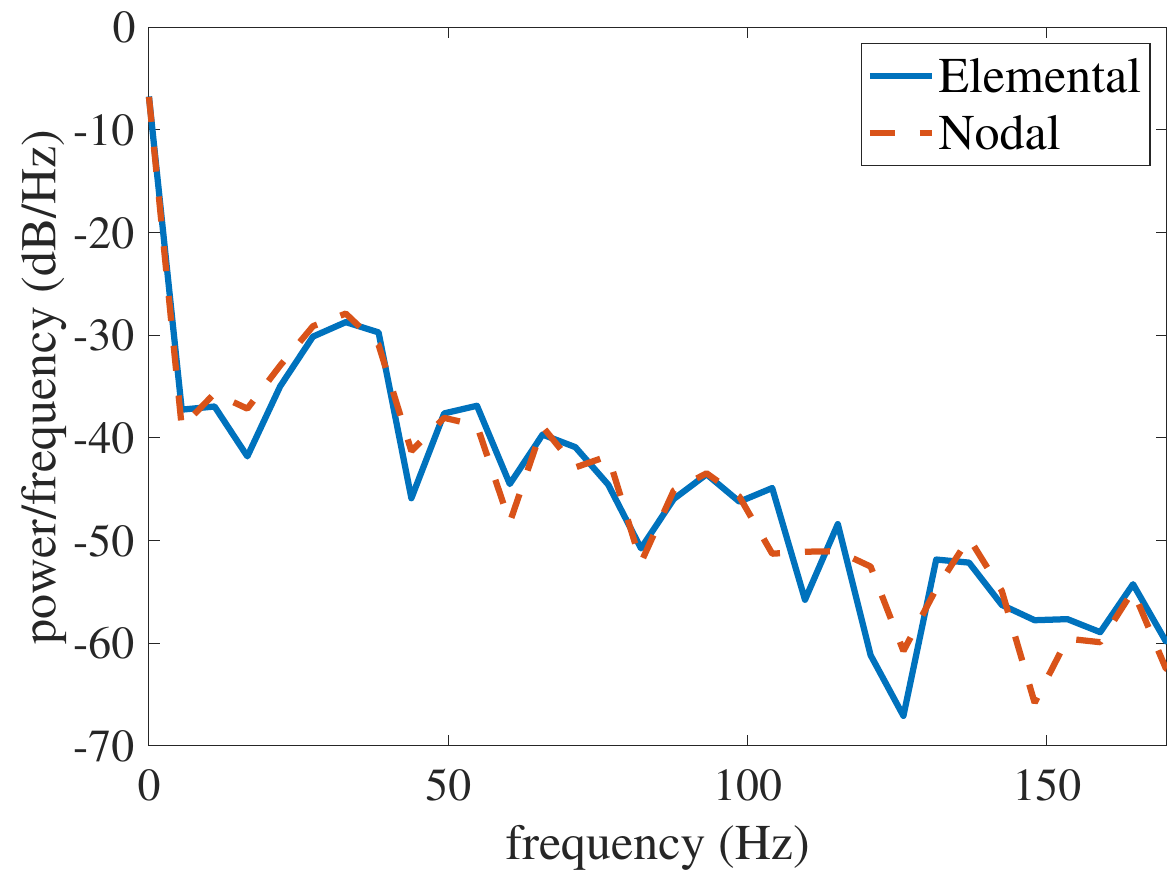} \\
(b) & (c) & (d) \\
\end{tabular}
\caption{
Model of a bovine pericardial bioprosthetic heart valve in an experimental pulse-duplicator~\cite{Lee2020, LeeJTCVS}.
(a) Axial flow in the aortic test section using nodal coupling.
Comparisons between results obtained for nodal and elemental coupling: (b) flow rates; (c) pressure waveforms; and (d) bioprosthetic heart valve (BHV) leaflet fluttering frequency.
The flow rate and pressure waveforms are in excellent agreement.
The BHV leaflet fluttering frequency quantifies the dominant frequency mode in spectral analysis performed on the leaflet tip displacement~\cite{LeeJTCVS}.
The BHV fluttering frequencies are 32.88~Hz for both methods, which suggests that the leaflet kinematics are in excellent agreement.
}
\label{fig:BHV_comparison}
\end{figure}

Figures~\ref{fig:BHV_comparison}b and~\ref{fig:BHV_comparison}c show that the BHV model using nodal coupling produces excellent agreements in bulk flow rates and pressure waveforms with elemental coupling.
To verify that we can also reproduce the consistent leaflet kinematics, we assess fluttering frequencies from leaflet tip position time series data.
We use the MATLAB Signal Processing Toolbox (MathWorks, Inc, Natick, MA, USA) to determine the power spectral density, and the second highest peak is used to determine the dominant frequency that characterizes leaflet fluttering~\cite{LeeJTCVS}.

\begin{figure}
\centering
\includegraphics[width=.7\linewidth]{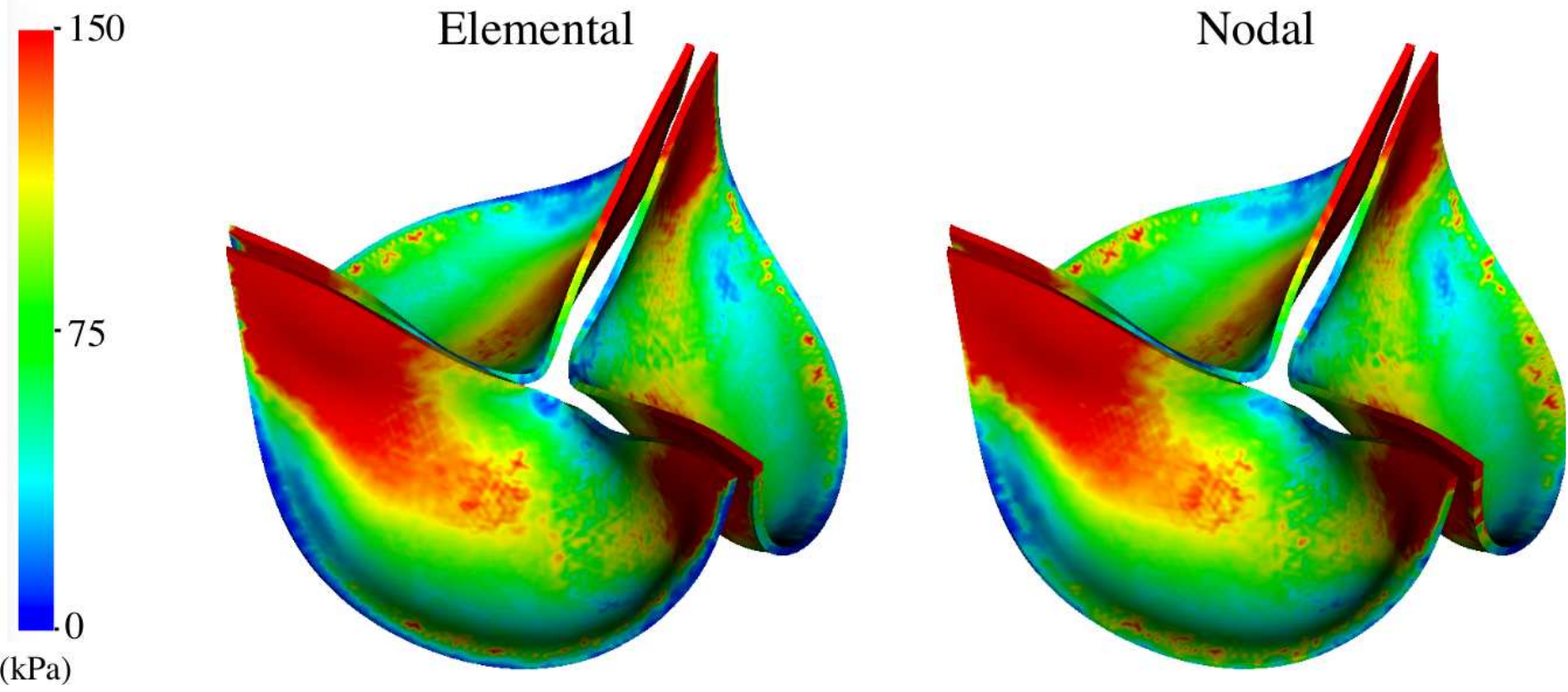}
\caption{Comparison of von Mises stress in $kPa$ of the closed and pressured BHV model obtained for elemental and nodal coupling.
  The stress distributions are remarkably consistent.}
\label{fig:von_mises}
\end{figure}

\begin{table}
\centering
\begin{tabular}{| c | c | c | c |}
\hline
Density & $\rho$ & $1.0$ & g/cm$^3$\\
\hline
Viscosity & $\mu$ & $1.0$ & cP\\
\hline
Material model & - & Modified Holzapfel-Gasser-Ogden & - \\
\hline
\end{tabular}
\caption{Parameters for the BHV benchmark (Section~\ref{BHV}).}
\label{tb:bhv}
\end{table}

\begin{table}
\centering
% aortic test section has 560161 elements
% aortic test section has 166799 nodes
%
% valve has 37615 elements
% valve has 72564 nodes
%
% 597776 total elements
% 239363 total nodes
% these nodes don't sum to the nodal workload - we must have deleted some from the meshes
\begin{tabular}{| l | r |}
\hline
total elements & 597776 \\
\hline
interaction points, nodal coupling & 238917 \\
\hline
interaction points, elemental coupling & 3329376 \\
\hline
\end{tabular}
\caption{
  Number of elements and interaction points for the BHV benchmark (Section~\ref{BHV}) at the first timestep.
  For elemental coupling, the average number of quadrature points per element here is five, which corresponds to most elements using the standard four-point tetrahedral Gauss quadrature rule and the remainder using rules with more points.
  Nodal coupling uses approximately an order of magnitude fewer interaction points because each node is shared by multiple elements.
}
\label{tb:bhv-ib-points}
\end{table}

We also observe the effect of using nodal coupling on a BHV leaflet stress distribution in diastole of a cardiac cycle, which is quantified by the von Mises stress~\cite{LeeJTCVS}.
Figure~\ref{fig:von_mises} compares von Mises stress and shows that using nodal coupling recapitulates the stress distribution obtained by elemental coupling.

\section{Discussion and Conclusion}
This work analyzes the differences between nodally and elementally coupled IFED methods and their relation to mass lumping.
We present a general framework from which both the new nodal coupling scheme detailed in this study and the elementally coupled method of Griffith and Luo~\cite{Griffith2017} may be derived.
These two coupling schemes are summarized in Algorithms~\ref{alg:nodal-algorithm} and \ref{alg:elemental-algorithm}.
This framework complements the available set of nodally coupled IB methods, such as the original IB method of Peskin~\cite{Peskin1972, Peskin2002} and other FE-based nodal IB methods~\cite{Zhang2004, Wang2004}.
Specifically, the type of method may be arrived at by the choice of quadrature rule for the mass matrix in the force projection~\eqref{eq:consistent-div-pk1-projection}--\eqref{eq:inconsistent-div-pk1-projection} and the choice of interaction points for force spreading~\eqref{eq:semidiscrete-spread-1}--\eqref{eq:semidiscrete-spread-3} and velocity projection~\eqref{eq:semidiscrete-interp-1}--\eqref{eq:semidiscrete-interp-3}.
We can derive the elemental coupling method~\cite{Griffith2017} by using a high-order Gauss quadrature rule for the mass matrix and adaptively-chosen Gauss quadrature points for the interaction points in the fluid-solid coupling operators.
Remarkably, the nodally coupled scheme described herein is obtained for \emph{any} nodal quadrature rule with arbitrary non-zero weights that is used for both integrating the mass matrix and the fluid-solid coupling operators (Theorem~\ref{thm:fully-nodal-ignore-weights}).
It is precisely because of this arbitrariness of weights that we may use elements that must typically be avoided in finite element computations with mass lumping, such as standard $\Ptwo$ elements.
In our nodal coupling scheme, the entries of the lumped mass \emph{exactly} cancel out with coupling weights, and thus the coupling scheme depends only on the positions of the nodes and not on the nodal quadrature weights.
This cancellation enables the use of even higher-order elements with mass lumping.
Hence, we can use \emph{any diagonal matrix} as the lumped mass matrix with this method, \emph{regardless of whether or not the original finite element space was suitable for mass lumping}.

Following the work of Vadala-Roth~\etal~\cite{Vadala-Roth2020}, we adapt several standard solid mechanics benchmarks to the IB framework to test our method's capabilities for incompressible finite deformations and to explore different element types and choice of mesh factor ratio $\mfac$.
Section~\ref{subsec:mfac} provides some guidance on the nodally coupled method's mesh spacing requirements.
We also examine an implementation of the Turek-Hron benchmark as well as some three-dimensional problems. The large-scale dynamic FSI model of a bioprosthetic heart valve in a pulse-duplicator system developed by Lee et al~\cite{LeeJTCVS} was shown to yield remarkably consistent results for both nodal and elemental coupling strategies.
Notably, this model was previously demonstrated to yield leaflet dynamics that are in excellent agreement with experimental data from the pulse-duplicator system~\cite{Lee2020,LeeJTCVS}.

In all benchmarks considered herein, elemental and nodal coupling are in excellent agreement when $\mfac \leq 1$, but nodal coupling can sometimes exhibit poor accuracy with large values of $\mfac$.
This is not surprising since, with relatively large element sizes, the interaction points will be far enough apart that there will be gaps in the Cartesian grid representation of the force density and, therefore, the potential for catastrophic leaks through the structure.
Subsection~\ref{subsec:mfac} examines two cases in which nodal coupling performs significantly worse than elemental coupling with very large values of $\mfac$.
As was shown in the work of Griffith and Luo~\cite{Griffith2017}, adaptive quadrature may be used to circumvent the oft-repeated rule of thumb for IB methods that the structural discretization must be twice as fine as the Cartesian grid.
Our results suggest that even with nodal coupling, we are able to use structural meshes that are about twice as \emph{coarse} as the Cartesian grid for a shear-dominant case, which suggests that the commonly used rule on the spacing of interaction points may need further investigation.
Lee and Griffith~\cite{Lee2021} have shown that for pressure-loaded cases, the structural node spacing must be at least as fine as the Cartesian grid and our results suggest that the same holds for nodal coupling.
However, this is still worth noting because our result suggests that we may be able to use nodal spacings as large as $\mfac = 1$, which is again an improvement from the rule of thumb often cited for IB methods that structural nodes should be approximately half a meshwidth apart from each other.
Further, numerical benchmarks suggest that the choice of the structural discretization is problem-dependent, which is consistent with prior studies of this methodology.

The nodally coupled IFED method avoids solving linear systems involving nontrivial mass matrices and using elemental quadrature rules to define dense sets of interaction points.
These two changes make the nodal coupling algorithm substantially more computationally efficient than the elementally coupled one when using the same number of degrees of freedom in each case.
Overall, for situations that require $\mfac \leq 1$ for accuracy reasons, nodal IFED methods appear to provide comparable accuracy to elementally coupled IFED methods.
Further, the ability of the nodally coupled IFED method to use higher-order elements with simple mass lumping strategies.
These features accelerate the IFED method and enable its use in complex modeling applications.

\section*{Acknowledgments}
B.E.G. acknowledges funding from the NIH (Awards R01HL157631 and U01HL143336) and NSF (Awards OAC 1450327, OAC 1652541, OAC 1931516, and CBET 1757193).
Simulations were performed using computational facilities provided by the University of North Carolina at Chapel Hill through the Research Computing Division of UNC Information Technology Services.

\bibliography{mass_lumping_paper.bib}
\end{document}